\documentclass[11pt,]{article}
\pdfoutput=1
\usepackage{todonotes}

\usepackage[linktocpage,colorlinks,linkcolor=blue,anchorcolor=blue,citecolor=blue,urlcolor=blue,pagebackref]{hyperref}
\usepackage{tikz}
\RequirePackage[authoryear]{natbib}
\setcitestyle{authoryear,open={(},close={)}} 

\usepackage{mathtools,amsthm,amssymb,bm,mathrsfs}
\allowdisplaybreaks

\usepackage{algorithm}
\usepackage[noend]{algpseudocode}
\usepackage{booktabs,tabularx,array,longtable}
\usepackage{enumitem}
\usepackage{xcolor}

\makeatletter
\providecommand*{\theHALG@line}{}
\renewcommand*{\theHALG@line}{\thealgorithm.\arabic{ALG@line}}
\makeatother
\usepackage{aliascnt}
\usepackage[capitalise,nameinlink,noabbrev]{cleveref}

\usepackage[margin=1in]{geometry}

\renewcommand*{\backref}[1]{\ifx#1\relax \else Page #1 \fi}
\renewcommand*{\backrefalt}[4]{%
  \ifcase #1 \footnotesize{(Not cited.)}%
  \or        \footnotesize{(Cited on page~#2.)}%
  \else      \footnotesize{(Cited on pages~#2.)}%
  \fi
}
\numberwithin{equation}{section}

\usepackage{tikz}
\usetikzlibrary{arrows.meta,angles,quotes,calc}

\theoremstyle{plain}
\newtheorem{theorem}{Theorem}[section]
\newaliascnt{proposition}{theorem}
\newtheorem{proposition}[proposition]{Proposition}
\aliascntresetthe{proposition}
\newaliascnt{lemma}{theorem}
\newtheorem{lemma}[lemma]{Lemma}
\aliascntresetthe{lemma}
\newaliascnt{corollary}{theorem}
\newtheorem{corollary}[corollary]{Corollary}
\aliascntresetthe{corollary}
\theoremstyle{definition}
\newaliascnt{definition}{theorem}
\newtheorem{definition}[definition]{Definition}
\aliascntresetthe{definition}
\newaliascnt{assumption}{theorem}
\newtheorem{assumption}[assumption]{Assumption}
\aliascntresetthe{assumption}
\newaliascnt{example}{theorem}
\newtheorem{example}[example]{Example}
\aliascntresetthe{example}
\theoremstyle{remark}
\newaliascnt{remark}{theorem}
\newtheorem{remark}[remark]{Remark}
\aliascntresetthe{remark}

\crefname{theorem}{Theorem}{Theorems}
\crefname{proposition}{Proposition}{Propositions}
\crefname{lemma}{Lemma}{Lemmas}
\crefname{corollary}{Corollary}{Corollaries}
\crefname{definition}{Definition}{Definitions}
\crefname{assumption}{Assumption}{Assumptions}
\crefname{remark}{Remark}{Remarks}
\crefname{example}{Example}{Examples}

\newcommand{\R}{\mathbb{R}}
\newcommand{\N}{\mathbb{N}}
\newcommand{\Nzero}{\mathbb{N}_0}
\newcommand{\Z}{\mathbb{Z}}
\newcommand{\E}{\mathbb{E}}
\newcommand{\Pp}{\mathbb{P}}
\newcommand{\1}{\mathbf{1}}

\newcommand{\TV}{\mathrm{TV}}
\newcommand{\Law}{\operatorname{Law}}
\newcommand{\supp}{\operatorname{supp}}
\newcommand{\dist}{\operatorname{dist}}
\newcommand{\ov}{\operatorname{OV}}
\newcommand{\Tr}{\operatorname{tr}}
\newcommand{\ip}[2]{\left\langle #1,#2\right\rangle}
\newcommand{\norm}[1]{\left\lVert #1\right\rVert}
\newcommand{\abs}[1]{\left\lvert #1\right\rvert}
\newcommand{\dd}{\,\mathrm{d}}

\newcommand{\PM}{P_{\mathrm M}}
\newcommand{\PB}{P_{\mathrm B}}

\newcommand{\RB}{R_{\mathrm B}}
\newcommand{\flow}{\varphi}

\newcommand{\gap}{\operatorname{gap}}
\newcommand{\sinc}{\operatorname{sinc}}

\title{\textbf{A Profile-Separation Framework for Quantitative Convergence of No-U-Turn Samplers}}

\author{Krishnakumar Balasubramanian\\
Department of Statistics, University of California, Davis. \\
\texttt{kbala@ucdavis.edu}
}

\date{}
\begin{document}
\maketitle

\begin{abstract}
We study multinomial and biased-progressive No-U-Turn Samplers for strongly log-concave targets satisfying $mI_d\preceq \nabla^2U(x)\preceq LI_d,$ and $\|\nabla^2U(x)-\nabla^2U(y)\|_{\mathrm F}\le \gamma L^{3/2}\|x-y\|$ with \(\kappa\coloneqq L/m\).  We introduce profile separation, a sufficient sign condition on the stationary mean U-turn diagnostics, and combine it with diagnostic concentration, leapfrog fidelity, and whole-orbit energy control to show that on a high-probability certification event, every doubling realization reaches a common terminal depth through a genuine U-turn.  If \(T_\star\) is the selected physical trajectory length and \(a_\star=\sqrt m\,T_\star\), a terminal-depth transfer argument yields restricted conductance and warm-start mixing without lazifying either kernel.  Up to logarithmic warm-start and accuracy factors, the transition bounds are
\[
  \widetilde O\!\left(
    1+a_\star^2\kappa^2(1+\gamma)^{4/3}
  \right)
  \quad\text{and}\quad
  \widetilde O\!\left(
    1+a_\star^4\kappa^3(1+\gamma)^2
  \right)
\]
for multinomial and biased-progressive selection, respectively. These transition bounds are unconditional.  Gradient-work bounds are deterministic when the maximum-depth cap is comparable to the certified depth and otherwise take cap-aware expected and high-probability forms. The framework recovers the Gaussian dimension dependence under these work-accounting conditions, provides population-profile and exact-diagnostic verification for nonlinear product targets, a near-isotropic specialization of the practical-tree certificate, and quantifies when a fixed post-warmup metric removes linear anisotropy.
\end{abstract}

\medskip

\section{Introduction}
\label{sec:introduction}

Hamiltonian Monte Carlo (HMC)~\citep{duane1987hybrid,neal2011mcmc,betancourt2017geometric} suppresses random-walk behavior by using
Hamiltonian trajectories to generate distant proposals, but its
performance depends strongly on the integration time.  A trajectory
that is too short produces little movement, whereas a trajectory that
is too long can waste computation by retracing regions already
visited.  The No-U-Turn Sampler (NUTS) of
\citet{HoffmanGelman2014} addresses this tuning problem by constructing
a leapfrog orbit recursively and terminating it when endpoint momentum
vectors indicate that the trajectory has begun to turn back toward its
starting point.  This state-dependent choice of trajectory length, implemented in widely used probabilistic-programming systems such as Stan, PyMC, NumPyro, and Turing.jl, is central to the practical success of NUTS, but it is also the main obstacle to a quantitative mixing theory.

The difficulty is not merely that the integration time is random.  A
NUTS transition depends jointly on a refreshed momentum, a sequence of
random left--right doubling decisions, the signs of endpoint
diagnostics on every completed tree and recursively inspected subtree,
the numerical energy error over the resulting orbit, and the rule used
to select the next state from that orbit.  Consequently, quantitative
results for fixed-time HMC do not transfer directly to NUTS.  One
must first show that the adaptive tree construction itself is stable
and then determine whether the completed orbit contains enough
well-behaved fixed-time movement to cross arbitrary bottlenecks of the
target distribution.

The invariance and qualitative convergence theory of dynamic HMC and NUTS constructions has been developed by
\citet{DurmusEtAl2024} and within the Gibbs Self-tuning (GIST) framework of
\citet{BouRabeeCarpenterMarsden2026}.  These works isolate the orbit-selection and
re-rooting symmetries needed for reversibility and ergodicity.
Quantitative mixing results have appeared more recently and have so far
relied primarily on Gaussian structure.  For the canonical Gaussian,
\citet{BouRabeeOberdoerster2024} prove a
\(\widetilde O(d^{1/4})\) mixing bound by showing that the random U-turn
diagnostics concentrate around the deterministic function \(d\sin t\).
For anisotropic Gaussian targets,
\citet{Oberdoerster2025} \textcolor{black}{identifies} a covariance-dependent deterministic
term and show that its sign pattern explains both accelerated and
trapped orbit-selection regimes.  \citet{GruffazEtAl2026} compare
multinomial and biased-progressive state selection and obtain
\(\widetilde O(d^{1/4})\) Gaussian mixing results for both variants.
These analyses establish that concentration of a deterministic
time-dependent signal can make the adaptive NUTS transition nearly
state independent.  

This paper develops such a formulation for $(m,L)$-strongly log-concave
smooth targets.  We study the revised random-doubling orbit builder with
recursive sub-U-turn rejection and two state-selection rules:
multinomial NUTS and biased-progressive NUTS.  The step size, kinetic
metric, and maximum depth are fixed during the stationary sampling
phase.  No fixed-length HMC branch, artificial self-loop, modified
stopping rule, or auxiliary safeguard is added to either kernel.

\paragraph{The stationary U-turn profile.}
Let \((X_t,V_t) \in \mathbb{R}^{2d}\) denote exact Hamiltonian flow initialized from
equilibrium.  We consider the stationary U-turn profile
\begin{align*}
  u(t)
  =
  \E\!\left[
    V_0^\top(X_t-X_0)
  \right]
  =
  \E\!\left[
    V_t^\top(X_t-X_0)
  \right].
\end{align*}
Time reversibility and differentiation of the equilibrium
mean-square displacement give
\begin{align*}
  u(t)
  =
  \frac12\frac{\dd}{\dd t}
  \E\norm{X_t-X_0}^2.
\end{align*}
Gaussian integration by parts gives the additional Jacobi-field
representation
\begin{align*}
  u(t)
  =
  \E\Tr(J_t),
  \qquad
  J_t=\frac{\partial X_t}{\partial V_0}.
\end{align*}
Thus \(u(t)\) is the equilibrium mean of either endpoint diagnostic
used by the NUTS stopping rule.  It is also the derivative of the
mean-square displacement and the trace of the linear response of
position to the initial momentum.

For a leapfrog step size \(h\), the physical durations inspected by the
doubling procedure are $  \tau_k=h(2^k-1).$ The central population-level object in this paper is
\emph{profile separation}, introduced in
\Cref{def:profile-separated}.  A depth \(k_\star\) is profile separated
when the profile is uniformly positive at every relevant preterminal
dyadic duration and uniformly negative at the candidate terminal
duration, with a margin of order \(d/\sqrt m\).  Very short intervals
are excluded from this macroscopic margin condition and are handled by
a deterministic no-U-turn estimate.


The population profile alone does not determine a numerical NUTS tree.
We therefore introduce intrinsic quantities that measure precisely the
errors relevant to the practical stopping rule: concentration of the
exact endpoint diagnostics around \(u(t)\), leapfrog error in those
scalar diagnostics, failure of a uniform energy window, and failure of
positivity on very short checked intervals.  These quantities do not
require a bound on the full Hamiltonian Jacobian.

The main result, \Cref{thm:adaptive}, states that if a depth
\(k_\star<k_{\max}\) is profile separated and the combined stochastic
and numerical diagnostic error is smaller than the profile margin,
then, on a high-probability phase-space event,

\begin{enumerate}[noitemsep, label=(\roman*),leftmargin=2.2em]
\item every completed orbit and recursively inspected proper subtree
      below depth \(k_\star\) has two positive numerical endpoint
      diagnostics;

\item every possible completed orbit at depth \(k_\star\) has two
      negative numerical endpoint diagnostics;

\item every realization of the random left--right doubling decisions
      returns the same cardinality
      \(K_\star=2^{k_\star}\);

\item termination occurs through the U-turn criterion, strictly before
      the maximum-depth cap; and

\item numerical energies remain uniformly comparable along every
      possible terminal orbit.
\end{enumerate}

Writing
\[
  T_\star=(K_\star-1)h,
  \qquad
  a_\star=\sqrt m\,T_\star,
  \qquad
  \kappa=\frac Lm,
\]
the theorem gives the warm-start transition bounds
\begin{align*}
  n_{\mathrm M}
  &\le
  C\left[
    1+
    a_\star^2\kappa^2(1+\gamma)^{4/3}
  \right]
  \log\!\left(\frac{CM}{\epsilon}\right),
  \\
  n_{\mathrm B}
  &\le
  C\left[
    1+
    a_\star^4\kappa^3(1+\gamma)^2
  \right]
  \log\!\left(\frac{CM}{\epsilon}\right),
\end{align*}
for multinomial and biased-progressive NUTS, respectively.  These are unconditional transition bounds.  On a certified transition the orbit cost is of order \(K_\star\), but an uncertified transition may run to the maximum size
\[
  K_{\mathrm{cap}}=2^{k_{\max}}.
\]
Accordingly, the deterministic work bound for \(n_a\) transitions is of order \(K_{\mathrm{cap}}n_a\).  We prove sharper expected and high-probability bounds that interpolate between \(K_\star n_a\) and \(K_{\mathrm{cap}}n_a\), and recover deterministic order \(K_\star n_a\) when \(K_{\mathrm{cap}}=O(K_\star)\).  These bounds are deliberately expressed in
terms of the trajectory length selected by the U-turn rule rather than
an externally prescribed integration time.

\subsection{Contributions}
The main contributions are as follows.

\begin{enumerate}[leftmargin=2.4em,label=\textup{(\arabic*)}]

\item
\emph{A target-independent stopping certificate:}
We abstract the deterministic-profile mechanism implicit in the
Gaussian analyses of
\citet{BouRabeeOberdoerster2024,Oberdoerster2025} into profile separation.
The profile
\[
  u(t)
  =
  \frac12\frac{\dd}{\dd t}
  \E\norm{X_t-X_0}^2
  =
  \E\Tr(J_t)
\]
is defined for a general strongly log-concave target and replaces the
explicit Gaussian sine functions.  The intrinsic terminal-depth
certificate in \Cref{prop:intrinsic-certificate} transfers its signs
simultaneously to every full orbit and recursive subtree that the
practical leapfrog tree can inspect.

\item
\emph{A terminal-depth-to-conductance theorem for two NUTS selectors:}
Once a common terminal depth and an energy window have been certified,
we show that each NUTS transition contains an explicit mixture of
fixed-index leapfrog proposals.  The selector weights are triangular
for multinomial selection and double-tent shaped for
biased-progressive selection.  A fixed-time overlap estimate then
gives overlap of nearby NUTS transition laws, and Cheeger isoperimetry
converts this local overlap into restricted conductance.  This is
formalized by \Cref{thm:terminal-transfer}.  The result separates the
trajectory problem---proving where the adaptive tree stops---from the
mixing problem---proving that the resulting kernel crosses every
measurable cut.

\item
\emph{Positive operators for the original, non-lazy kernels:}
The initial-point symmetry of revised random doubling and the
finite-orbit Biased Progressive Sampling (BPS) detailed-balance identity are established properties
of the modern NUTS construction
\citep{DurmusEtAl2024,BouRabeeCarpenterMarsden2026}.  We use these identities in a
different way.  For multinomial NUTS, orbit re-rooting is represented
as an orthogonal projection on an augmented Hilbert space, implying $  \PM\succeq0.$ For BPS, the finite-orbit selector can have negative eigenvalues, but
reversibility implies $  \RB=\PB^2\succeq0.$ These positive operators allow the conductance argument to control the
original multinomial chain and the unchanged two-step BPS skeleton,
without artificially lazifying either sampler.

\item
\emph{General trajectory and numerical-analysis modules:}
We prove equilibrium identities for the nonlinear profile,
mean-square-displacement comparison inequalities, and the universal
initial-positivity bound
\[
  u(t)
  \ge
  \frac d{\sqrt L}\sin(\sqrt L\,t),
  \qquad
  0<t\le\frac{\pi}{2\sqrt L}.
\]
We establish concentration of the exact scalar endpoint diagnostics,
a deterministic leapfrog-to-flow comparison for those diagnostics,
and a simultaneous no-U-turn lemma for short numerical intervals.
From the one-step energy moment estimate of
\citet{ChenGatmiryJiang2026}, we derive a whole-orbit energy window by
a stopped induction and volume-preserving changes of variables.  From
the same work we use the fixed-index leapfrog proposal-overlap
estimate. Unlike~\cite{ChenGatmiryJiang2026}, we do not use Metropolis correction, fixed-length HMC kernel or lazification.

\item
\emph{Non-Gaussian profile verification:}
We give a Koopman spectral representation
\[
  u(t)
  =
  \int_{(0,\infty)}
  \omega\sin(\omega t)\,\nu_U(\dd\omega)
\]
and derive a positive-frequency crossing criterion.  It gives profile crossing and exact-diagnostic concentration
for one-dimensional strongly log-concave smooth targets and for independent
nonlinear products at arbitrary finite condition number. We also give perturbative profile bounds for near-isotropic coupled
targets and a metric-adapted result showing that a fixed post-warmup
kinetic metric can remove raw linear anisotropy.  For the product class, the practical leapfrog fidelity and whole-orbit energy conditions remain separate hypotheses unless verified by an additional numerical argument.  For more general
coupled nonlinear targets, we formulate polynomial tangent-stability
and relative log-Hessian variation conditions that suffice to verify
the intrinsic certificate.  These are sufficient verification routes,
not assumptions built into the main theorem.

\item
\emph{Gaussian specialization and reconciliation with recent theory:}
For \(\pi=N(0,C)\), with precision \(\Lambda=C^{-1}\), we show
\begin{align*}
  u(t)
  =
  \Tr\!\left(
    \Lambda^{-1/2}\sin(\Lambda^{1/2}t)
  \right)
  =
  \Tr\!\left(
    \sin(C^{-1/2}t)C^{1/2}
  \right).
\end{align*}
The second expression is exactly the deterministic uniform term in
\citet{Oberdoerster2025}.  We prove a finite-dimensional margin
certificate, direct mode-by-mode leapfrog fidelity, and a uniform
energy bound (with no exponential dependence on \(\kappa\)).  We then
show that profile separation recovers the two-scale accelerated/trapped
dichotomy: the accelerated phase is precisely the absence of a
negative profile value during the first fast period, while trapped
step sizes are dyadic preimages of the negative fast-profile set.

For the exact-flow, equal-energy Gaussian profile kernels associated
with both selectors, let \(\widetilde T=Kh\) be the selector-rescaling time
and \(\widetilde a=\sqrt m\,\widetilde T\).  The physical span of the
\(K\)-point orbit is \((K-1)h\), which is universally comparable to
\(\widetilde T\) for \(K\ge4\).  We obtain
\[
  K n_\sigma
  =
  \widetilde O\!\left(
    \sqrt\kappa\,d^{1/4}
    \frac{
      \widetilde a
    }{
      \min\{\widetilde a^2,1\}
    }
    \log\!\frac{M}{\epsilon}
  \right).
\]
This gives
\(\widetilde O(\sqrt\kappa\,d^{1/4})\) when
\(\widetilde a=\Theta(1)\) and
\(\widetilde O(\kappa d^{1/4})\) when
\(\widetilde a=\Theta(\kappa^{-1/2})\), subject to the stated time-grid
resolution condition.  The sharp interpolation is proved for the
Gaussian profile kernels.  We do not claim the corresponding sharp
practical leapfrog interpolation for anisotropic BPS.

\end{enumerate}

\paragraph{Relation to existing proof techniques.}
The paper builds upon several established ingredients with many new
NUTS-specific reductions.  Initial-point symmetry of random doubling,
finite-orbit BPS detailed balance, and the ideal triangular and
double-tent selector laws come from the modern dynamic-HMC and Gaussian
NUTS literature
\citep{DurmusEtAl2024,BouRabeeCarpenterMarsden2026,Oberdoerster2025,GruffazEtAl2026}.
The single-step leapfrog energy moment and fixed-index proposal-overlap
estimates are obtained following \citet{ChenGatmiryJiang2026}.  The conversion from conductance to
warm-start mixing for a positive kernel is a standard
nonnegative-spectrum variant of the classical conductance argument.

The new proof architecture begins before those ingredients can be
used.  It identifies a population profile for nonlinear targets,
quantifies the stochastic and numerical errors of the actual U-turn
diagnostics, certifies one common terminal tree depth for every
doubling realization on a high-probability phase-space event, derives energy-window minorizations of both
state-selection rules, and converts the resulting fixed-time movement
into conductance for the original NUTS kernels.  In this sense, the
main contribution is not a new fixed-time HMC bound.  It is a
trajectory-to-mixing reduction showing when the adaptive stopping
mechanism of practical NUTS behaves predictably enough for fixed-time
HMC estimates to become useful. The transition-mixing conclusions are global; computational work is accounted for separately because uncertified transitions remain part of the actual Markov chain and may run to the maximum-depth cap.

\subsection{Related work}
\label{sec:related-work}

Hamiltonian Monte Carlo (HMC) was introduced by
\citet{duane1987hybrid} and developed as a general-purpose MCMC
methodology by \citet{neal2011mcmc}.  Its practical and theoretical
performance depends on both the numerical step size and the integration
time.  High-dimensional scaling for leapfrog HMC was studied by
\citet{BeskosEtAl2013}, while nonasymptotic convergence analyses for
strongly log-concave and more general targets were developed using
continuous-dynamics arguments, numerical-integrator estimates,
coupling, and conductance
\citep{MangoubiSmith2019,MangoubiSmith2021,
BouRabeeEberleZimmer2020,ChenDwivediWainwrightYu2020}.
Randomizing the integration time provides a principled way to avoid
periodic or poorly tuned trajectories, as formalized by randomized HMC
\citep{BouRabeeSanzSerna2017} and exploited for sharp Gaussian sampling
bounds by \citet{ApersGriblingSzilagyi2024}.  Most directly relevant to
the present analysis, \citet{ChenGatmiryJiang2026} obtain a warm-start
\(\widetilde O(d^{1/4})\) gradient complexity for Metropolized HMC under
smoothness, Frobenius-Hessian regularity, and isoperimetry; we use their
single-step energy-moment and fixed-time proposal-overlap estimates, but
not their fixed-length kernel, Metropolis correction, or lazification.

The No-U-Turn Sampler (NUTS) was introduced by
\citet{HoffmanGelman2014} to adapt the HMC path length through a
recursively constructed trajectory and an endpoint U-turn criterion.
The dependence of the completed orbit and selected state on the current
position, refreshed momentum, doubling decisions, and numerical
trajectory makes correctness and convergence substantially more
delicate than for fixed-time HMC.  \citet{DurmusEtAl2024} place NUTS
within a general dynamic-HMC framework and establish invariance,
irreducibility, and qualitative ergodicity under appropriate
conditions.  The GIST framework of
\citet{BouRabeeCarpenterMarsden2026} gives an auxiliary-variable
formulation of locally adaptive HMC and identifies NUTS, randomized HMC,
and related path-length adaptations as instances of a common
construction.  These works provide the orbit-selection, re-rooting, and
detailed-balance structure used below; their principal focus is
algorithmic correctness and qualitative convergence rather than
nonasymptotic mixing for nonlinear strongly log-concave targets.

Quantitative theory for NUTS is more recent and has largely relied on
Gaussian structure.  \citet{BouRabeeOberdoerster2024} prove a
\(\widetilde O(d^{1/4})\) gradient bound for the canonical Gaussian by
showing that concentration makes the locally adapted NUTS transition
nearly state independent.  \citet{Oberdoerster2025} \textcolor{black}{extends} this
viewpoint to anisotropic and two-scale Gaussian targets, where a
covariance-dependent deterministic U-turn signal explains both
accelerated and trapped orbit-selection regimes.
\citet{GruffazEtAl2026} compare multinomial and biased-progressive NUTS,
establish \(\widetilde O(d^{1/4})\) Gaussian mixing for both selectors,
and study their qualitative ergodicity.  The present work abstracts the
deterministic-signal mechanism through the stationary U-turn profile
and the sufficient condition of profile separation.  Combined with
intrinsic diagnostic concentration, leapfrog fidelity, and energy
control, this yields a terminal-depth-to-conductance analysis for both
selectors on strongly log-concave targets, while recovering the
deterministic Gaussian profile and the previously identified two-scale
regimes as special cases.

Recent work has broadened high-accuracy sampling theory beyond Metropolized HMC.  In complementary directions, \citet{LuLuo2026Windowed} give exact
windowed-thinning implementations and Gaussian-cold-start query bounds
for the bouncy particle and Zigzag samplers, while
\citet{MonmarcheWang2026Entropy} prove square-root entropic acceleration
for continuous-time randomized HMC under a log-Sobolev inequality and
mild nonconvexity, but show that uniform exponential contraction in
relative entropy fails for bouncy particle and Zigzag samplers even for a standard Gaussian
target.  Finally, \citet{ChenChewiRakhlinZhang2026} use path-space
rejection correction of underdamped Langevin diffusion to obtain
\(\widetilde O(\kappa^{2/3}d^{1/3})\) first-order query complexity in
Rényi divergence, with improved \(d^{1/5}\) dependence in a bounded
higher-smoothness regime using Hessian--vector queries.  These results
are complementary to the HMC analysis: they replace the practical
leapfrog--Metropolis implementation and warm-start conductance argument
by, respectively, exact PDMP simulation, continuous-time entropic
analysis, or rejection correction of an entire diffusion path.

\subsection{Organization.}
\Cref{sec:common-kernel} defines the common random-doubling orbit and
the two state-selection rules.
\Cref{sec:main} introduces profile separation, the intrinsic
diagnostic moduli, and the main theorem.
\Cref{sec:spectral} develops reversibility and the positive operators
used in the mixing argument. \Cref{sec:terminal} states the terminal-depth transfer theorem.
\Cref{sec:analytic} records the short-trajectory, energy, and
fixed-time overlap estimates.
\Cref{sec:profile-proof} develops the nonlinear stationary profile and
the terminal-depth certificate.
The non-Gaussian and Gaussian specializations appear in
\Cref{sec:non-gaussian-specializations,sec:gaussian-specialization}, respectively.
The positive-kernel conductance argument, selector laws, and detailed
specialization proofs are deferred to the appendices.

\section{The common dynamic NUTS orbit}
\label{sec:common-kernel}

\subsection{Hamiltonian dynamics and leapfrog integration}

Let
\[
  \pi(\dd x)=Z^{-1}e^{-U(x)}\dd x,
  \qquad x\in\R^d,
\]
where \(U:\R^d\to\R\) is such that \(U\in C^3(\R^d)\) and \(Z<\infty\).  Introduce an independent momentum \(v\sim N(0,I_d)\) and the phase-space law
\begin{align*}
  \bar\pi(\dd x\,\dd v)
  =\pi(\dd x)\varphi(\dd v)
  =\bar Z^{-1}e^{-H(x,v)}\dd x\,\dd v,
  \qquad
  H(x,v)=U(x)+\frac12\norm{v}^2.
\end{align*}
For \(h>0\), one leapfrog step \(\Phi_h\) is
\begin{align}
  v_{1/2}&=v-\frac h2\nabla U(x),\label{eq:lf1}\\
  x'&=x+h v_{1/2},\notag\\
  v'&=v_{1/2}-\frac h2\nabla U(x').\label{eq:lf3}
\end{align}
Defining the momentum-kick and position-drift maps respectively by $K_\tau(x,v)
  =
  \bigl(x,\;v-\tau\nabla U(x)\bigr)$ and $D_\tau(x,v)
  =
  \bigl(x+\tau v,\;v\bigr)$, the leapfrog map is $\Phi_h
  =
  K_{h/2}\circ D_h\circ K_{h/2}$. Equivalently,
\[
  \Phi_h(x,v)
  =
  \left(
    x+h v-\frac{h^2}{2}\nabla U(x),
    \;
    v-\frac h2\nabla U(x)
      -\frac h2\nabla U\!\left(
        x+h v-\frac{h^2}{2}\nabla U(x)
      \right)
  \right).
\]
The momentum-flip map is $  S(x,v)=(x,-v).$ Recall that a \emph{shear map} leaves one block of coordinates fixed and translates
the other block by an amount depending only on the fixed block.  In this context,
\(K_\tau\) leaves \(x\) unchanged and translates \(v\) by
\(-\tau\nabla U(x)\), whereas \(D_\tau\) leaves \(v\) unchanged and
translates \(x\) by \(\tau v\).  Although such shear maps may distort shapes,
they do not change phase-space volume.  Hence each leapfrog substep is
volume preserving, and therefore their composition $K_{h/2}\circ D_h\circ K_{h/2}
$ is also volume preserving. Direct substitution into \eqref{eq:lf1}--\eqref{eq:lf3} gives
\begin{align}
  \Phi_h^{-1}=S\circ\Phi_h\circ S.
  \label{eq:lf-reversible}
\end{align}

\subsection{Index orbits, endpoint U-turns, and recursive sub-U-turns}

An \emph{index orbit} is a finite consecutive set \(I\subset\Z\)
containing zero. Given \(z=(x,v)\), write
\[
  z_i=(x_i,v_i)=\Phi_h^i z,
  \qquad i\in\Z.
\]
Writing \(a=\min I\) and \(b=\max I\), its endpoint U-turn indicator is
\begin{align}
  u_z(I)
  =
  \1\!\left\{
    v_b^\top(x_b-x_a)<0
    \ \text{or}\ 
    v_a^\top(x_b-x_a)<0
  \right\}.
  \label{eq:uturn}
\end{align}

\begin{tikzpicture}[
    >=Latex,
    point/.style={
        circle,
        fill=black,
        inner sep=2pt
    },
    displacement/.style={
        ->,
        very thick,
        blue
    },
    velocity/.style={
        ->,
        very thick,
        red
    },
    every node/.style={
        font=\small
    }
]

\begin{scope}[xshift=0cm]

\node[font=\bfseries] at (2,2.8) {No U-turn: $u_z(I)=0$};

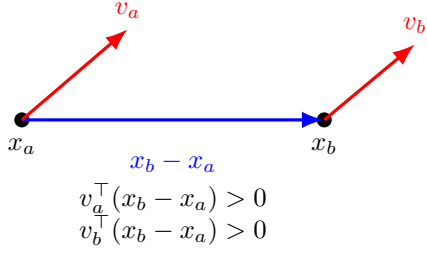
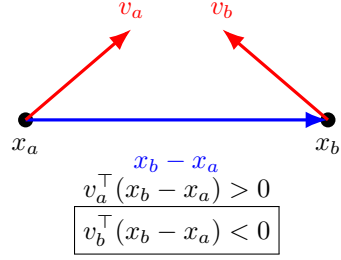
\coordinate (A1) at (0,0);
\coordinate (B1) at (4,0);

\node[point,label=below:$x_a$] at (A1) {};
\node[point,label=below:$x_b$] at (B1) {};

\draw[displacement]
    (A1) -- (B1)
    node[midway,below=8pt] {$x_b-x_a$};

\draw[velocity]
    (A1) -- ++(1.4,1.2)
    node[above] {$v_a$};

\draw[velocity]
    (B1) -- ++(1.2,1.0)
    node[above] {$v_b$};

\node[align=center] at (2,-1.25) {
    $v_a^\top(x_b-x_a)>0$\\
    $v_b^\top(x_b-x_a)>0$
};

\end{scope}

\draw[dashed,gray] (6,-1.8) -- (6,3);

\begin{scope}[xshift=8cm]

\node[font=\bfseries] at (2,2.8) {U-turn: $u_z(I)=1$};

\coordinate (A2) at (0,0);
\coordinate (B2) at (4,0);

\node[point,label=below:$x_a$] at (A2) {};
\node[point,label=below:$x_b$] at (B2) {};

\draw[displacement]
    (A2) -- (B2)
    node[midway,below=8pt] {$x_b-x_a$};

\draw[velocity]
    (A2) -- ++(1.4,1.2)
    node[above] {$v_a$};

\draw[velocity]
    (B2) -- ++(-1.4,1.2)
    node[above] {$v_b$};

\node[align=center] at (2,-1.25) {
    $v_a^\top(x_b-x_a)>0$\\
    $\boxed{v_b^\top(x_b-x_a)<0}$
};

\end{scope}

\node[blue] at (4.8,-2.2) {$\longrightarrow$ displacement};
\node[red]  at (8.0,-2.2) {$\longrightarrow$ velocity};

\end{tikzpicture}

\vspace{0.2in}
For a dyadic interval \(I\), let \(I^-\) and \(I^+\) be its left and right halves.  Define the recursive sub-U-turn indicator by
\begin{align}
  \sigma_z(I)
  =
  \begin{cases}
    0, & |I|=1,\\[0.25em]
    \max\{\sigma_z(I^-),\sigma_z(I^+),u_z(I)\},&|I|\ge2.
  \end{cases}
  \label{eq:subuturn}
\end{align}
Thus \(\sigma_z(I)=1\) exactly when at least one dyadic descendant of \(I\), including \(I\) itself, satisfies the endpoint criterion.

\begin{tikzpicture}[
  every node/.style={font=\small},
  boxyes/.style={
    draw,
    rounded corners,
    fill=red!15,
    align=center,
    minimum width=2.5cm,
    minimum height=1.1cm
  },
  boxno/.style={
    draw,
    rounded corners,
    fill=green!15,
    align=center,
    minimum width=2.5cm,
    minimum height=1.1cm
  },
  leaf/.style={
    draw,
    rounded corners,
    fill=gray!10,
    align=center,
    minimum width=1.3cm,
    minimum height=0.9cm
  }
]

\node[font=\bfseries] at (0,9.3) {Recursive sub-U-turn indicator};

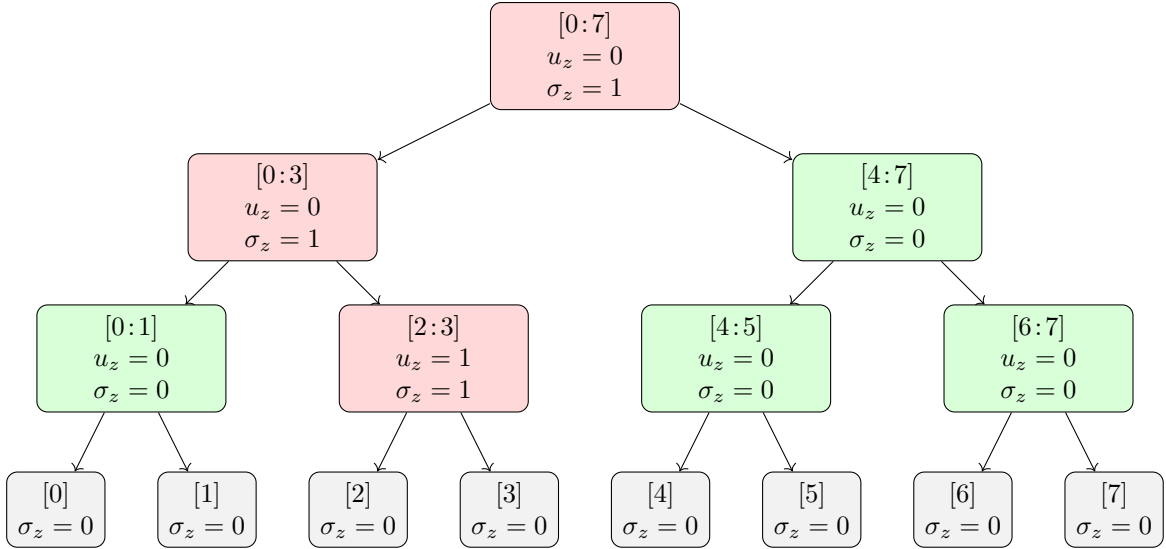
\node[boxyes] (r) at (0,8) {$[0\!:\!7]$\\ $u_z=0$\\ $\sigma_z=1$};

\node[boxyes] (l1) at (-4,6) {$[0\!:\!3]$\\ $u_z=0$\\ $\sigma_z=1$};
\node[boxno]  (r1) at ( 4,6) {$[4\!:\!7]$\\ $u_z=0$\\ $\sigma_z=0$};

\draw[->] (r) -- (l1);
\draw[->] (r) -- (r1);

\node[boxno]  (l2a) at (-6,4) {$[0\!:\!1]$\\ $u_z=0$\\ $\sigma_z=0$};
\node[boxyes] (l2b) at (-2,4) {$[2\!:\!3]$\\ $u_z=1$\\ $\sigma_z=1$};

\node[boxno]  (r2a) at ( 2,4) {$[4\!:\!5]$\\ $u_z=0$\\ $\sigma_z=0$};
\node[boxno]  (r2b) at ( 6,4) {$[6\!:\!7]$\\ $u_z=0$\\ $\sigma_z=0$};

\draw[->] (l1) -- (l2a);
\draw[->] (l1) -- (l2b);
\draw[->] (r1) -- (r2a);
\draw[->] (r1) -- (r2b);

\node[leaf] (a0) at (-7,2) {$[0]$\\ $\sigma_z=0$};
\node[leaf] (a1) at (-5,2) {$[1]$\\ $\sigma_z=0$};
\node[leaf] (a2) at (-3,2) {$[2]$\\ $\sigma_z=0$};
\node[leaf] (a3) at (-1,2) {$[3]$\\ $\sigma_z=0$};

\node[leaf] (b0) at ( 1,2) {$[4]$\\ $\sigma_z=0$};
\node[leaf] (b1) at ( 3,2) {$[5]$\\ $\sigma_z=0$};
\node[leaf] (b2) at ( 5,2) {$[6]$\\ $\sigma_z=0$};
\node[leaf] (b3) at ( 7,2) {$[7]$\\ $\sigma_z=0$};

\draw[->] (l2a) -- (a0);
\draw[->] (l2a) -- (a1);
\draw[->] (l2b) -- (a2);
\draw[->] (l2b) -- (a3);

\draw[->] (r2a) -- (b0);
\draw[->] (r2a) -- (b1);
\draw[->] (r2b) -- (b2);
\draw[->] (r2b) -- (b3);

\node[align=left] at (0,0.3) {%
\begin{minipage}{12cm}
\centering
Red nodes indicate $\sigma_z(I)=1$.\\[0.3em]
The interval $[2\!:\!3]$ is the first one where the endpoint test fires, so $u_z([2\!:\!3])=1$.\\
Then the recursion propagates this upward:
$\sigma_z([2\!:\!3])=1 \Rightarrow \sigma_z([0\!:\!3])=1 \Rightarrow \sigma_z([0\!:\!7])=1$.
\end{minipage}
};

\end{tikzpicture}

\subsection{Random doubling}

Fix a maximum depth \(k_{\max}\in\N\) and write $K=2^{k_{\max}}.$  The orbit builder begins with \(I_0=\{0\}\).  At level \(r\), if \(|I_r|=K\) or \(u_z(I_r)=1\), it stops and returns \(I_r\).  Otherwise it samples a fair left/right bit and proposes the adjacent block \(N_r\) of size \(|I_r|\) on that side.  If \(\sigma_z(N_r)=1\), it rejects the block and returns \(I_r\).  If \(\sigma_z(N_r)=0\), it accepts the doubling, sets
\begin{align*}
  I_{r+1}=I_r\sqcup N_r,
\end{align*}
and continues.  For example, starting from \(I_0=\{0\}\), the builder may choose right and
accept \(N_0=\{1\}\), giving \(I_1=[0:1]\). It may then choose left and
accept \(N_1=[-2:-1]\), giving \(I_2=[-2:1]\). If the next proposed block
\(N_2=[2:5]\) satisfies \(\sigma_z(N_2)=1\), it is rejected and the final
returned orbit is \(I_2=[-2:1]\). The returned history is
\begin{align}
  \mathcal H=(I_0,N_0,I_1,N_1,\ldots,N_{R-1},I_R),
  \label{eq:history}
\end{align}
where \(I_R\) is the final orbit.  The history is needed only to define BPS; multinomial NUTS uses only \(I_R\).

\begin{algorithm}[t]
\caption{Common revised NUTS orbit builder}
\label{alg:orbit}
\begin{algorithmic}
\Require Phase point \(z=(x,v)\), step size \(h\), maximum size \(K=2^{k_{\max}}\)
\State Set \(I\gets\{0\}\), \(\mathcal H\gets(I)\)
\While{\(|I|<K\) and \(u_z(I)=0\)}
  \State Draw \(B\sim\operatorname{Bernoulli}(1/2)\)
  \If{\(B=0\)}
    \State Let \(N\) be the adjacent block of length \(|I|\) to the left
  \Else
    \State Let \(N\) be the adjacent block of length \(|I|\) to the right
  \EndIf
  \If{\(\sigma_z(N)=1\)}
    \State \textbf{break}
  \Else
    \State Append \(N\) to \(\mathcal H\), set \(I\gets I\sqcup N\), and append \(I\) to \(\mathcal H\)
  \EndIf
\EndWhile
\State \Return final orbit \(I\) and accepted-doubling history \(\mathcal H\)
\end{algorithmic}
\end{algorithm}

Let \(O_z(I)\) denote the probability that \cref{alg:orbit}, started from \(z\), returns the interval \(I\).  The following symmetry is the structural reason that both index-selection mechanisms preserve the target. For a dyadic interval \(D\subset\Z\), let \(\mathscr D(D)\) be the family containing \(D\) and every interval produced by recursively bisecting \(D\) into its left and right halves.  Then \(\sigma_z(D)=\max_{E\in\mathscr D(D)}u_z(E)\).

\begin{proposition}[]
\label{prop:initial-point-symmetry}
For every possible returned interval \(I\) and every \(i\in I\),
\begin{align}
  O_{\Phi_h^iz}(I-i)=O_z(I),
  \qquad
  I-i=\{r-i:r\in I\}.
  \label{eq:initial-point-symmetry}
\end{align}
\end{proposition}

The identity \eqref{eq:initial-point-symmetry} is an established
property of the revised random-doubling NUTS orbit-selection kernel.
In the notation of \citet[Proposition~5]{DurmusEtAl2024}, it is
the symmetry condition
\[
  P_h(J+j\mid \Phi_h^{-j}z)=P_h(J\mid z),
  \qquad -j\in J.
\]
The same property is discussed explicitly as \emph{initial-point
symmetry} in \citet[Remark~7]{BouRabeeCarpenterMarsden2026}.  It is the
kernel-level counterpart of the re-rooting requirement underlying
the original NUTS construction; see
\citet[Condition~C.4]{HoffmanGelman2014}.  The recursive rejection of
extensions containing sub-U-turns is essential for this identity:
after re-rooting at any point of a returned orbit, the alternative
sequence of intermediate doubled intervals must not terminate
prematurely.

We use \eqref{eq:initial-point-symmetry} to compare different
rootings of the same physical leapfrog orbit.  For multinomial state
selection, combining this symmetry with the Boltzmann re-rooting
identity makes the augmented orbit re-rooting operator self-adjoint.
Its idempotence then identifies it as an orthogonal projection, from
which positive semidefiniteness of the multinomial NUTS position
kernel follows.  For biased-progressive state selection, the same
symmetry, together with detailed balance of the finite-orbit BPS
kernel, yields reversibility of the corresponding phase-space and
position kernels.

\subsection{State selection on the completed orbit}
\label{sec:selectors}

For a finite index set \(A\subset\Z\), define its unnormalized Boltzmann weight along the orbit from \(z\) by
\begin{align*}
  W_z(A)=\sum_{j\in A}e^{-H(\Phi_h^jz)}.
\end{align*}
For nonempty \(A\), write
\begin{align}
  \mu_z^A(j)
  =
  \frac{e^{-H(\Phi_h^jz)}}{W_z(A)},
  \qquad j\in A.
  \label{eq:restricted-multinomial}
\end{align}

\subsubsection{Multinomial NUTS}

Starting from \(x\), draw \(v\sim N(0,I_d)\), run \cref{alg:orbit} from \(z=(x,v)\), and let \(I\) be the returned orbit.  Multinomial NUTS draws $J\sim\mu_z^I$ and returns the position component of \(\Phi_h^Jz\).  Its \emph{position kernel} is denoted by \(\PM=\PM^{h,K}\).

\begin{algorithm}[t]
\caption{Capped multinomial NUTS transition}
\label{alg:mul}
\begin{algorithmic}
\Require Current position \(x\), step size \(h\), maximum size \(K\)
\State Draw \(v\sim N(0,I_d)\) and set \(z=(x,v)\)
\State Run \cref{alg:orbit} and obtain final orbit \(I\)
\State Draw \(J\in I\) according to \(\mu_z^I\) in \eqref{eq:restricted-multinomial}
\State \Return \(X'=\Pi_x\Phi_h^Jz\)
\end{algorithmic}
\end{algorithm}

\subsubsection{Biased-progressive NUTS}

Suppose \cref{alg:orbit} accepts \(R\) doublings and returns the history \eqref{eq:history}.  BPS initializes \(C_0=0\).  At accepted doubling \(r\), define
\begin{align}
  A_r
  =
  1\wedge\frac{W_z(N_r)}{W_z(I_r)}.
  \label{eq:bps-acceptance}
\end{align}
Independently conditional on the history and orbit states, set
\begin{align}
  C_{r+1}
  =
  \begin{cases}
    Y_r, &\text{with probability }A_r,\quad Y_r\sim\mu_z^{N_r},\\
    C_r, &\text{with probability }1-A_r.
  \end{cases}
  \label{eq:bps-recursion}
\end{align}
The next position is \(\Pi_x\Phi_h^{C_R}z\).  The resulting \emph{position kernel} is denoted by \(\PB=\PB^{h,K}\).  The recursion \eqref{eq:bps-recursion} is exactly the biased progressive index-selection rule: it may be implemented online as the orbit is built or applied afterward to the recorded accepted-doubling history.

\begin{algorithm}[t]
\caption{Capped NUTS-BPS transition}
\label{alg:bps}
\begin{algorithmic}
\Require Current position \(x\), step size \(h\), maximum size \(K\)
\State Draw \(v\sim N(0,I_d)\) and set \(z=(x,v)\)
\State Run \cref{alg:orbit} and obtain \(\mathcal H=(I_0,N_0,\ldots,N_{R-1},I_R)\)
\State Set \(C\gets0\)
\For{\(r=0,\ldots,R-1\)}
  \State Set \(A_r\gets 1\wedge W_z(N_r)/W_z(I_r)\)
  \State With probability \(A_r\), draw \(C\sim\mu_z^{N_r}\); otherwise retain \(C\)
\EndFor
\State \Return \(X'=\Pi_x\Phi_h^Cz\)
\end{algorithmic}
\end{algorithm}

The two kernels use exactly the same momentum law, leapfrog map, U-turn rule, recursive subtree checks, random-doubling bits, and maximum depth.  They differ only in how the final root is selected.  The distinction relevant to the proof is summarized in \cref{tab:selector-summary}; the time laws and positivity statements are proved later.

\begin{table}[t]
\centering
\caption{Selector-specific structure used in the proof.}
\label{tab:selector-summary}
\begin{tabularx}{\textwidth}{@{}>{\raggedright\arraybackslash}p{0.18\textwidth}>{\raggedright\arraybackslash}p{0.27\textwidth}>{\raggedright\arraybackslash}p{0.27\textwidth}>{\raggedright\arraybackslash}X@{}}
\toprule
Variant & Selection mechanism & Ideal signed-index law at a completed size-\(K\) orbit & Positive operator used for mixing \\
\midrule
Multinomial & One Boltzmann draw from the full orbit & \(\lambda_{\mathrm M,K}(j)=(K-|j|)/K^2\) & \(\PM\) itself \\
BPS & Progressive old/new replacement; final update targets the last new half & \(\lambda_{\mathrm B,K}(j)=2\min\{|j|,K-|j|\}/K^2\), \(j\ne0\) & \(\PB^2\) \\
\bottomrule
\end{tabularx}
\end{table}

\begin{remark}[Notation used in the main argument]
The selector label is \(a\in\{\mathrm M,\mathrm B\}\);
\(\underline a_{\mathrm M}\) and \(\underline a_{\mathrm B}\) are
selector masses assigned to the useful fixed-time band, whereas
\(a_\star=\sqrt m\,T_\star\) is the normalized selected trajectory
length.  The function \(R(t)\) denotes mean-square displacement,
\(\mathcal R\) denotes orbit re-rooting, and
\(R_a=P_a^{r_a}\) denotes the positive kernel used for mixing.
Finally, \(K_{\mathrm{cap}}=2^{k_{\max}}\) is the deterministic maximum
orbit size, while \(K_\star=2^{k_\star}\) is the cardinality certified
on the good phase-space event.
\end{remark}

Our results concern \cref{alg:orbit,alg:mul,alg:bps} with a fixed step size and fixed maximum depth during sampling.  They do not cover adaptation that continues across retained samples, learned mass matrices, variable step sizes within a trajectory, divergence-based truncation, slice selection, or implementation-specific stopping predicates different from \eqref{eq:uturn}--\eqref{eq:subuturn}.

\section{Profile separation and the main theorem}
\label{sec:main}



We make the following assumptions in the rest of the draft, mainly motivated by~\cite{ChenGatmiryJiang2026}. In particular the Frobenius Hessian regularity below is motivated (and implied) by $\gamma$-self concordance. 

\begin{assumption}[Strong log-concavity and smoothness]
\label{ass:convex-smooth}
There are constants \(0<m\le L<\infty\) such that
\begin{align}
mI_d\preceq\nabla^2U(x)\preceq LI_d,
\qquad x\in\R^d.
\label{eq:strong-convexity}
\end{align}
The condition number is \(\kappa=L/m\).
\end{assumption}

\begin{assumption}[Frobenius Hessian regularity]
\label{ass:hessian-lip}
There is a dimensionless \(\gamma\ge0\) such that
\begin{align}
\norm{\nabla^2U(x)-\nabla^2U(y)}_{\mathrm F}
\le \gamma L^{3/2}\norm{x-y},
\qquad x,y\in\R^d.
\label{eq:hessian-lip}
\end{align}
Equivalently, $\norm{\nabla^3U(x)[a,\cdot,\cdot]}_{\mathrm F}
\le \gamma L^{3/2}\norm a.$ We use the lower-curvature normalization $\bar\gamma=\gamma\kappa^{3/2},$ which is the Frobenius Hessian-Lipschitz coefficient after the rescaling \(y=\sqrt m(x-x_\star)\).
\end{assumption}

\begin{definition}[Warm start]
\label{def:warm}
A probability law \(\nu\) is \(M\)-warm with respect to \(\pi\) if $\nu(A)\le M\pi(A)$ for every measurable set \(A\).  Equivalently, \(\nu\ll\pi\) and \(\dd\nu/\dd\pi\le M\), \(\pi\)-almost everywhere.
\end{definition}

We refer to~\cite{zhang2026algorithmic} for practical approaches to achieve such warm starts. The theorem concerns the stationary kernel after warmup: the mass matrix, step size \(h\), and maximum depth \(k_{\max}\) are fixed.  Under \eqref{eq:strong-convexity}, \(\pi\) satisfies the Cheeger inequality
\begin{align}
\pi(S_3)\ge \psi\,\dist(S_1,S_2)\pi(S_1)\pi(S_2),
\qquad
\psi=(\log2)\sqrt m,
\label{eq:cheeger}
\end{align}
for every measurable partition \(\R^d=S_1\sqcup S_2\sqcup S_3\); see \cite{BakryGentilLedoux2014}.

\subsection{The stationary profile and the sufficient certificate}

Let \((X_t,V_t)\) solve exact Hamiltonian dynamics
\begin{align*}
\dot X_t=V_t,
\qquad
\dot V_t=-\nabla U(X_t),
\end{align*}
with \((X_0,V_0)\sim\pi\otimes N(0,I_d)\).  Define
\begin{align*}
D_t^-=V_0^\top(X_t-X_0),
\qquad
D_t^+=V_t^\top(X_t-X_0),
\end{align*}
\begin{align}
R(t)=\E\norm{X_t-X_0}^2,
\qquad
u(t)=\E D_t^-=\E D_t^+.
\label{eq:profile-def}
\end{align}
The equality of the means and the identity \(u=R'/2\) are proved in \Cref{lem:profile-representations}.

For a leapfrog step size \(h\), the physical duration of a dyadic orbit of cardinality \(2^k\) is
\begin{align*}
\tau_k=h(2^k-1),
\qquad k\in\Nzero.
\end{align*}

\begin{definition}[]
\label{def:profile-separated}
Fix \(\rho_0>0\), a dimensionless horizon \(S>0\), and a short-time cutoff \(T_{\mathrm{sm}}>0\).  A depth \(k_\star\ge2\) is \((\rho_0,S,T_{\mathrm{sm}})\)-profile separated if
\begin{align}
\tau_{k_\star}\le\frac S{\sqrt m},
\qquad
u(\tau_{k_\star})\le-\rho_0\frac d{\sqrt m},
\label{eq:profile-terminal-margin}
\end{align}
and
\begin{align}
u(\tau_r)\ge\rho_0\frac d{\sqrt m}
\quad
\text{for every }1\le r<k_\star
\text{ with }\tau_r\ge T_{\mathrm{sm}}.
\label{eq:profile-preterminal-margin}
\end{align}
\end{definition}

Equivalently, define the normalized dyadic margin
\begin{align*}
\mathfrak m_h(k)
=
\frac{\sqrt m}{d}
\min\!\left\{
-u(\tau_k),
\min_{\substack{1\le r<k\\ \tau_r\ge T_{\mathrm{sm}}}}
u(\tau_r)
\right\},
\end{align*}
with the inner minimum interpreted as \(+\infty\) when the index set is empty.  Then \(k\) is profile separated with margin \(\rho_0\) exactly when \(\tau_k\le S/\sqrt m\) and \(\mathfrak m_h(k)\ge\rho_0\).

The cutoff is needed because \(u(t)=dt+O(t^3)\) near the origin and therefore cannot have a fixed \(d/\sqrt m\) margin at time \(h\).  Shorter intervals are certified by a deterministic no-U-turn estimate.  At longer times, profile separation is a finite-grid nonresonance condition.

We emphasize that profile separation is a sufficient abstract condition used in the proof.  We do not show that it is necessary for rapid mixing, necessary for NUTS to stop before the cap, or necessary for the random stopping depth to concentrate.  In particular, the condition can even fail when a dyadic time lies close to a profile zero even though the realized random diagnostics still select useful trajectories.  

\subsection{Intrinsic diagnostic stability}
\label{sec:intrinsic-stability}

The profile is a deterministic stationary-mean object.  To transfer its signs to the practical leapfrog tree, we
measure only the scalar quantities that enter the NUTS stopping and selection rules.  Let
\(Z_0=(X_0,V_0)\sim\bar\pi\), let
\[
Z_t=(X_t,V_t)=\flow_tZ_0,
\qquad
\widehat Z_j=(\widehat X_j,\widehat V_j)=\Phi_h^jZ_0,
\]
and, for an integer interval \(I=[a:b]\cap\Z\), define
\begin{align*}
D_I^-
&=V_{ah}^{\top}(X_{bh}-X_{ah}),
&
D_I^+
&=V_{bh}^{\top}(X_{bh}-X_{ah}),
\\
\widehat D_I^-
&=\widehat V_a^{\top}(\widehat X_b-\widehat X_a),
&
\widehat D_I^+
&=\widehat V_b^{\top}(\widehat X_b-\widehat X_a).
\end{align*}
For an integer interval \(I=[a:b]\cap\Z\), write
\[
  \# I=b-a+1,
  \qquad
  \ell_h(I)=h(b-a)=h(\# I-1)
\]
for its cardinality and physical span, respectively.  In particular, a singleton has
\(\ell_h(I)=0\). For a power of two \(K\), let \(\mathscr I_K\) denote the family of every completed orbit
and every recursively inspected dyadic suborbit that can occur under some realization of the
doubling bits before terminal cardinality \(K\).  Since every endpoint lies in
\([-K+1:K-1]\),
\begin{align*}
N_K:=|\mathscr I_K|\le(2K-1)^2\le4K^2.
\end{align*}

\begin{definition}[]
\label{def:diagnostic-concentration-modulus}
For \(p\ge2\) and \(S>0\), define
\begin{align}
\mathfrak C_p(S;U)
=
\max_{\sigma\in\{-,+\}}
\sup_{0\le t\le S/\sqrt m}
\frac{
\sqrt m\,\norm{D_t^\sigma-u(t)}_{L^p(\bar\pi)}
}{\sqrt{dp}+p}.
\label{eq:intrinsic-Cp}
\end{align}
\end{definition}

For \(\eta>0\), let
\begin{align*}
\mathcal H_{h,K}(\eta)
=
\left\{
\max_{|j|\le K}
\abs{H(\Phi_h^jZ_0)-H(Z_0)}\le\eta
\right\}.
\end{align*}

\begin{definition}[]
\label{def:numerical-fidelity-modulus}
For \(p\ge2\), define
\begin{align}
\mathfrak E_{p,h}(K,\eta;U)
=
\frac{\sqrt m}{d}
\left\|
\max_{I\in\mathscr I_K}
\max_{\sigma\in\{-,+\}}
\abs{\widehat D_I^\sigma-D_I^\sigma}
\,\1_{\mathcal H_{h,K}(\eta)}
\right\|_{L^p(\bar\pi)}.
\label{eq:intrinsic-Ep}
\end{align}
Also put
\begin{align}
\delta_H(h,K,\eta)
&=\bar\pi\bigl(\mathcal H_{h,K}(\eta)^c\bigr),
\label{eq:intrinsic-delta-H}\\
\delta_{\mathrm{sm}}(h,K,T_{\mathrm{sm}})
&=
\bar\pi\!\left(
\begin{array}{c}
\exists I\in\mathscr I_K:\ \# I\ge2,\ \ell_h(I)<T_{\mathrm{sm}},
\min\{\widehat D_I^-,\widehat D_I^+\}\le0
\end{array}
\right).
\label{eq:intrinsic-delta-sm}
\end{align}
\end{definition}

Singleton intervals are excluded from \(\delta_{\mathrm{sm}}\): their displacement is zero and their U-turn indicator is automatically zero, so strict positivity is neither true nor needed.

The first modulus measures fluctuations of the exact scalar stopping statistic around its
population profile.  The second measures only the numerical error of that statistic, not the
operator norm of the full tangent flow.  The two failure probabilities isolate the other two
facts required by the proof: comparable numerical Boltzmann weights over the completed orbit
and positivity of the very short dyadic intervals for which a fixed \(d/\sqrt m\) profile
margin is unavailable.

For a candidate separated depth \(k_\star\), write
\begin{align*}
K_\star=2^{k_\star},
\qquad
T_\star=(K_\star-1)h,
\qquad
a_\star=\sqrt m\,T_\star.
\end{align*}
For \(p\ge2\), set
\begin{align}
\Xi_{p,h}(S,K_\star)
=
e\,\mathfrak C_p(S;U)
\left(\sqrt{\frac pd}+\frac pd\right)
+e\,\mathfrak E_{p,h}(K_\star,\log2;U),
\label{eq:intrinsic-Xi}
\end{align}
\begin{align*}
\zeta_{p,h}(K_\star)
=(2N_{K_\star}+1)e^{-p}
+\delta_H(h,K_\star,\log2)
+\delta_{\mathrm{sm}}(h,K_\star,T_{\mathrm{sm}}).
\end{align*}

Let
\begin{align}
t_{\mathrm{ov}}
=
\frac{c_{\mathrm{ov}}}{\sqrt L(1+\gamma)^{1/3}},
\qquad
r_\star=\frac{t_{\mathrm{ov}}}{K_\star h},
\label{eq:intrinsic-overlap-time}
\end{align}
where \(c_{\mathrm{ov}}>0\) is the universal constant in the fixed-index overlap estimate.
After decreasing \(c_{\mathrm{ov}}\), profile negativity at \(T_\star\) and the universal
initial positivity in \eqref{eq:profile-positive} imply \(r_\star\le1/4\).  Define
\begin{align*}
\underline a_{\mathrm M}=c_1r_\star,
\qquad
\underline a_{\mathrm B}=c_2r_\star^2,
\qquad
\theta_\star
=
\min\left\{1,\frac{\psi t_{\mathrm{ov}}}{128}\right\},
\end{align*}
where \(c_1,c_2>0\) are universal selector-mass constants.

For \(a\in\{\mathrm M,\mathrm B\}\), \(n_a\) denotes the number of
applications of the NUTS position kernel \(P_a\), and
\(N_{\nabla U}^{(a)}\) denotes the total number of density and gradient
evaluations used by those transitions.  Let
\[
  K_{\mathrm{cap}}=2^{k_{\max}}.
\]
Every transition uses at most \(C_{\mathrm{cost}}K_{\mathrm{cap}}\)
evaluations, while a transition on the certification event of
\Cref{prop:intrinsic-certificate} uses at most
\(C_{\mathrm{cost}}K_\star\).  Thus the unconditional deterministic
bound is
\[
  N_{\nabla U}^{(a)}\le
  C_{\mathrm{cost}}K_{\mathrm{cap}}n_a,
\]
and \Cref{prop:exceptional-work} gives refined expected and
high-probability bounds.  Separately,
\(N_{K_\star}=|\mathscr I_{K_\star}|\) denotes the number of possible
completed or recursively inspected intervals and is used only in the
simultaneous probability bounds.

\begin{theorem}[Profile separation and diagnostic stability imply genuine U-turn mixing]
\label{thm:adaptive}
Suppose \Cref{ass:convex-smooth,ass:hessian-lip} hold, let \(\nu\) be \(M\)-warm, and
let \(\epsilon\in(0,1/2)\).  Suppose that \(k_\star<k_{\max}\) is
\((\rho_0,S,T_{\mathrm{sm}})\)-profile separated and that
\begin{align}
\Xi_{p,h}(S,K_\star)\le\frac{\rho_0}{4}
\label{eq:intrinsic-margin-condition}
\end{align}
for some \(p\ge2\).  Assume the fixed-time resolution condition
\begin{align}
\frac{t_{\mathrm{ov}}}{h}\ge8.
\label{eq:intrinsic-resolution}
\end{align}
For a selector \(a\in\{\mathrm M,\mathrm B\}\), suppose
\begin{align}
\zeta_{p,h}(K_\star)
\le
c_0\,\underline a_a\,\theta_\star
\frac{\epsilon^2}{M^2},
\label{eq:intrinsic-failure-condition}
\end{align}
where \(c_0>0\) is a sufficiently small universal constant.

Set
\[
  s=\frac{\epsilon^2}{128M^2},
  \qquad
  \delta_a=\frac{3\underline a_a}{32}.
\]
Then there are measurable sets \(D_a\subset\R^d\) and
\(G_{a,x}\subset\R^d\), \(x\in D_a\), satisfying
\[
  \sup_{x\in D_a}\Pp_V(G_{a,x}^c)\le\delta_a,
  \qquad
  \pi(D_a^c)
  \le
  \frac{\zeta_{p,h}(K_\star)}{\delta_a}
  \le
  \frac{\theta_\star s}{64}.
\]
Every \((x,v)\) with \(x\in D_a\) and \(v\in G_{a,x}\) has the following properties:

\begin{enumerate}[label=(\roman*)]
\item every nontrivial completed orbit and recursively inspected proper suborbit below depth
\(k_\star\) has both numerical endpoint diagnostics strictly positive; singleton intervals have U-turn indicator zero;
\item every possible completed orbit at depth \(k_\star\) has both numerical endpoint
diagnostics strictly negative;
\item the numerical Hamiltonian differs from its initial value by at most \(\log2\) at
every point of the terminal orbit.
\end{enumerate}
Consequently, for every such certified phase point \((x,v)\), every realization of the left--right doubling decisions stops at cardinality
\(K_\star\) through the U-turn rule, and the maximum-depth cap is inactive.

For multinomial NUTS,
\begin{align}
n_{\mathrm M}
\le
C\left[1+a_\star^2\kappa^2(1+\gamma)^{4/3}\right]
\log\!\left(\frac{CM}{\epsilon}\right),
\label{eq:intrinsic-transitions-M}
\end{align}
and for biased-progressive NUTS,
\begin{align*}
n_{\mathrm B}
\le
C\left[1+a_\star^4\kappa^3(1+\gamma)^2\right]
\log\!\left(\frac{CM}{\epsilon}\right).
\end{align*}
For every selector satisfying \eqref{eq:intrinsic-failure-condition},
\begin{align}
\norm{\nu P_a^{n_a}-\pi}_{\mathrm{TV}}\le\epsilon.
\label{eq:intrinsic-tv}
\end{align}
On the certification event, the returned orbit has cardinality
\begin{align}
K_\star=1+\frac{T_\star}{h},
\label{eq:intrinsic-K}
\end{align}
and one transition uses at most
\(C_{\mathrm{cost}}K_\star\) density and gradient evaluations.  Without
an additional restriction on the maximum depth, the unconditional
deterministic work bound at the mixing time is only
\begin{align}
N_{\nabla U}^{(a)}
\le
C_{\mathrm{cost}}K_{\mathrm{cap}}n_a,
\qquad
K_{\mathrm{cap}}=2^{k_{\max}}.
\label{eq:intrinsic-work-cap}
\end{align}
Refined expected and high-probability work bounds are stated in
\Cref{prop:exceptional-work}.
All constants denoted by \(C,c_0,c_1,c_2,c_{\mathrm{ov}}\) are universal. The evaluation-cost constant \(C_{\mathrm{cost}}\) is also universal.
\end{theorem}

\begingroup\color{black}
\begin{proposition}[]
\label{prop:exceptional-work}
Let
\[
  K_{\mathrm{cap}}=2^{k_{\max}},
  \qquad
  \zeta_\star=\zeta_{p,h}(K_\star),
  \qquad
  q_\star=\min\{1,M\zeta_\star\}.
\]
For \(a\in\{\mathrm M,\mathrm B\}\), let
\(\mathsf W_n^{(a)}\) be the total number of density and gradient
evaluations used by the first \(n\) applications of \(P_a\), starting
from an \(M\)-warm law.  Then
\begin{align}
  \E_\nu \mathsf W_n^{(a)}
  \le
  C_{\mathrm{cost}}n
  \left[
    K_\star+(K_{\mathrm{cap}}-K_\star)q_\star
  \right].
  \label{eq:expected-work-exceptional}
\end{align}
For every \(\delta\in(0,1)\), with probability at least \(1-\delta\),
\begin{align}
  \mathsf W_n^{(a)}
  \le
  C_{\mathrm{cost}}
  \left[
    nK_\star+
    (K_{\mathrm{cap}}-K_\star)
    \min\left\{n,\frac{nq_\star}{\delta}\right\}
  \right].
  \label{eq:high-probability-work-exceptional}
\end{align}
Moreover,
\begin{align}
  \Pp_\nu\!\left(
    \mathsf W_n^{(a)}
    \le
    C_{\mathrm{cost}}nK_\star
  \right)
  \ge
  1-nq_\star.
  \label{eq:all-transitions-certified}
\end{align}
If, in addition,
\begin{align}
  K_{\mathrm{cap}}
  \le
  C_{\mathrm{cap}}K_\star,
  \label{eq:cap-comparability}
\end{align}
then the deterministic bound
\begin{align}
  \mathsf W_n^{(a)}
  \le
  C_{\mathrm{cost}}C_{\mathrm{cap}}K_\star n
  \label{eq:deterministic-cap-work}
\end{align}
holds for every realization.  In particular, the expected work is of
order \(K_\star n\) whenever
\[
  \left(\frac{K_{\mathrm{cap}}}{K_\star}-1\right)q_\star=O(1).
\]
\end{proposition}

\begin{proof}
Let \(\mu_t=\nu P_a^t\).  By the \(\pi\)-invariance of \(P_a\), proved
in \Cref{sec:spectral}, warmness is preserved: if
\(\mu_t\le M\pi\), then for every measurable set \(A\),
\[
  \mu_tP_a(A)
  =
  \int P_a(x,A)\,\mu_t(\dd x)
  \le
  M\int P_a(x,A)\,\pi(\dd x)
  =
  M\pi(A).
\]
After refreshing the Gaussian momentum, the phase-space law at
transition \(t\) is therefore dominated by \(M\bar\pi\).  If
\(\mathcal C_{\star,t}\) denotes the certification event at that
transition, then \Cref{prop:intrinsic-certificate} gives
\[
  \Pp_\nu(\mathcal C_{\star,t}^c)
  \le
  \min\{1,M\bar\pi(\mathcal C_\star^c)\}
  \le
  q_\star.
\]
Let
\[
  B_n
  =
  \sum_{t=0}^{n-1}
  \mathbf 1_{\mathcal C_{\star,t}^c}
\]
be the number of uncertified transitions.  A certified transition uses
at most \(C_{\mathrm{cost}}K_\star\) evaluations, while every transition
uses at most \(C_{\mathrm{cost}}K_{\mathrm{cap}}\).  Hence
\[
  \mathsf W_n^{(a)}
  \le
  C_{\mathrm{cost}}
  \left[
    nK_\star+(K_{\mathrm{cap}}-K_\star)B_n
  \right].
\]
Since \(\E_\nu B_n\le nq_\star\), this proves
\eqref{eq:expected-work-exceptional}.  Markov's inequality gives
\[
  \Pp_\nu\!\left(
    B_n>\frac{nq_\star}{\delta}
  \right)
  \le
  \delta,
\]
and the trivial inequality \(B_n\le n\) proves
\eqref{eq:high-probability-work-exceptional}.  A union bound gives
\(\Pp_\nu(B_n\ge1)\le nq_\star\), proving
\eqref{eq:all-transitions-certified}.  Finally,
\eqref{eq:deterministic-cap-work} follows directly from
\eqref{eq:cap-comparability}.
\end{proof}
\endgroup

\begingroup\color{black}
The result shows that, whenever the stationary profile is separated
from zero by more than the combined stochastic and numerical errors,
the practical NUTS tree behaves predictably on the certification event:
every doubling realization reaches the same depth, stops through the
U-turn criterion before the maximum-depth cap, and maintains comparable
numerical energies along the selected orbit.  The resulting multinomial
and biased-progressive kernels mix globally from an \(M\)-warm start in
\(n_{\mathrm M}\) and \(n_{\mathrm B}\) transitions, respectively.  A
certified transition costs order \(K_\star\), while the cost accumulated
by the actual chain is governed by \Cref{prop:exceptional-work}; order
\(K_\star n_a\) is deterministic under cap comparability and otherwise
holds in the stated expected or high-probability sense.  The different
transition rates arise from the different amounts of probability that
the two state-selection rules assign to the short fixed-time proposal
band used in the conductance argument.
\endgroup

We also emphasize that the result does not assume a polynomial bound on the full Hamiltonian Jacobian.  Its
hypotheses are the scalar concentration and fidelity bounds in
\eqref{eq:intrinsic-Cp}--\eqref{eq:intrinsic-Ep}, together with the two probabilities in
\eqref{eq:intrinsic-delta-H}--\eqref{eq:intrinsic-delta-sm}.  A tangent-flow estimate is one
way to verify these quantities.  The failure requirement depends on the selector because the BPS double-tent law puts
quadratically less mass than the multinomial triangular law near a short fixed-time band.

\paragraph{Proof strategy.}
We first prove that a reversible positive-semidefinite kernel mixes from
an \(M\)-warm start once its restricted conductance is bounded below (\Cref{sec:spectral}),
and apply this result to $R_{\mathrm M}=P_{\mathrm M},$ and $R_{\mathrm B}=P_{\mathrm B}^2.$ It therefore remains to establish conductance.  Profile separation,
diagnostic concentration, leapfrog fidelity, and whole-orbit energy
control imply that \textcolor{black}{on the certification event, every doubling realization terminates at the same
certified orbit size \(K_\star\)}.  On this orbit, each NUTS kernel
minorizes an explicit mixture of fixed-index leapfrog proposals.
Local overlap of those proposals yields local overlap of the NUTS
transition laws, and the Cheeger inequality converts this local overlap
into a restricted-conductance lower bound, completing the mixing proof. These steps are executed in \cref{sec:terminal}-\cref{sec:profile-proof}.

\section{Spectral structure and reversibility}
\label{sec:spectral}

The two selectors require different spectral arguments.  Multinomial re-rooting is an augmented-space projection, which proves one-step positivity.  BPS uses a reversible finite-orbit transition that is not generally a projection; the unchanged two-step position skeleton supplies positivity instead.

Let $\mathsf Z=\R^d\times\R^d,$ $z=(x,v)$ and let \(\bar\pi\) be the phase-space law defined in \cref{sec:common-kernel}.  Let \(\mathscr I\) denote the countable collection of finite consecutive subsets of \(\Z\) that contain \(0\).  Recall that \(O_z(I)\) is the probability that random doubling started at \(z\) returns \(I\in\mathscr I\).

Define a probability measure on the augmented space
\(\mathsf Z\times\mathscr I\) by
\begin{align*}
  \eta(\dd z,\dd I)
  =
  \bar\pi(\dd z)\,O_z(I),
\end{align*}
where counting measure is understood in the orbit coordinate.  For \(i\in I\), define
\begin{align*}
  q_i(z,I)
  =
  \frac{\exp[-H(\Phi_h^iz)]}
  {\sum_{r\in I}\exp[-H(\Phi_h^rz)]}.
\end{align*}
For \(g\in L^2(\eta)\), define the re-rooting operator
\begin{align*}
  (\mathcal Rg)(z,I)
  =
  \sum_{i\in I}
  q_i(z,I)\,
  g(\Phi_h^iz,I-i).
\end{align*}

\begin{proposition}[Orbit re-rooting is an orthogonal projection]
\label{prop:R-projection}
The operator \(\mathcal R\) satisfies
\[
  \mathcal R^2=\mathcal R,
  \qquad
  \mathcal R^*=\mathcal R
\]
on \(L^2(\eta)\).  Consequently,
\[
  \ip{g}{\mathcal Rg}_{L^2(\eta)}
  =
  \norm{\mathcal Rg}_{L^2(\eta)}^2
  \ge0.
\]
\end{proposition}

\begin{proof}
We first record the measure identity associated with re-rooting.  Fix
\(I\in\mathscr I\) and \(i\in I\), and set
\[
  z'=\Phi_h^iz,
  \qquad
  I'=I-i.
\]
The map \(z\mapsto z'\) is volume preserving.  By
\cref{prop:initial-point-symmetry},
\[
  O_{z'}(I')=O_z(I).
\]
Writing \(\bar\pi(z)=\bar Z^{-1}e^{-H(z)}\) for the phase-space density, a
direct calculation gives
\begin{align*}
  \bar\pi(z)q_i(z,I)
  &=
  \frac{\bar Z^{-1}e^{-H(z)}e^{-H(z')}}
  {\sum_{r\in I}e^{-H(\Phi_h^rz)}}
  \notag\\
  &=
  \bar\pi(z')
  \frac{e^{-H(z)}}
  {\sum_{r'\in I'}e^{-H(\Phi_h^{r'}z')}}
  \notag\\
  &=
  \bar\pi(z')q_{-i}(z',I').
\end{align*}
Consequently,
\begin{align}
  \bar\pi(\dd z)O_z(I)q_i(z,I)
  =
  \bar\pi(\dd z')O_{z'}(I')q_{-i}(z',I').
  \label{eq:reroot-measure-identity}
\end{align}

We next verify that \(\mathcal R\) is a bounded operator on \(L^2(\eta)\).
For a bounded measurable \(g\), Jensen's inequality gives
\[
  |(\mathcal Rg)(z,I)|^2
  \le
  \sum_{i\in I}q_i(z,I)
  |g(\Phi_h^iz,I-i)|^2.
\]
Integrating and using \eqref{eq:reroot-measure-identity} together with the
bijection $(I,i,z)\longleftrightarrow(I-i,-i,\Phi_h^iz)$ yields
\begin{align*}
  \norm{\mathcal Rg}_{L^2(\eta)}^2
  &\le
  \sum_{I\in\mathscr I}\sum_{i\in I}
  \int
  |g(\Phi_h^iz,I-i)|^2
  \bar\pi(\dd z)O_z(I)q_i(z,I)
  \notag\\
  &=
  \sum_{I'\in\mathscr I}\sum_{r\in I'}
  \int
  |g(z',I')|^2
  \bar\pi(\dd z')O_{z'}(I')q_r(z',I')
  \notag\\
  &=
  \norm{g}_{L^2(\eta)}^2.
\end{align*}
Thus \(\mathcal R\) extends uniquely from bounded measurable functions to a
contraction on all of \(L^2(\eta)\).  It is therefore enough to prove
idempotence and self-adjointness first for bounded functions; continuity then
extends both identities to \(L^2(\eta)\).

For idempotence, applying \(\mathcal R\) twice gives
\begin{align}
  (\mathcal R^2g)(z,I)
  &=
  \sum_{i\in I}q_i(z,I)
  \sum_{k\in I-i}
  q_k(\Phi_h^iz,I-i)
  g(\Phi_h^{i+k}z,I-i-k).
  \label{eq:R2-expand}
\end{align}
Fix \(i\in I\) and set \(r=i+k\).  As \(k\) ranges through \(I-i\), the index
\(r\) ranges through \(I\), and \(I-i-k=I-r\).  Moreover,
\begin{align*}
  q_k(\Phi_h^iz,I-i)=
  \frac{e^{-H(\Phi_h^{i+k}z)}}
  {\sum_{m\in I-i}e^{-H(\Phi_h^{m+i}z)}}=
  \frac{e^{-H(\Phi_h^rz)}}
  {\sum_{s\in I}e^{-H(\Phi_h^sz)}}
  =q_r(z,I).
\end{align*}
Hence the inner sum in \eqref{eq:R2-expand} is $
  \sum_{r\in I}q_r(z,I)g(\Phi_h^rz,I-r)
  =
  (\mathcal Rg)(z,I),$ independently of \(i\).  Since \(\sum_{i\in I}q_i(z,I)=1\), \textcolor{black}{we have}
\begin{align}
  \mathcal R^2g=\mathcal Rg.
  \label{eq:R-idempotent}
\end{align}

For self-adjointness, let \(f,g\) be bounded.  By definition,
\begin{align*}
  \ip{f}{\mathcal Rg}_{L^2(\eta)}
  &=
  \sum_{I\in\mathscr I}\sum_{i\in I}
  \int
  f(z,I)g(\Phi_h^iz,I-i)
  \bar\pi(\dd z)O_z(I)q_i(z,I).
\end{align*}
Apply \eqref{eq:reroot-measure-identity} and the same bijection of triples.
With \(r=-i\), the right-hand side becomes
\begin{align*}
  \sum_{I'\in\mathscr I}\sum_{r\in I'}
  \int
  f(\Phi_h^rz',I'-r)g(z',I')
  \bar\pi(\dd z')O_{z'}(I')q_r(z',I')
  &=
  \ip{\mathcal Rf}{g}_{L^2(\eta)}.
\end{align*}
Thus \(\mathcal R^*=\mathcal R\).  Together with
\eqref{eq:R-idempotent}, this makes \(\mathcal R\) an orthogonal projection.
Therefore
\[
  \ip{g}{\mathcal Rg}_{L^2(\eta)}
  =
  \ip{\mathcal Rg}{\mathcal Rg}_{L^2(\eta)}
  =
  \norm{\mathcal Rg}_{L^2(\eta)}^2,
\]
as claimed.
\end{proof}

We now compress this projection to the phase and position spaces.  Define the isometry
\[
  \mathcal A:L^2(\bar\pi)\longrightarrow L^2(\eta),
  \qquad
  (\mathcal Af)(z,I)=f(z).
\]
Indeed,
\[
  \norm{\mathcal Af}_{L^2(\eta)}^2
  =
  \int_{\mathsf Z}
  |f(z)|^2
  \left(\sum_{I\in\mathscr I}O_z(I)\right)
  \bar\pi(\dd z)
  =
  \norm{f}_{L^2(\bar\pi)}^2.
\]
The phase-space NUTS transition operator is
\begin{align}
  \mathcal K=\mathcal A^*\mathcal R\mathcal A.
  \label{eq:K-compression}
\end{align}
To verify \eqref{eq:K-compression}, observe that
\begin{align*}
  (\mathcal A^*\mathcal R\mathcal Af)(z)
  &=
  \sum_{I\in\mathscr I}O_z(I)
  \sum_{i\in I}q_i(z,I)f(\Phi_h^iz),
\end{align*}
which is exactly orbit generation followed by multinomial re-rooting.

Define another isometry
\[
  \mathcal B:L^2(\pi)\longrightarrow L^2(\bar\pi),
  \qquad
  (\mathcal Bf)(x,v)=f(x).
\]
Its adjoint integrates the momentum: $ (\mathcal B^*g)(x)
  =
  \int_{\R^d}g(x,v)\varphi(\dd v).$ Momentum refreshment, the phase-space transition, and projection onto position therefore give
\begin{align}
  \PM
  =
  \mathcal B^*\mathcal K\mathcal B
  =
  \mathcal B^*\mathcal A^*\mathcal R\mathcal A\mathcal B.
  \label{eq:P-compression}
\end{align}

\begin{corollary}[Positivity]
\label{cor:PSD}
For every \(f\in L^2(\pi)\),
\begin{align}
  \ip{f}{\PM f}_{L^2(\pi)}
  =
  \norm{\mathcal R\mathcal A\mathcal Bf}_{L^2(\eta)}^2
  \ge0.
  \label{eq:PSD}
\end{align}
Thus \(\PM\) is a self-adjoint positive-semidefinite contraction, and its spectrum is contained in \([0,1]\).
\end{corollary}

\begin{proof}
Equation \eqref{eq:P-compression} and
\(\mathcal R=\mathcal R^*=\mathcal R^2\) give
\[
  \ip{f}{\PM f}
  =
  \ip{\mathcal A\mathcal Bf}
      {\mathcal R\mathcal A\mathcal Bf}
  =
  \ip{\mathcal R\mathcal A\mathcal Bf}
      {\mathcal R\mathcal A\mathcal Bf}.
\]
Reversibility follows from self-adjointness.  Every Markov operator is an
\(L^2(\pi)\) contraction, so its self-adjoint spectrum lies in
\([-1,1]\); \eqref{eq:PSD} excludes the negative part.
\end{proof}

\subsection{Reversibility and a positive two-step skeleton for NUTS-BPS}
\label{sec:bps-spectral}

For a returned orbit \(I\), let $ q_{\mathrm B,z}^{I}(j),$ for $j\in I,$ be the probability that the recursion \eqref{eq:bps-recursion}, initialized at root zero, selects index \(j\).

\begingroup\color{black}
We now define the same finite selector from an arbitrary physical root.  Let
\(I\) have cardinality \(2^R\), and let \(\mathcal T_I\) be its canonical binary
bisection tree: the root is \(I\), and every non-singleton node is split into
its left and right halves.  For \(a\in I\), let
\[
  I_0^{(a)}=\{a\}\subset I_1^{(a)}\subset\cdots\subset I_R^{(a)}=I
\]
be the unique ancestor chain of the leaf \(a\) in \(\mathcal T_I\), and put
\(N_r^{(a)}=I_{r+1}^{(a)}\setminus I_r^{(a)}\).  Starting from candidate
\(C_0=a\), apply the BPS update \eqref{eq:bps-recursion} successively to the
old/new pairs \((I_r^{(a)},N_r^{(a)})\).  We denote the resulting transition
probability from physical root \(a\) to physical root \(b\) by
\(\bar q_{\mathrm B,z}^{I}(a,b)\).  When \(a=0\), the accepted-doubling
history of Algorithm~\ref{alg:bps} is exactly this ancestor chain, so
\(q_{\mathrm B,z}^{I}(j)=\bar q_{\mathrm B,z}^{I}(0,j)\).  This construction
is the finite rooted-tree kernel to which the detailed-balance result quoted
below applies.
\endgroup

\begin{proposition}[Finite-orbit BPS detailed balance]
\label{prop:bps-finite-balance}
For every returned weighted dyadic orbit \(I\), the finite BPS root kernel is reversible with respect to the normalized weights
\[
  \omega_z^I(a)
  =
  \frac{e^{-H(\Phi_h^az)}}{W_z(I)},
  \qquad a\in I.
\]
Equivalently,
\begin{align}
  e^{-H(\Phi_h^az)}
  \bar q_{\mathrm B,z}^{I}(a,b)
  =
  e^{-H(\Phi_h^bz)}
  \bar q_{\mathrm B,z}^{I}(b,a),
  \qquad a,b\in I.
  \label{eq:bps-finite-detailed-balance}
\end{align}
In particular, for \(j\in I\),
\begin{align}
  e^{-H(z)}q_{\mathrm B,z}^{I}(j)
  =
  e^{-H(\Phi_h^jz)}
  q_{\mathrm B,\Phi_h^jz}^{I-j}(-j).
  \label{eq:bps-reroot-balance}
\end{align}
\end{proposition}

\begin{proof}
The detailed-balance identity \eqref{eq:bps-finite-detailed-balance} is Proposition~6 of Durmus et al. \cite{DurmusEtAl2024}, proved there by an explicit binary-tree representation of the progressive recursion.  We verify the implication \eqref{eq:bps-reroot-balance}, which is the form needed below.

The law \(q_{\mathrm B,z}^{I}(j)\) is the row of the finite root kernel that starts from physical root zero: $q_{\mathrm B,z}^{I}(j)
  =\bar q_{\mathrm B,z}^{I}(0,j).$ Re-root the same physical orbit at \(z_j=\Phi_h^jz\).  Its index set is \(I-j\).  Starting from index zero in the re-rooted representation is the same as starting from physical root \(j\) in the original representation, while selecting index \(-j\) is the same as selecting physical root zero.  Therefore $q_{\mathrm B,z_j}^{I-j}(-j)
  =\bar q_{\mathrm B,z}^{I}(j,0).$ Substitute \(a=0\) and \(b=j\) into \eqref{eq:bps-finite-detailed-balance} to obtain \eqref{eq:bps-reroot-balance}.
\end{proof}

Define the BPS phase-space operator
\begin{align}
  (\mathcal K_{\mathrm B}f)(z)
  =
  \sum_{I\in\mathscr I}O_z(I)
  \sum_{j\in I}q_{\mathrm B,z}^{I}(j)f(\Phi_h^jz).
  \label{eq:bps-phase-operator}
\end{align}

\begin{proposition}[\textcolor{black}{Reversibility of NUTS-BPS}]
\label{prop:bps-reversible}
The phase-space operator \(\mathcal K_{\mathrm B}\) is self-adjoint on \(L^2(\bar\pi)\).  Consequently,
\begin{align}
  \PB=\mathcal B^*\mathcal K_{\mathrm B}\mathcal B
  \label{eq:bps-position-compression}
\end{align}
is self-adjoint on \(L^2(\pi)\), and \(\pi\) is invariant for \(\PB\).
\end{proposition}

\begin{proof}
Let \(f,g\) be bounded measurable functions on phase space.  From \eqref{eq:bps-phase-operator},
\begin{align}
  \ip{f}{\mathcal K_{\mathrm B}g}_{L^2(\bar\pi)}
  &=
  \sum_{I\in\mathscr I}\sum_{j\in I}
  \int
  f(z)g(\Phi_h^jz)
  O_z(I)q_{\mathrm B,z}^{I}(j)
  \bar\pi(\dd z).
  \label{eq:bps-self-adjoint-start}
\end{align}
For a fixed \((I,j)\), set
\[
  z'=\Phi_h^jz,
  \qquad I'=I-j,
  \qquad j'=-j.
\]
Leapfrog volume preservation gives \(\dd z'=\dd z\).  By \cref{prop:initial-point-symmetry}, $O_{z'}(I')=O_z(I).$ By \eqref{eq:bps-reroot-balance},
\begin{align*}
  \bar\pi(\dd z)q_{\mathrm B,z}^{I}(j)
  =
  \bar\pi(\dd z')q_{\mathrm B,z'}^{I'}(j').
\end{align*}
The map \((z,I,j)\mapsto(z',I',j')\) is an involution.  Applying it to \eqref{eq:bps-self-adjoint-start} yields
\begin{align*}
  \ip{f}{\mathcal K_{\mathrm B}g}_{L^2(\bar\pi)}
  &=
  \sum_{I'\in\mathscr I}\sum_{j'\in I'}
  \int
  f(\Phi_h^{j'}z')g(z')
  O_{z'}(I')q_{\mathrm B,z'}^{I'}(j')
  \bar\pi(\dd z')\\
  &=
  \ip{\mathcal K_{\mathrm B}f}{g}_{L^2(\bar\pi)}.
\end{align*}
Thus \(\mathcal K_{\mathrm B}\) is self-adjoint.  The identity \eqref{eq:bps-position-compression} follows from Gaussian momentum refreshment and projection exactly as for multinomial NUTS.  Compression by \(\mathcal B\) preserves self-adjointness.  Finally, a self-adjoint Markov operator fixes constants, so for every bounded \(f\),
\[
  \int \PB f\,\dd\pi
  =\ip{1}{\PB f}_{\pi}
  =\ip{\PB1}{f}_{\pi}
  =\ip{1}{f}_{\pi},
\]
which proves invariance.
\end{proof}

The multinomial projection argument does not extend to the BPS selector.  The obstruction appears on the smallest possible finite orbit.

\begin{example}[A negative finite-orbit BPS eigenvalue]
\label{ex:bps-negative}
Consider a two-point orbit with equal Boltzmann weights.  From either root, the old set is the singleton containing that root and the new set is the other singleton.  The replacement probability \eqref{eq:bps-acceptance} equals one.  The finite BPS root transition is therefore
\begin{align*}
  \begin{pmatrix}0&1\\1&0\end{pmatrix},
\end{align*}
whose eigenvalues are \(1\) and \(-1\).  Thus the finite BPS re-rooting operator is reversible but not positive semidefinite.  This example does not assert that every position-space BPS kernel has negative spectrum; it shows that the multinomial orthogonal-projection proof cannot be transferred to BPS.
\end{example}

\begin{proposition}[Positive two-step BPS skeleton]
\label{prop:bps-square-PSD}
The unchanged two-step skeleton $  \RB=\PB^2$ is reversible and positive semidefinite on \(L^2(\pi)\).  More precisely, for every \(f\in L^2(\pi)\), $\ip{f}{\RB f}_{\pi}
  =\norm{\PB f}_{L^2(\pi)}^2
\ge0.$
\end{proposition}

\begin{proof}
By \cref{prop:bps-reversible}, \(\PB^*=\PB\).  Therefore $\RB^*=(\PB^2)^*=\PB^{*2}=\PB^2=\RB,$ and $
  \ip{f}{\RB f}
  =\ip{f}{\PB^*\PB f}
  =\ip{\PB f}{\PB f}
  =\norm{\PB f}_2^2.$ This proves both self-adjointness and positive semidefiniteness.  One transition of \(\RB\) is exactly two consecutive transitions of the original NUTS-BPS chain; no self-loop or auxiliary kernel has been added.
\end{proof}

\section{Terminal-depth certificates}
\label{sec:terminal}

The subsequent selector and conductance arguments require only a terminal-depth certificate: every realization of the doubling bits returns the same orbit cardinality, and all numerical energies on every possible returned orbit remain in a fixed window.  The profile analysis will produce this certificate through a genuine U-turn before the cap.

\begin{definition}[Terminal-depth certificate]
\label{def:terminal-certificate}
Fix a power of two \(K=2^k\), an energy tolerance \(\eta>0\), and a position \(x\in\R^d\).  A momentum \(v\) is \((K,\eta)\)-certified if the following statements hold.
\begin{enumerate}[label=(\roman*)]
\item For every realization of the fair left/right doubling bits, Algorithm \ref{alg:orbit} returns an orbit of cardinality \(K\).
\item Every leapfrog point that can belong to such an orbit satisfies
\begin{align*}
\max_{|j|\le K-1}
\abs{H(\Phi_h^j(x,v))-H(x,v)}\le\eta.
\end{align*}
\end{enumerate}
The set of certified momenta is denoted by \(G_x(K,\eta)\).
\end{definition}

In the application below, all proper dyadic descendants are nonturning, while every full interval of size \(K\) has a U-turn.  The builder accepts the final doubling and terminates on the ensuing full-orbit check, strictly before the maximum depth.

\begin{proposition}[Deterministic terminal depth]
\label{prop:terminal-depth} This deterministic sign-to-depth mechanism abstracts the orbit-uniformity argument used for Gaussian NUTS in \citet[Proposition~4]{Oberdoerster2025}.
Let \(K=2^k\).  Suppose that, for a phase point \(z\), every dyadic interval of cardinality strictly less than \(K\) that can be inspected by Algorithm \ref{alg:orbit} has U-turn indicator zero.  Suppose in addition that either
\begin{enumerate}[label=(\alph*)]
\item \(K=2^{k_{\max}}\), or
\item every possible length-\(K\) interval containing the initial index has U-turn indicator one.
\end{enumerate}
Then every realization of the doubling bits returns an orbit of cardinality \(K\).  In case (b), termination occurs through the U-turn criterion strictly before the maximum-depth cap.
\end{proposition}

\begin{proof}
At depth zero the current orbit has one point.  Assume inductively that the builder has accepted all extensions through depth \(r<k\).  The proposed new half has cardinality \(2^r<K\).  Every dyadic descendant of that half also has cardinality less than \(K\), so its U-turn indicator is zero by assumption.  Hence the recursive sub-U-turn indicator of the new half is zero, and the extension is accepted.  This proves that the builder reaches depth \(k\) for every sequence of bits.

In case (a), the orbit is returned because its cardinality equals the maximum size.  In case (b), the accepted length-\(K\) orbit has U-turn indicator one.  The next evaluation of the while-condition therefore terminates the builder and returns that orbit.  Since \(k<k_{\max}\), the cap is not reached.
\end{proof}

For \(j\in\Z\), define the fixed-index leapfrog proposal
\begin{align*}
Q_{j,x}=\Law\!\left(\Pi_x\Phi_h^j(x,V)\right),
\qquad V\sim N(0,I_d).
\end{align*}
The ideal signed-index laws associated with a certified orbit are
\begin{align}
\lambda_{\mathrm M,K}(j)=\frac{K-|j|}{K^2},
\qquad |j|\le K-1,
\label{eq:terminal-mul-law}
\end{align}
and
\begin{align}
\lambda_{\mathrm B,K}(0)=0,
\qquad
\lambda_{\mathrm B,K}(j)
=\frac{2}{K^2}\min\{|j|,K-|j|\},
\qquad 1\le|j|\le K-1.
\label{eq:terminal-bps-law}
\end{align}
Their derivation is given in \Cref{sec:selector-minorization}.

For \(a\in\{\mathrm M,\mathrm B\}\), define the positive-skeleton exponents and
kernels by
\begin{align}
  r_{\mathrm M}=1,
  \qquad
  r_{\mathrm B}=2,
  \qquad
  R_a=P_a^{r_a}.
  \label{eq:positive-skeleton-exponents}
\end{align}

Before we proceed, for finite measures \(\mu,\nu\) dominated by a common measure \(\lambda\), define
\begin{align}\label{equation:overlap}
\ov(\mu,\nu)
=
\int
\min\left\{
\frac{\dd\mu}{\dd\lambda},
\frac{\dd\nu}{\dd\lambda}
\right\}\dd\lambda.
\end{align}
For probability measures, note that we have
\begin{align*}
\ov(\mu,\nu)=1-\norm{\mu-\nu}_{\mathrm{TV}}.
\end{align*}

\begin{theorem}[Terminal-depth transfer]
\label{thm:terminal-transfer}
Fix \(a\in\{\mathrm M,\mathrm B\}\), a power of two \(K\ge8\), a measurable
set \(D\subset\R^d\), \(\eta>0\), \(\delta\in[0,1)\), \(M\ge1\), and
\(\epsilon\in(0,1/2)\). Suppose that, for every \(x\in D\), $\Pp_V\bigl(G_x(K,\eta)^c\bigr)\le\delta.$ Let \(\mathcal J\subset\{1,\ldots,K-1\}\), and assume
\begin{align}
  a_{a,K}:=\sum_{j\in\mathcal J}\lambda_{a,K}(j)>0.
  \label{eq:band-mass}
\end{align}
Suppose there are \(\Delta>0\) and \(\beta\in(0,1]\) such that
\begin{align}
  \ov(Q_{j,x},Q_{j,y})\ge\beta
  \label{eq:fixed-overlap-assumption}
\end{align}
for every \(j\in\mathcal J\) and every \(x,y\in D\) with
\(\norm{x-y}\le\Delta\). Let $c_{\mathrm M}(\eta)=e^{-2\eta},$ $c_{\mathrm B}(\eta)=e^{-4\eta},$ $\rho_a=c_a(\eta)\bigl(a_{a,K}\beta-2\delta\bigr),$ $\theta=\min\{1,\psi\Delta\},$ and assume \(\rho_a>0\). Let $s={\epsilon^2}/{128M^2}$ and suppose
\begin{align}
  \pi(D^c)\le\frac{\theta s}{64}.
  \label{eq:D-general-tail}
\end{align}
Then the positive skeleton in
\eqref{eq:positive-skeleton-exponents} satisfies
\begin{align}
  \Phi_s(R_a)\ge c\rho_a\theta
  \label{eq:general-conductance}
\end{align}
with the universal choice \(c=3/256\). Consequently, for every
\(M\)-warm probability law \(\nu\),
\begin{align}
  \norm{\nu P_a^n-\pi}_{\mathrm{TV}}\le\epsilon
  \label{eq:general-transfer-tv}
\end{align}
whenever
\begin{align}
  n\ge
  C\rho_a^{-2}\theta^{-2}
  \log\!\left(\frac{CM}{\epsilon}\right),
  \label{eq:general-transfer-time}
\end{align}
where \(C\) is universal and includes the factor \(r_a\le2\).
\end{theorem}

\begin{proof}[Proof of \Cref{thm:terminal-transfer}]
Fix \(a\in\{\mathrm M,\mathrm B\}\). For \(x\in D\), define $Q_{a,x}=\sum_j\lambda_{a,K}(j)Q_{j,x}$ and let \(\widetilde Q_{a,x}\) be the subprobability measure obtained by
multiplying the common momentum integral by
\(\1_{G_x(K,\eta)}\). Because the selector weights sum to one,
\begin{align}
  Q_{a,x}(\R^d)-\widetilde Q_{a,x}(\R^d)
  =\Pp_V\bigl(G_x(K,\eta)^c\bigr)
  \le\delta.
  \label{eq:generic-restriction-mass}
\end{align}
The selector minorizations in
\Cref{lem:mul-minorization,lem:bps-minorization} give $  P_a(x,\cdot)\ge c_a(\eta)\widetilde Q_{a,x}(\cdot).$

Let \(x,y\in D\) satisfy \(\norm{x-y}\le\Delta\). By
\eqref{eq:mixture-overlap}, \eqref{eq:band-mass}, and
\eqref{eq:fixed-overlap-assumption}, we have $\ov(Q_{a,x},Q_{a,y}) \ge \sum_{j\in\mathcal J}\lambda_{a,K}(j) \ov(Q_{j,x},Q_{j,y}) \ge a_{a,K}\beta.$ Equations \eqref{eq:restriction-overlap} and
\eqref{eq:generic-restriction-mass} imply $ \ov(\widetilde Q_{a,x},\widetilde Q_{a,y})
  \ge a_{a,K}\beta-2\delta.$ If two probability measures dominate \(c\mu\) and \(c\nu\), respectively,
where \(\mu\) and \(\nu\) are finite measures, then their overlap is at
least \(c\ov(\mu,\nu)\). Hence $  \ov(P_a(x,\cdot),P_a(y,\cdot))\ge\rho_a.$ For \(a=\mathrm M\), \(R_a=P_a\). For \(a=\mathrm B\), total
variation contracts under a common Markov kernel, so
\begin{align*}
  \norm{R_a(x,\cdot)-R_a(y,\cdot)}_{\mathrm{TV}} =
  \norm{P_a^2(x,\cdot)-P_a^2(y,\cdot)}_{\mathrm{TV}} \le
  \norm{P_a(x,\cdot)-P_a(y,\cdot)}_{\mathrm{TV}}.
\end{align*}
Thus, for both variants,
\begin{align}
  \norm{R_a(x,\cdot)-R_a(y,\cdot)}_{\mathrm{TV}}
  \le1-\rho_a.
  \label{eq:generic-skeleton-overlap}
\end{align}

We now prove the conductance bound. Let \(S\) satisfy
\(s<\alpha:=\pi(S)\le1/2\), let $q=\mathcal Q_{R_a}(S,S^c),$ $ d_0=\pi(D^c),$ and define
\begin{align*}
A=\left\{x\in S\cap D:R_a(x,S^c)<\frac{\rho_a}{4}\right\},
\quad\text{and}\quad
B=\left\{y\in S^c\cap D:R_a(y,S)<\frac{\rho_a}{4}\right\}.
\end{align*}
The definition of \(q\) and Markov's inequality give
\begin{align}
  \pi(A)\ge\alpha-d_0-\frac{4q}{\rho_a}.
  \label{eq:generic-A-mass}
\end{align}
Because \(R_a\) is reversible,
\(\mathcal Q_{R_a}(S^c,S)=q\), and therefore
\begin{align}
  \pi(B)\ge1-\alpha-d_0-\frac{4q}{\rho_a}.
  \label{eq:generic-B-mass}
\end{align}
For \(x\in A\) and \(y\in B\),
\[
  R_a(x,S)>1-\frac{\rho_a}{4},
  \qquad
  R_a(y,S)<\frac{\rho_a}{4},
\]
so
\begin{align*}
  \norm{R_a(x,\cdot)-R_a(y,\cdot)}_{\mathrm{TV}}
  >1-\frac{\rho_a}{2}.
\end{align*}
Comparison with \eqref{eq:generic-skeleton-overlap} shows that $\dist(A,B)>\Delta.$

If \(q\ge\rho_a\alpha/16\), then, since \(\theta\le1\),
\begin{align}
  q\ge\frac{\rho_a\theta\alpha}{16}.
  \label{eq:generic-large-flow-case}
\end{align}
Suppose instead that \(q<\rho_a\alpha/16\). Since
\(d_0\le\theta s/64<\alpha/64\),
\eqref{eq:generic-A-mass}--\eqref{eq:generic-B-mass} imply
\begin{align*}
  \pi(A)\ge\frac\alpha2,
  \qquad
  \pi(B)\ge\frac14.
\end{align*}
Let \(C=(A\cup B)^c\). Apply \eqref{eq:cheeger} to the partition
\(A\sqcup B\sqcup C\). Since
\(\dist(A,B)>\min\{\Delta,\psi^{-1}\}\),
\begin{align*}
  \pi(C)
  \ge\theta\pi(A)\pi(B)
  \ge\frac{\theta\alpha}{8}.
\end{align*}
On the other hand,
\begin{align*}
  \pi(C)
  \le2d_0+\frac{8q}{\rho_a}
  \le\frac{\theta s}{32}+\frac{8q}{\rho_a}.
\end{align*}
Combining the last two displays and using \(s<\alpha\) gives
\[
  \frac{8q}{\rho_a}
  \ge
  \frac{\theta\alpha}{8}-\frac{\theta s}{32}
  >
  \frac{3\theta\alpha}{32},
\]
and hence
\begin{align*}
  q\ge\frac{3\rho_a\theta\alpha}{256}.
\end{align*}
Together with \eqref{eq:generic-large-flow-case}, this proves
\[
  \mathcal Q_{R_a}(S,S^c)
  \ge
  \frac{3}{256}\rho_a\theta\pi(S).
\]
Since \(\pi(S)-s\le\pi(S)\),
\eqref{eq:general-conductance} follows with \(c=3/256\).

By \Cref{cor:PSD,prop:bps-square-PSD}, \(R_a\) is reversible and positive
semidefinite. Let
\[
  N_0
  =
  C\rho_a^{-2}\theta^{-2}
  \log\!\left(\frac{CM}{\epsilon}\right)
\]
with \(C\) large enough for \Cref{prop:PSD-mixing}. Then
\(\norm{\nu R_a^m-\pi}_{\mathrm{TV}}\le\epsilon\) for every integer
\(m\ge N_0\). For an arbitrary integer \(n\), write
\(n=r_am+q_0\), where \(m=\lfloor n/r_a\rfloor\) and
\(0\le q_0<r_a\). Since
\[
  \nu P_a^n=(\nu R_a^m)P_a^{q_0}
\]
and total variation contracts under \(P_a^{q_0}\), the same error bound holds
whenever \(m\ge N_0\). Absorbing \(r_a\le2\) into the universal constant
proves \eqref{eq:general-transfer-tv}--\eqref{eq:general-transfer-time}.
\end{proof}

\section{Trajectory, energy, and fixed-time estimates}
\label{sec:analytic}

This section records the deterministic and probabilistic estimates used to turn a separated exact-flow profile into a certified leapfrog orbit.  No artificial stopping mechanism is introduced.  The short-time lemma is used only for the first few dyadic intervals, where the profile has not yet accumulated a fixed macroscopic margin.
\subsection{A deterministic simultaneous no-U-turn lemma}
\label{sec:no-uturn}

Write a leapfrog trajectory as $(q_j,p_j)=\Phi_h^j(q_0,p_0),$ and $g_j=\nabla U(q_j).$  Combining \eqref{eq:lf1}--\eqref{eq:lf3} gives
\begin{align}
  q_{j+1}
  &=
  q_j+hp_j-\frac{h^2}{2}g_j,
  \label{eq:lf-position-recursion}\\
  p_{j+1}
  &=
  p_j-\frac h2(g_j+g_{j+1}).
  \label{eq:lf-momentum-recursion}
\end{align}

\begin{lemma}[Short numerical trajectories do not turn]
\label{lem:no-uturn}
Let \(K\in\N\), \(T=Kh\), and $  G_0=\norm{\nabla U(q_0)}.
$ Assume \(\norm{p_0}>0\).  Suppose \(U\) is \(L\)-smooth and, for some \(B\ge0\),
\begin{align}
  G_0\le B\sqrt L\,\norm{p_0}.
  \label{eq:gradient-momentum-ratio}
\end{align}
If
\begin{align}
  \sqrt L\,T
  \le
  \frac{1}{64(B+1)},
  \label{eq:no-uturn-time-condition}
\end{align}
then, for every pair of integers \(-K\le i<j\le K\) satisfying
\((j-i)h\le T\),
\begin{align}
  p_i^\top(q_j-q_i)>0,
  \qquad
  p_j^\top(q_j-q_i)>0.
  \label{eq:no-uturn-conclusion}
\end{align}
Consequently, every orbit and every suborbit of physical length at most \(T\) that is inspected by \cref{alg:orbit} has neither a U-turn nor a sub-U-turn.
\end{lemma}

\begin{proof}
Fix integers \(i<j\), set $r=j-i,$ $\tau=rh,$ and define $G_*=\max_{|n|\le K}\norm{g_n}.$ Iterating \eqref{eq:lf-momentum-recursion} gives, for \(t\ge1\),
\begin{align}
  p_{i+t}
  =
  p_i-\frac h2g_i
  -h\sum_{s=1}^{t-1}g_{i+s}
  -\frac h2g_{i+t}.
  \label{eq:momentum-expanded}
\end{align}
For \(t=0,\ldots,r-1\), substitute
\eqref{eq:momentum-expanded} into
\eqref{eq:lf-position-recursion}:
\begin{align*}
  q_{i+t+1}-q_{i+t}
  &=
  hp_i-\frac{h^2}{2}g_i
  -h^2\sum_{s=1}^{t}g_{i+s}.
\end{align*}
Summing over \(t\) yields
\begin{align}
  q_j-q_i
  =
  rhp_i
  -
  h^2\left[
    \frac r2g_i
    +
    \sum_{s=1}^{r-1}(r-s)g_{i+s}
  \right].
  \label{eq:position-expanded}
\end{align}
Because
\[
  \frac r2+\sum_{s=1}^{r-1}(r-s)
  =
  \frac{r^2}{2},
\]
Cauchy--Schwarz gives
\begin{align}
  p_i^\top(q_j-q_i)
  \ge
  \tau\norm{p_i}^2
  -
  \frac{\tau^2}{2}\norm{p_i}G_*.
  \label{eq:left-diagnostic-lower}
\end{align}

To obtain the right-endpoint bound, start at
\((q_j,-p_j)\) and apply \(r\) forward leapfrog steps.
By \eqref{eq:lf-reversible}, the resulting point is
\((q_i,-p_i)\).  Applying \eqref{eq:left-diagnostic-lower} to
this reversed trajectory gives
\begin{align}
  p_j^\top(q_j-q_i)
  \ge
  \tau\norm{p_j}^2
  -
  \frac{\tau^2}{2}\norm{p_j}G_*.
  \label{eq:right-diagnostic-lower}
\end{align}

It remains to bound \(G_*\) and the momenta.  Applying
\eqref{eq:position-expanded} forward or backward from index \(0\)
shows that, for every \(|n|\le K\),
\begin{align*}
  \norm{q_n-q_0}
  \le
  T\norm{p_0}+\frac{T^2}{2}G_*.
\end{align*}
Since \(\nabla U\) is \(L\)-Lipschitz,
\begin{align*}
  G_*\le
  G_0+
  L\max_{|n|\le K}\norm{q_n-q_0} \le
  G_0+LT\norm{p_0}+\frac{LT^2}{2}G_*.
\end{align*}
Condition \eqref{eq:no-uturn-time-condition} implies
\(LT^2<1\), and therefore
\begin{align}
  G_*
  \le
  \frac{G_0+LT\norm{p_0}}
  {1-LT^2/2}.
  \label{eq:Gstar-bound}
\end{align}
Likewise, \eqref{eq:momentum-expanded} gives
\begin{align}
  \norm{p_n-p_0}\le TG_*,
  \qquad |n|\le K.
  \label{eq:momentum-displacement}
\end{align}

Set \(u=\sqrt L\,T\).  By
\eqref{eq:gradient-momentum-ratio} and \eqref{eq:Gstar-bound},
\begin{align*}
  \frac{TG_*}{\norm{p_0}}
  \le
  \frac{u(B+u)}{1-u^2/2}.
\end{align*}
Under \eqref{eq:no-uturn-time-condition},
\[
  u(B+u)\le u(B+1)\le\frac1{64},
  \qquad
  1-\frac{u^2}{2}\ge\frac{63}{64}.
\]
Consequently,
\begin{align}
  TG_*\le\frac1{63}\norm{p_0}.
  \label{eq:TG-small}
\end{align}
Equations \eqref{eq:momentum-displacement} and \eqref{eq:TG-small}
imply $\norm{p_n} \ge
  \frac{62}{63}\norm{p_0},$ and  $|n|\le K.$

For the chosen \(i,j\),
\[
  \frac{\tau G_*}{2}
  \le
  \frac{TG_*}{2}
  \le
  \frac{\norm{p_0}}{126}
  <
  \frac12\norm{p_i},
\]
and the same inequality holds with \(p_j\) in place of \(p_i\).
Substitution into \eqref{eq:left-diagnostic-lower} and
\eqref{eq:right-diagnostic-lower} gives
\[
  p_i^\top(q_j-q_i)
  >
  \frac{\tau}{2}\norm{p_i}^2>0,
  \qquad
  p_j^\top(q_j-q_i)
  >
  \frac{\tau}{2}\norm{p_j}^2>0.
\]
Every sub-U-turn is built recursively from endpoint U-turn tests on
subintervals, so \eqref{eq:no-uturn-conclusion} rules out all of them.
\end{proof}

\subsection{Imported leapfrog estimates}\label{sec:analytic-inputs}

We next record the two nontrivial analytic estimates imported from \cite{ChenGatmiryJiang2026}.  Their notation uses a strongly Hessian-Lipschitz coefficient; after the normalization
\eqref{eq:hessian-lip}, that coefficient is \(\gamma L^{3/2}\).
The constants below are universal.

\begin{proposition}[Single-step energy moment]
\label{prop:energy-input}
Let \(Z=(X,V)\sim\bar\pi\), let \(\ell\ge2\) be even, and define $\Delta_hH(Z)=H(\Phi_hZ)-H(Z).$ For \(h\sqrt L\) sufficiently small,
\begin{align}
  \norm{\Delta_hH}_{L^\ell(\bar\pi)}
  \le
  C_E(1+\gamma)\Big[h^3\ell^{3/2}L^{3/2}d_\ell^{1/2}+
  h^5\ell^{1/2}L^{5/2}d_\ell
  +
  h^7L^{7/2}d_\ell^{3/2}
  \Big].
  \label{eq:energy-input}
\end{align}
\end{proposition}

\begin{proposition}[Fixed-index proposal overlap]
\label{prop:overlap-input}
For \(j\in\N\), define
\[
  Q_{j,x}
  =
  \Law\!\left(\Pi_x\Phi_h^j(x,V)\right),
  \qquad V\sim N(0,I_d).
\]
If \(jh\sqrt L\le1/4\), then
\begin{align}
  \norm{x-y}\le\frac{jh}{64}
  \quad\Longrightarrow\quad
  \norm{Q_{j,x}-Q_{j,y}}_{\TV}
  \le
  \left[
    \frac1{32}
    +
    2(jh)^6\gamma^2L^3
  \right]^{1/2}.
  \label{eq:overlap-input}
\end{align}
\end{proposition}

\begin{remark}[Use of the imported estimates]
\label{rem:imported}
The proof below does not use the Metropolis correction, the fixed-length HMC kernel, its lazification, or its conductance theorem from~\cite{ChenGatmiryJiang2026}.  Only
\eqref{eq:energy-input} and \eqref{eq:overlap-input} are used.
\end{remark}

\subsection{Uniform energy control over the terminal orbit}
\label{sec:energy-orbit}

Let
\begin{align*}
  A_\ell(h)
  =
  C_E(1+\gamma)\left[
    h^3\ell^{3/2}L^{3/2}d_\ell^{1/2}
    +
    h^5\ell^{1/2}L^{5/2}d_\ell
    +
    h^7L^{7/2}d_\ell^{3/2}
  \right].
\end{align*}
The following lemma upgrades the stationary single-step estimate to a uniform estimate along a deterministic leapfrog orbit.  The induction is needed because leapfrog does not preserve \(\bar\pi\).

\begin{lemma}[Whole-orbit energy window]
\label{lem:whole-energy}
Let \(\eta_0>0\), \(K\in\N\), and let \(Z\sim\bar\pi\).  Then
\begin{align}
  \bar\pi\!\left(
    \max_{|j|\le K}
    \abs{H(\Phi_h^jZ)-H(Z)}
    >
    \eta_0
  \right)
  \le
  2K e^{\eta_0}
  \left(
    \frac{KA_\ell(h)}{\eta_0}
  \right)^\ell.
  \label{eq:whole-energy}
\end{align}
\end{lemma}

\begin{proof}
We first treat forward iterates.  Write $z_j=\Phi_h^jz_0,$,  $ a={\eta_0}/{K},$ and define
\begin{align*}
  W_k
  =
  \bigcap_{j=0}^{k-1}
  \left\{
    \abs{H(z_{j+1})-H(z_j)}\le a
  \right\},
  \qquad
  W_0=\mathsf Z.
\end{align*}
On \(W_{k-1}\),
\begin{align}
  \abs{H(z_{k-1})-H(z_0)}
  \le
  (k-1)a
  \le\eta_0.
  \label{eq:W-energy}
\end{align}
Let
\[
  \Delta_k(z_0)=H(z_k)-H(z_{k-1}).
\]
Using the change of variables \(y=\Phi_h^{k-1}z_0\), volume preservation, and
\eqref{eq:W-energy}, we obtain
\begin{align*}
  &\E_{\bar\pi}\!\left[
    \abs{\Delta_k(Z)}^\ell\1_{W_{k-1}}(Z)
  \right]
  \notag\\
  &\quad=
  \bar Z^{-1}
  \int
  \abs{\Delta_hH(\Phi_h^{k-1}z)}^\ell
  \1_{W_{k-1}}(z)e^{-H(z)}\dd z
  \notag\\
  &\quad\le
  e^{\eta_0}\bar Z^{-1}
  \int
  \abs{\Delta_hH(y)}^\ell e^{-H(y)}\dd y
  \notag\\
  &\quad\le
  e^{\eta_0}A_\ell(h)^\ell.
\end{align*}
Markov's inequality therefore gives
\begin{align}
  \Pp_{\bar\pi}(W_k^c\cap W_{k-1})
  \le
  e^{\eta_0}
  \left(
    \frac{KA_\ell(h)}{\eta_0}
  \right)^\ell.
  \label{eq:W-failure}
\end{align}
The disjoint decomposition
\[
  W_K^c
  =
  \bigsqcup_{k=1}^K(W_k^c\cap W_{k-1})
\]
and \eqref{eq:W-failure} imply
\begin{align*}
  \Pp_{\bar\pi}(W_K^c)
  \le
  K e^{\eta_0}
  \left(
    \frac{KA_\ell(h)}{\eta_0}
  \right)^\ell.
\end{align*}
On \(W_K\), every cumulative forward energy difference is at most
\(Ka=\eta_0\).

For backward iterates, use
\(\Phi_h^{-1}=S\Phi_hS\) and \(H\circ S=H\).
The law \(\bar\pi\) is invariant under \(S\), so the same argument gives the same bound.  A union bound over the forward and backward failures proves
\eqref{eq:whole-energy}.
\end{proof}

The lemma is used at the terminal orbit size selected by the U-turn rule.  In the proof of \Cref{prop:tangent-certificate}, the relation $Kh\le {S}/{\sqrt m}+h$ and the step-size choice \eqref{eq:adaptive-h}, together with \eqref{eq:eta-ad-explicit}, imply after reducing the universal constant \(c\) that
\begin{align*}
  \frac{K A_\ell(h)}{\log 2}\le e^{-6}.
\end{align*}
Consequently, \Cref{lem:whole-energy} gives
\begin{align*}
  \bar\pi\!\left(
    \max_{|j|\le K}
    \abs{H(\Phi_h^jZ)-H(Z)}
    >\log2
  \right)
  \le 4K e^{-6\ell}.
\end{align*}
The complete calculation, including the dependence on \(\kappa,S\), and \(\bar\gamma\), appears in \eqref{eq:adaptive-energy-leading-general}--\eqref{eq:adaptive-energy-small-general}.

\section{Stationary U-turn geometry}
\label{sec:profile-proof}

Recent analyses of NUTS for Gaussian targets identify a deterministic
time-dependent quantity around which the random endpoint U-turn
diagnostics concentrate.  For the canonical Gaussian target, this
quantity is \(d\sin t\); for a Gaussian target \(\mathcal N(0,C)\), it is $f_{\mathrm{unif}}(t)
  =
  \operatorname{tr}
  \left(
    \sin(C^{-1/2}t)\,C^{1/2}
  \right).$ These deterministic profiles are then used to show that, away from
their zeros, all relevant full-orbit and recursive subtree tests have
predictable signs and the random-doubling procedure terminates at a
common depth with high probability; see
\textcolor{black}{\cite{BouRabeeOberdoerster2024} and \cite{Oberdoerster2025}}.

This section extends that deterministic-profile mechanism beyond the
Gaussian setting.  For a general strongly log-concave target, we define
the stationary U-turn profile
\[
  u(t)
  =
  \mathbb E_{\bar\pi}
  \left[
    V_0^\top(X_t-X_0)
  \right]
  =
  \mathbb E_{\bar\pi}
  \left[
    V_t^\top(X_t-X_0)
  \right],
\]
and derive the equivalent representations
\[
  u(t)
  =
  \frac12\frac{d}{dt}
  \mathbb E_{\bar\pi}\|X_t-X_0\|^2
  =
  \mathbb E_{\bar\pi}\operatorname{tr}(J_t),
\]
where \(J_t=\partial X_t/\partial V_0\) is the Hamiltonian Jacobi
field.  Further concentration estimates, numerical comparison, terminal-depth certificate, and mixing conclusion in \Cref{thm:adaptive} are proved.  All constants are finite for every fixed \(\kappa\) and finite dimensionless horizon \(S\).

\subsection{Equilibrium identities}

Let \(\flow_t\) denote the exact Hamiltonian flow generated by \(H(x,v)=U(x)+\norm v^2/2\), and initialize \(Z_0=(X_0,V_0)\sim\bar\pi\).  Since exact Hamiltonian flow preserves \(\bar\pi\), the process \((Z_t)_{t\in\R}\) is stationary.  It is also reversible under momentum reflection:
\begin{align}
(X_0,V_0,X_t,V_t)
\overset{d}{=}
(X_t,-V_t,X_0,-V_0).
\label{eq:flow-time-reversal-law}
\end{align}

\begin{lemma}[Profile representations]
\label{lem:profile-representations}
Under \eqref{eq:strong-convexity}, the quantities in \eqref{eq:profile-def} are continuously differentiable and satisfy
\begin{align*}
u(t)=\frac12R'(t)=\E\Tr J_t,
\qquad
J_t=\frac{\partial X_t}{\partial V_0},
\end{align*}
where
\begin{align}
\ddot J_t+\nabla^2U(X_t)J_t=0,
\qquad
J_0=0,
\qquad
\dot J_0=I_d.
\label{eq:Jacobi}
\end{align}
\end{lemma}

\begin{proof}
Strong convexity implies Gaussian tails for \(\pi\), and the Hamiltonian vector field has globally Lipschitz first derivative because \(\norm{\nabla^2U}_{\mathrm{op}}\le L\).  Hence the flow and its first derivative have finite moments on every bounded time interval.  Differentiation under the expectation below is therefore justified by dominated convergence.

Under the transformation in \eqref{eq:flow-time-reversal-law}, $V_0^\top(X_t-X_0)
\longmapsto
(-V_t)^\top(X_0-X_t)
=V_t^\top(X_t-X_0).$ Thus \(\E D_t^-=\E D_t^+\).  Moreover,
\begin{align*}
R'(t)=2\E\left[(X_t-X_0)^\top\dot X_t\right]=2\E\left[(X_t-X_0)^\top V_t\right]
=2\E D_t^+.
\end{align*}
This proves \(u(t)=R'(t)/2\).

For the Jacobi representation, write $J_t={\partial X_t}/{\partial V_0}.$ Gaussian integration by parts in the initial momentum gives, coordinate by coordinate,
\begin{align*}
\E D_t^-=\sum_{i=1}^d\E\left[V_{0,i}(X_{t,i}-X_{0,i})\right]
=\sum_{i=1}^d\E\left[\frac{\partial X_{t,i}}{\partial V_{0,i}}\right]
=\E\Tr J_t.
\end{align*}
The term involving \(X_{0,i}\) vanishes because \(X_0\) and \(V_0\) are independent and centered in momentum.  Finally, differentiating
\(
\dot X_t=V_t
\)
and
\(
\dot V_t=-\nabla U(X_t)
\)
with respect to \(V_0\) gives
\[
\dot J_t=\frac{\partial V_t}{\partial V_0},
\qquad
\ddot J_t=-\nabla^2U(X_t)J_t.
\]
The initial conditions are \(J_0=0\) and \(\dot J_0=I_d\), proving \eqref{eq:Jacobi}.
\end{proof}

\subsection{Mean-square displacement and universal initial positivity}

\begin{lemma}[Mean-square displacement inequalities]
\label{lem:R-inequalities}
Under \eqref{eq:strong-convexity},
\begin{align}
2d-LR(t)\le R''(t)\le2d-mR(t).
\label{eq:R-differential}
\end{align}
In addition,
\begin{align*}
R(t)\le\frac{4d}{m},
\qquad
\abs{R'(t)}\le\frac{4d}{\sqrt m},
\qquad
\abs{u'(t)}\le2\kappa d.
\end{align*}
Finally,
\begin{align}
u(t)\ge\frac d{\sqrt L}\sin(\sqrt L\,t)>0,
\qquad
0<t\le\frac{\pi}{2\sqrt L}.
\label{eq:profile-positive}
\end{align}
\end{lemma}

\begin{proof}
Let \(\Delta_t=X_t-X_0\).  Since \(\dot V_t=-\nabla U(X_t)\),
\begin{align}
R''(t)=2\E\norm{V_t}^2
-2\E[\Delta_t^\top\nabla U(X_t)].
\label{eq:Rsecond-first}
\end{align}
Stationarity gives \(\E\norm{V_t}^2=d\).  Set $A(t)=\E[\Delta_t^\top\nabla U(X_t)].$ Time-reversal exchangeability of \((X_0,X_t)\) yields $\E[\Delta_t^\top\nabla U(X_0)]=-A(t),$ and therefore
\begin{align*}
2A(t)=\E\bigl[\Delta_t^\top\{\nabla U(X_t)-\nabla U(X_0)\}\bigr].
\end{align*}
Strong monotonicity and \(L\)-Lipschitz continuity of \(\nabla U\) imply
\[
\frac m2R(t)\le A(t)\le\frac L2R(t).
\]
Substitution into \eqref{eq:Rsecond-first} proves \eqref{eq:R-differential}.

Let \(x_\star\) be the minimizer of \(U\).  Integration by parts gives $\E[(X-x_\star)^\top\nabla U(X)]=d.$ Strong convexity then gives \(\E\norm{X-x_\star}^2\le d/m\), so
\[
R(t)\le2\E\norm{X_t-x_\star}^2+2\E\norm{X_0-x_\star}^2\le\frac{4d}{m}.
\]
Cauchy--Schwarz gives
\[
\abs{R'(t)}
=2\abs{\E[\Delta_t^\top V_t]}
\le2\sqrt{R(t)}\sqrt d
\le\frac{4d}{\sqrt m}.
\]
Since \(u'=R''/2=d-A(t)\), the preceding bounds imply \(\abs{u'(t)}\le2\kappa d\).

For initial positivity, define $y_L(t)=\frac{2d}{L}\{1-\cos(\sqrt L\,t)\}.$ Then \(y_L''+Ly_L=2d\) and \(y_L(0)=y_L'(0)=0\).  With \(w=R-y_L\), the left inequality in \eqref{eq:R-differential} gives
\[
w''+Lw\ge0,
\qquad w(0)=w'(0)=0.
\]
Writing \(f=w''+Lw\ge0\), variation of constants gives $w'(t)=\int_0^t\cos(\sqrt L(t-s))f(s)\dd s$. For \(0\le t\le\pi/(2\sqrt L)\), the kernel is nonnegative.  Hence
\[
R'(t)\ge y_L'(t)=\frac{2d}{\sqrt L}\sin(\sqrt L\,t).
\]
Since \(u=R'/2\), this is \eqref{eq:profile-positive}.
\end{proof}
\subsection{Concentration of the exact endpoint diagnostics}

We use the following consequence of the logarithmic Sobolev inequality for strongly log-concave measures; see, for example, \cite[Chapter~5]{BakryGentilLedoux2014}.

\begin{proposition}[An \(L^p\) gradient inequality]
\label{prop:Lp-gradient}
Let \(\mu(\dd z)\propto e^{-W(z)}\dd z\) on \(\R^n\), with \(\nabla^2W\succeq I_n\).  For every locally Lipschitz \(F\) and every \(p\ge2\),
\begin{align*}
\norm{F-\E_\mu F}_{L^p(\mu)}
\le C\sqrt p\,
\norm{\norm{\nabla F}}_{L^p(\mu)},
\end{align*}
where \(C\) is universal.
\end{proposition}

Let \(x_\star\) be the minimizer of \(U\), and introduce dimensionless variables
\begin{align}
Y=\sqrt m\,(X-x_\star),
\qquad W=V,
\qquad s=\sqrt m\,t,
\label{eq:dimensionless-variables}
\end{align}
with potential
\begin{align*}
\widetilde U(y)=U\!\left(x_\star+\frac{y}{\sqrt m}\right)-U(x_\star).
\end{align*}
Then
\begin{align}
I_d\preceq\nabla^2\widetilde U(y)\preceq\kappa I_d.
\label{eq:scaled-hessian}
\end{align}
The phase-space law
\begin{align*}
\widetilde\pi(\dd y\,\dd w)
\propto e^{-\widetilde U(y)-\norm w^2/2}\dd y\,\dd w
\end{align*}
is therefore one-strongly log-concave.

\begingroup\color{black}
Fix a deterministic verification horizon \(S>0\), put
\[
  S_\kappa^+=S+\kappa^{-1/2},
\]
and define
\begin{align}
  \mathfrak A_S(U)
  &=
  \sup_{z\in\R^{2d}}
  \sup_{|s|\le S_\kappa^+}
  \norm{D\widetilde\flow_s(z)}_{\mathrm{op}},
  \label{eq:tangent-amplification}\\
  \mathfrak C_S(U)
  &=C_0\sqrt\kappa\bigl(1+\mathfrak A_S(U)\bigr),
  \label{eq:concentration-envelope}\\
  \mathfrak N_S(U)
  &=C_0(1+S)\bigl(1+\mathfrak A_S(U)\bigr)
    \kappa^{3/2}\bigl(1+\kappa^2+\bar\gamma\bigr),
  \label{eq:numerical-envelope}
\end{align}
where \(C_0\ge1\) is universal.  These quantities are used only as one
sufficient verification of the intrinsic moduli in
\Cref{def:diagnostic-concentration-modulus,def:numerical-fidelity-modulus}.
\endgroup

\begin{proposition}[\textcolor{black}{Tangent-flow verification of the intrinsic moduli}]
\label{prop:stability-envelope}
The universal constant \(C_0\) in
\eqref{eq:concentration-envelope}--\eqref{eq:numerical-envelope} may be chosen so that
\begin{align}
  \mathfrak C_p(S;U)&\le \mathfrak C_S(U),
  \label{eq:tangent-implies-Cp}\\
  \mathfrak E_{p,h}(K,1;U)
  &\le
  C\mathfrak N_S(U)(mh^2)\frac{(d+p)^{3/2}}{d}
  \label{eq:tangent-implies-Ep-unit}
\end{align}
whenever the numerical orbit has scaled horizon at most \(S_\kappa^+\) and the unit-width
energy event holds.  Equivalently, on that event the pointwise diagnostic error is bounded by
\[
  \max_{I,\sigma}|\widehat D_I^\sigma-D_I^\sigma|
  \le
  \frac{\mathfrak N_S(U)}{\sqrt m}(mh^2)B(Z_0)^3.
\]
\end{proposition}

\begin{proof}
Energy conservation and \eqref{eq:scaled-hessian} give
\begin{align}
\norm{Y_s}+\norm{W_s}
\le C\sqrt\kappa\,(\norm y+\norm w),
\qquad |s|\le S_\kappa^+.
\label{eq:stability-index-path-bound}
\end{align}
By definition, every block of the exact variational matrix is bounded by
\(\mathfrak A_S(U)\).  Direct differentiation of either endpoint diagnostic gives
\begin{align*}
\norm{\nabla F_s^\pm(y,w)}
\le
C\sqrt\kappa\bigl(1+\mathfrak A_S(U)\bigr)
(\norm y+\norm w).
\end{align*}
The logarithmic-Sobolev moment inequality therefore proves part (i) with
\eqref{eq:concentration-envelope}.

For the numerical comparison, \Cref{lem:local-defect} may be sharpened to
\begin{align*}
\norm{\widehat\Phi_{\bar h}(z)-\widetilde\flow_{\bar h}(z)}
\le
C\bigl(1+\kappa^2+\bar\gamma\bigr)
\bar h^3B(z)^2.
\end{align*}
On the whole-orbit energy event, every numerical state has phase norm at most
\(C\sqrt\kappa B(z_0)\).  The telescoping exact-flow propagation argument written out in the proof of \Cref{lem:global-LF-error} propagates each local defect
by exact flow for at most \(S_\kappa^+\) units of scaled time.  Each propagation costs
at most \(\mathfrak A_S(U)\), rather than a product of one-step Gronwall factors.
Summing at most \((S_\kappa^+)/\bar h\) defects gives
\begin{align}
\max_{|n|\bar h\le S_\kappa^+}
\norm{\widehat\Phi_{\bar h}^{\,n}z-
      \widetilde\flow_{n\bar h}z}
\le
C(1+S)\mathfrak A_S(U)\,
\kappa\bigl(1+\kappa^2+\bar\gamma\bigr)
\bar h^2B(z)^2.
\label{eq:stability-index-global-summary}
\end{align}
The endpoint diagnostics are bilinear in one momentum and one displacement.  Set
\[
  E
  =
  C(1+S)\mathfrak A_S(U)\,
  \kappa\bigl(1+\kappa^2+\bar\gamma\bigr)
  \bar h^2B(z_0)^2,
  \qquad
  R=C\sqrt\kappa\,B(z_0).
\]
By \eqref{eq:stability-index-global-summary}, each numerical endpoint differs from its
exact-flow counterpart by at most \(E\), and by
\eqref{eq:stability-index-path-bound} and the whole-orbit energy bound, every exact or
numerical endpoint has phase norm at most \(R\), after increasing the universal constant
in \(R\).  For example, writing the left endpoint diagnostic as a momentum paired with
the endpoint displacement and adding and subtracting the mixed exact--numerical pairing
gives
\[
  \bigl|\widehat F^-_{a,b}-F^-_{a,b}\bigr|
  \le
  E(2R)+R(2E)
  =4RE.
\]
The identical calculation with the right endpoint momentum gives the same bound for the
right diagnostic.  Consequently,
\[
  \max_{I,\sigma}
  \bigl|\widehat F_I^\sigma-F_I^\sigma\bigr|
  \le
  C(1+S)\mathfrak A_S(U)\,
  \kappa^{3/2}\bigl(1+\kappa^2+\bar\gamma\bigr)
  \bar h^2B(z_0)^3.
\]
Increasing \(C_0\) and replacing \(\mathfrak A_S(U)\) by
\(1+\mathfrak A_S(U)\) proves part~(ii).  
\end{proof}

\begin{corollary}[Selected-horizon curvature fallback]
\label{cor:curvature-only-fallback}
Under \eqref{eq:scaled-hessian},
\begin{align}
\mathfrak A_S(U)
\le
\sqrt\kappa\exp\!\left(\frac12\sqrt\kappa\,S_\kappa^+\right)
\le
e^{1/2}\sqrt\kappa\exp\!\left(\frac12S\sqrt\kappa\right).
\label{eq:curvature-only-A-bound}
\end{align}
Consequently, curvature alone yields polynomial tangent amplification whenever
\(S\sqrt\kappa=O(\log\kappa)\), and in particular whenever
\(S=O(\kappa^{-1/2})\).
\end{corollary}

\begin{proof}
For a variational solution
\(
\dot\xi=\eta,
\dot\eta=-H_s\xi
\)
with \(I_d\preceq H_s\preceq\kappa I_d\), put
\(
\mathcal E_s=\kappa\norm\xi^2+\norm\eta^2
\).
Then
\[
\mathcal E_s'
=2\ip{\eta}{(\kappa I_d-H_s)\xi}
\le\sqrt\kappa\,\mathcal E_s.
\]
Gronwall and comparison of the weighted and Euclidean norms prove the first inequality in
\eqref{eq:curvature-only-A-bound}.  The second follows from
\(\sqrt\kappa S_\kappa^+=\sqrt\kappa S+1\).
\end{proof}

\begin{lemma}[Moment concentration of endpoint diagnostics]
\label{lem:diagnostic-concentration}
Fix \(S>0\).  For every \(p\ge2\) and \(0\le t\le S/\sqrt m\),
\begin{align}
\norm{D_t^\pm-u(t)}_{L^p(\bar\pi)}
\le
\frac{\mathfrak C_S(U)}{\sqrt m}
\left(\sqrt{dp}+p\right).
\label{eq:diagnostic-Lp}
\end{align}
The same estimate holds for the endpoint diagnostics of any exact-flow interval \([a,b]\) with \(0\le b-a\le S/\sqrt m\).
\end{lemma}

\begin{proof}
We work in the variables \eqref{eq:dimensionless-variables}.  Let \(\widetilde\flow_s(y,w)=(Y_s,W_s)\).  The derivative matrix solves
\begin{align*}
\frac{\dd}{\dd s}D\widetilde\flow_s
=
\begin{pmatrix}
0&I_d\\
-\nabla^2\widetilde U(Y_s)&0
\end{pmatrix}
D\widetilde\flow_s.
\end{align*}
By the finite-time stability definition \eqref{eq:tangent-amplification},
\begin{align}
\norm{D\widetilde\flow_s}_{\mathrm{op}}
\le\mathfrak A_S(U),
\qquad 0\le s\le S.
\label{eq:flow-derivative-bound}
\end{align}
Energy conservation and \eqref{eq:scaled-hessian} imply
\begin{align*}
\frac12\norm{Y_s}^2+\frac12\norm{W_s}^2
\le \widetilde U(Y_s)+\frac12\norm{W_s}^2
=\widetilde U(Y_0)+\frac12\norm{W_0}^2 \le\frac\kappa2\norm{Y_0}^2+\frac12\norm{W_0}^2.
\end{align*}
Thus
\begin{align}
\norm{Y_s}+\norm{W_s}
\le C_\kappa(\norm{Y_0}+\norm{W_0}).
\label{eq:scaled-path-norm}
\end{align}
Define
\begin{align*}
F_s^-(y,w)=w^\top(Y_s-y),
\qquad\text{and}\qquad
F_s^+(y,w)=W_s^\top(Y_s-y).
\end{align*}
For the first function,
\begin{align*}
\nabla_yF_s^-
=(D_yY_s-I_d)^\top w,
\quad\text{and}\quad
\nabla_wF_s^-
=Y_s-y+(D_wY_s)^\top w.
\end{align*}
Equations \eqref{eq:flow-derivative-bound} and \eqref{eq:scaled-path-norm} yield
\begin{align}
\norm{\nabla F_s^-}
\le \mathfrak C_S(U)(\norm y+\norm w).
\label{eq:grad-Fminus-bound}
\end{align}
For \(F_s^+\), differentiating the product gives
\begin{align*}
\nabla_yF_s^+
&=(D_yW_s)^\top(Y_s-y)+(D_yY_s-I_d)^\top W_s,
\\
\nabla_wF_s^+
&=(D_wW_s)^\top(Y_s-y)+(D_wY_s)^\top W_s.
\end{align*}
The same two bounds imply
\begin{align}
\norm{\nabla F_s^+}
\le \mathfrak C_S(U)(\norm y+\norm w).
\label{eq:grad-Fplus-bound}
\end{align}

We next control the moments on the right.  Integration by parts and \(\nabla^2\widetilde U\succeq I_d\) give
\begin{align*}
\E\norm Y^2
\le\E[Y^\top\nabla\widetilde U(Y)]=d.
\end{align*}
Apply \Cref{prop:Lp-gradient} to the one-Lipschitz function \(y\mapsto\norm y\).  Since \(\E\norm Y\le\sqrt d\),
\begin{align}
\norm{\norm Y}_{L^p}\le\sqrt d+C\sqrt p.
\label{eq:Y-norm-Lp}
\end{align}
The same estimate holds for the Gaussian momentum \(W\).  Hence
\begin{align}
\norm{\norm Y+\norm W}_{L^p}
\le C(\sqrt d+\sqrt p).
\label{eq:phase-norm-Lp}
\end{align}
Applying \Cref{prop:Lp-gradient} to \(F_s^\pm\) and using \eqref{eq:grad-Fminus-bound}, \eqref{eq:grad-Fplus-bound}, and \eqref{eq:phase-norm-Lp},
\begin{align*}
\norm{F_s^\pm-\E F_s^\pm}_{L^p}
\le \mathfrak C_S(U)\sqrt p(\sqrt d+\sqrt p)
\le \mathfrak C_S(U)(\sqrt{dp}+p).
\end{align*}
Since
\(
D_t^\pm=m^{-1/2}F_{\sqrt mt}^\pm
\), this proves \eqref{eq:diagnostic-Lp}.

For an interval \([a,b]\), stationarity of exact Hamiltonian flow implies
\[
(X_a,V_a,X_b,V_b)
\overset d=
(X_0,V_0,X_{b-a},V_{b-a}).
\]
Its two endpoint diagnostics therefore have the same distributions as \(D_{b-a}^-\) and \(D_{b-a}^+\), proving the final statement.
\end{proof}

\begin{corollary}[Uniform concentration over a finite interval family]
\label{cor:uniform-diagnostic-concentration}
Let \(\mathscr I\) be a finite family of exact-flow intervals, each of duration at most \(S/\sqrt m\), and let \(N=|\mathscr I|\).  For \(\delta\in(0,1)\), put
\begin{align}
p=\max\left\{2,\log\frac{4N}{\delta}\right\}.
\label{eq:p-union}
\end{align}
Then, with probability at least \(1-\delta\),
\begin{align*}
\max_{I\in\mathscr I}
\max_{\sigma\in\{-,+\}}
\abs{D_I^\sigma-u(\ell(I))}
\le
\frac{e\mathfrak C_S(U)}{\sqrt m}
\left(\sqrt{dp}+p\right),
\end{align*}
where, for a continuous-time interval \(I=[a,b]\), \(\ell(I)=b-a\) denotes its physical duration.
\end{corollary}

\begin{proof}
For each centered diagnostic \(Z\), \Cref{lem:diagnostic-concentration} gives \(\norm Z_{L^p}\le A_p\), where
\[
A_p=\frac{\mathfrak C_S(U)}{\sqrt m}(\sqrt{dp}+p).
\]
Markov's inequality yields
\(
\Pp(|Z|>eA_p)\le e^{-p}
\).
There are \(2N\) diagnostics.  By \eqref{eq:p-union}, \(2Ne^{-p}\le\delta/2<\delta\), so a union bound proves the claim.
\end{proof}

\subsection{A deterministic leapfrog-to-flow comparison}
\label{sec:lf-flow-comparison}

The profile identities concern exact Hamiltonian flow, whereas the kernels in \Cref{alg:mul,alg:bps} use leapfrog.  We next compare every numerical endpoint diagnostic in a bounded orbit with its exact-flow counterpart.  The proof is carried out in the dimensionless variables \eqref{eq:dimensionless-variables}.  Let $\bar h=\sqrt m\,h,$ and $\bar\gamma=\gamma\kappa^{3/2}$. Then \eqref{eq:hessian-lip} becomes $\norm{\nabla^2\widetilde U(y)-\nabla^2\widetilde U(y')}_{\mathrm F}
\le \bar\gamma\norm{y-y'}.$ In particular, for all vectors \(a,b\),
\begin{align}
\norm{\nabla^3\widetilde U(y)[a,b]}
\le \bar\gamma\norm a\norm b.
\label{eq:scaled-third-vector}
\end{align}
Indeed, for every unit vector \(c\),
\[
\abs{\nabla^3\widetilde U(y)[a,b,c]}
\le
\norm{\nabla^3\widetilde U(y)[a,\cdot,\cdot]}_{\mathrm F}
\norm b\norm c.
\]

Let \(\widehat\Phi_{\bar h}\) be one leapfrog step for the scaled Hamiltonian $\widetilde H(y,w)=\widetilde U(y)+\frac12\norm w^2,$ and let \(\widetilde\flow_s\) be its exact flow.  Set $B(y,w)=1+\norm y+\norm w.$

\begin{lemma}[Local defect]
\label{lem:local-defect}
There are universal constants \(c,C>0\) such that, whenever
\(0<\bar h\le c/(1+\sqrt\kappa)\),
\begin{align}
\norm{\widehat\Phi_{\bar h}(z)-\widetilde\flow_{\bar h}(z)}
\le
C\bigl(1+\kappa^2+\bar\gamma\bigr)
\bar h^3B(z)^2,
\qquad z=(y,w).
\label{eq:local-defect}
\end{align}
\end{lemma}

\begin{proof}
Write \(g=\nabla\widetilde U\), let
\((Y_t,W_t)=\widetilde\flow_t(y,w)\), and abbreviate
\(h=\bar h\) and \(B=B(y,w)\). Since the minimizer of
\(\widetilde U\) is the origin, \eqref{eq:scaled-hessian} gives $  \norm{g(x)}\le\kappa\norm x.$ Reduce the universal constant \(c\) in the step-size assumption so that
\begin{align}
  h\le1,
  \qquad
  \kappa h^2\le\frac12.
  \label{eq:local-defect-small-step}
\end{align}
The exact position satisfies $Y_t=y+tw-\int_0^t(t-s)g(Y_s)\dd s.$ If \(M_h=\sup_{0\le t\le h}\norm{Y_t}\), then
\[
  M_h
  \le
  \norm y+h\norm w+\frac{\kappa h^2}{2}M_h
  \le
  B+\frac12M_h,
\]
so
\begin{align}
  \sup_{0\le t\le h}\norm{Y_t}\le2B.
  \label{eq:local-position-bound}
\end{align}
Consequently,
\begin{align}
  \norm{W_t}
  \le
  \norm w+\int_0^t\norm{g(Y_s)}\dd s
  \le
  B+2\kappa hB,
  \qquad 0\le t\le h,
  \label{eq:local-momentum-bound}
\end{align}
and
\begin{align}
  \norm{Y_t-y}
  \le
  tB+\kappa t^2B
  \le
  t(1+\kappa h)B.
  \label{eq:local-displacement-bound}
\end{align}

The leapfrog position is
\[
  \widehat y=y+hw-\frac{h^2}{2}g(y),
\]
whereas the exact integral equation gives $  Y_h=y+hw-\int_0^h(h-s)g(Y_s)\dd s.$ Using \eqref{eq:local-displacement-bound} and the
\(\kappa\)-Lipschitz property of \(g\),
\begin{align}
  \norm{\widehat y-Y_h}\le
  \kappa\int_0^h(h-s)\norm{Y_s-y}\dd s \le
  \frac{\kappa h^3}{6}(1+\kappa h)B
  \le
  C\kappa^2h^3B^2.
  \label{eq:position-local-defect}
\end{align}

For the momentum, $  \widehat w=w-\frac h2\{g(y)+g(\widehat y)\},$ whereas $  W_h=w-\int_0^hg(Y_t)\dd t.$ Set \(f(t)=g(Y_t)\). Adding and subtracting \(g(Y_h)\) gives
\begin{align}
  \norm{\widehat w-W_h}
  \le
  \frac{\kappa h}{2}\norm{\widehat y-Y_h} +
  \left\|
    \frac h2\{f(0)+f(h)\}-\int_0^hf(t)\dd t
  \right\|.
  \label{eq:momentum-defect-decomposition}
\end{align}

First suppose that \(\kappa h\le1\). Then
\eqref{eq:local-momentum-bound} gives \(\norm{W_t}\le3B\), while
\eqref{eq:local-position-bound} gives
\(\norm{g(Y_t)}\le2\kappa B\). Since
\begin{align*}
  f''(t)
  =
  \nabla^3\widetilde U(Y_t)[W_t,W_t]
  -\nabla^2\widetilde U(Y_t)g(Y_t),
\end{align*}
we obtain from \eqref{eq:scaled-third-vector}, $  \sup_{0\le t\le h}\norm{f''(t)}
  \le
  C(\kappa^2+\bar\gamma)B^2.$ The exact trapezoidal identity
\begin{align*}
  \frac h2\{f(0)+f(h)\}-\int_0^hf(t)\dd t
  =
  \frac12\int_0^h t(h-t)f''(t)\dd t
\end{align*}
therefore yields
\begin{align}
  \left\|
    \frac h2\{f(0)+f(h)\}-\int_0^hf(t)\dd t
  \right\|
  \le
  C(\kappa^2+\bar\gamma)h^3B^2.
  \label{eq:trapezoidal-small-kappah}
\end{align}

Now suppose that \(\kappa h>1\). Integration by parts gives
\[
  \frac h2\{f(0)+f(h)\}-\int_0^hf(t)\dd t
  =
  \int_0^h\left(t-\frac h2\right)f'(t)\dd t.
\]
Since \(f'(t)=\nabla^2\widetilde U(Y_t)W_t\),
\eqref{eq:local-momentum-bound} implies $\norm{f'(t)}
  \le
  \kappa B(1+2\kappa h).$ Hence
\begin{align}
  \left\|
    \frac h2\{f(0)+f(h)\}-\int_0^hf(t)\dd t
  \right\|
  \le
  \frac{\kappa h^2}{4}(1+2\kappa h)B
  \le
  C\kappa^2h^3B^2,
  \label{eq:trapezoidal-large-kappah}
\end{align}
where the last step uses \(\kappa h>1\) and \(B\ge1\).

Finally, the sharper first bound in
\eqref{eq:position-local-defect} and
\eqref{eq:local-defect-small-step} give
\begin{align*}
  \frac{\kappa h}{2}\norm{\widehat y-Y_h}
  \le
  C\kappa^2h^4(1+\kappa h)B
  \le
  C\kappa^2h^3B^2,
\end{align*}
because \(h(1+\kappa h)=h+\kappa h^2\le3/2\). Combining
\eqref{eq:momentum-defect-decomposition},
\eqref{eq:trapezoidal-small-kappah}, and
\eqref{eq:trapezoidal-large-kappah} proves
\begin{align*}
  \norm{\widehat w-W_h}
  \le
  C(1+\kappa^2+\bar\gamma)h^3B^2.
\end{align*}
Together with \eqref{eq:position-local-defect}, this is
\eqref{eq:local-defect}.
\end{proof}

\begin{lemma}[Energy-conditioned global phase and diagnostic error]
\label{lem:global-LF-error}
Fix \(S>0\).  Let \(0<\bar h\le c/(1+\sqrt\kappa)\), let
\(N\bar h\le S_\kappa^+\), and suppose that
\begin{align}
\max_{|j|\le N}
\abs{\widetilde H(\widehat\Phi_{\bar h}^{\,j}z_0)-\widetilde H(z_0)}
\le1.
\label{eq:LF-energy-window-assumption}
\end{align}
Then
\begin{align}
\max_{|n|\le N}
\norm{\widehat\Phi_{\bar h}^{\,n}(z_0)-
      \widetilde\flow_{n\bar h}(z_0)}
\le
C(1+S)\mathfrak A_S(U)\,
\kappa\bigl(1+\kappa^2+\bar\gamma\bigr)
\bar h^2B(z_0)^2.
\label{eq:global-phase-error}
\end{align}
Moreover, for every pair \(-N\le a<b\le N\),
\begin{align}
\max_{\sigma\in\{-,+\}}
\abs{\widehat F_{a,b}^{\sigma}-F_{a,b}^{\sigma}}
\le
\mathfrak N_S(U)\bar h^2B(z_0)^3.
\label{eq:global-diagnostic-error-scaled}
\end{align}
In the original variables,
\begin{align}
\max_{\sigma\in\{-,+\}}
\abs{\widehat D_{a,b}^{\sigma}-D_{a,b}^{\sigma}}
\le
\frac{\mathfrak N_S(U)}{\sqrt m}
(mh^2)B(z_0)^3.
\label{eq:global-diagnostic-error}
\end{align}
\end{lemma}

\begin{proof}
The energy window and strong convexity imply
\begin{align}
\max_{|j|\le N}B(\widehat\Phi_{\bar h}^{\,j}z_0)
\le C\sqrt\kappa\,B(z_0).
\label{eq:numerical-B-path}
\end{align}
Indeed, \(\widetilde U(y)\ge\norm y^2/2\), while
\(\widetilde U(y_0)\le\kappa\norm{y_0}^2/2\), and the additive energy tolerance is
absorbed by the constant in \(B\).

Fix \(n\in\{1,\ldots,N\}\), put
\(\widehat z_j=\widehat\Phi_{\bar h}^{\,j}z_0\), and define
\[
G_j=\widetilde\flow_{(n-j)\bar h}(\widehat z_j),
\qquad 0\le j\le n.
\]
Then \(G_0=\widetilde\flow_{n\bar h}(z_0)\), \(G_n=\widehat z_n\), and
\begin{align*}
G_{j+1}-G_j
=
\widetilde\flow_{(n-j-1)\bar h}
   (\widehat\Phi_{\bar h}(\widehat z_j))
-
\widetilde\flow_{(n-j-1)\bar h}
   (\widetilde\flow_{\bar h}(\widehat z_j)).
\end{align*}
The mean-value theorem and \eqref{eq:tangent-amplification} give
\[
\norm{G_{j+1}-G_j}
\le
\mathfrak A_S(U)
\norm{\widehat\Phi_{\bar h}(\widehat z_j)-
      \widetilde\flow_{\bar h}(\widehat z_j)}.
\]
By \Cref{lem:local-defect} and \eqref{eq:numerical-B-path}, the right side is at most
\[
C\mathfrak A_S(U)\,
\kappa\bigl(1+\kappa^2+\bar\gamma\bigr)
\bar h^3B(z_0)^2.
\]
Summing over \(j\), using \(n\bar h\le S_\kappa^+\le S+1\), and applying
reversibility for negative \(n\), proves \eqref{eq:global-phase-error}.

Set $E_N=C(1+S)\mathfrak A_S(U)\, \kappa\bigl(1+\kappa^2+\bar\gamma\bigr) \bar h^2B(z_0)^2,$ and $R_N=C\sqrt\kappa\,B(z_0).$ Equation \eqref{eq:global-phase-error} bounds the error at each endpoint by \(E_N\). Exact energy conservation and \eqref{eq:numerical-B-path} bound every exact and numerical
endpoint phase norm by \(R_N\), after increasing the universal constant.  For the left
endpoint diagnostic, add and subtract the pairing of the exact left momentum with the
numerical displacement.  Then
\begin{align*}
  \bigl|\widehat F^-_{a,b}-F^-_{a,b}\bigr|
  &\le
  \norm{\widehat W_a-W_a}\,
  \norm{\widehat Y_b-\widehat Y_a}
  +
  \norm{W_a}\,
  \norm{(\widehat Y_b-Y_b)-(\widehat Y_a-Y_a)}
  \\
  &\le
  E_N(2R_N)+R_N(2E_N)
  =4R_NE_N.
\end{align*}
Replacing the left momentum by the right momentum gives the same estimate for
\(\sigma=+\).  Hence
\[
  \max_{\sigma\in\{-,+\}}
  \bigl|\widehat F_{a,b}^{\sigma}-F_{a,b}^{\sigma}\bigr|
  \le
  C(1+S)\mathfrak A_S(U)\,
  \kappa^{3/2}\bigl(1+\kappa^2+\bar\gamma\bigr)
  \bar h^2B(z_0)^3
  \le
  \mathfrak N_S(U)\bar h^2B(z_0)^3,
\]
after increasing the universal constant \(C_0\) in
\eqref{eq:numerical-envelope}.  This proves
\eqref{eq:global-diagnostic-error-scaled}.  Finally,
\(\bar h^2=mh^2\), and each scaled endpoint diagnostic is \(\sqrt m\) times its
original-coordinate counterpart.  Dividing the preceding inequality by \(\sqrt m\)
proves \eqref{eq:global-diagnostic-error}.

\end{proof}

\subsection{Short checked intervals}

For a maximum orbit size \(K\), let \(\mathscr I_K\) be the family of all integer intervals that can occur as a completed orbit or a recursively inspected dyadic suborbit under some realization of the doubling bits, with terminal size at most \(K\).  Every endpoint lies in \([-K+1:K-1]\), so
\begin{align*}
|\mathscr I_K|\le(2K-1)^2\le4K^2.
\end{align*}

\begin{lemma}[Typical numerical endpoints on the energy event]
\label{lem:typical-numerical-endpoints}
Fix finite \(\kappa,S\).  There are constants \(b_0,B_0,c_0,C_0>0\),
polynomial in \(\kappa\), \(S\), \(\mathfrak A_S(U)\), and \(\bar\gamma\), such that the following
holds.  Let \(K\bar h\le S_\kappa^+\), let $p\ge\max\{2,\log(8K/\delta)\},$ assume \(d\ge C_0p\), and suppose
\begin{align}
\mathfrak N_S(U)\bar h^2(d+p)
\le c_0\sqrt d.
\label{eq:endpoint-error-small}
\end{align}
If \(Z_0\sim\widetilde\pi\), then the probability that both the energy event
\eqref{eq:LF-energy-window-assumption} holds and some numerical endpoint
\((\widehat Y_n,\widehat W_n)\), \(|n|\le K\), violates
\begin{align}
\norm{\widehat W_n}\ge b_0\sqrt d,
\qquad
\norm{\nabla\widetilde U(\widehat Y_n)}\le B_0\sqrt d
\label{eq:typical-numerical-endpoints}
\end{align}
is at most \(\delta\).
\end{lemma}

\begin{proof}
For each fixed \(n\), stationarity of exact Hamiltonian flow gives
\((Y_{n\bar h},W_{n\bar h})\sim\widetilde\pi\).  Gaussian concentration yields
universal \(b_1,c_1>0\) such that
\[
\Pp(\norm W<b_1\sqrt d)\le e^{-c_1d}.
\]
By \eqref{eq:Y-norm-Lp} and Markov's inequality with moment order \(p\),
\[
\Pp(\norm Y>C(\sqrt d+\sqrt p))\le e^{-p}.
\]
Since \(\norm{\nabla\widetilde U(Y)}\le\kappa\norm Y\), a union bound over the
\(2K+1\) exact endpoint times gives, with probability at least \(1-\delta/2\),
\begin{align}
\min_{|n|\le K}\norm{W_{n\bar h}}\ge2b_0\sqrt d,
\qquad
\max_{|n|\le K}\norm{\nabla\widetilde U(Y_{n\bar h})}
\le\frac{B_0}{2}\sqrt d.
\label{eq:exact-endpoint-typicality}
\end{align}
The dimension condition makes the Gaussian lower-tail contribution smaller than
\(\delta/4\).

The event \(B(Z_0)\le C(\sqrt d+\sqrt p)\) fails with probability at most
\(\delta/2\).  On its intersection with the numerical energy event,
\Cref{lem:global-LF-error} gives
\[
\max_{|n|\le K}
\norm{(\widehat Y_n,\widehat W_n)-(Y_{n\bar h},W_{n\bar h})}
\le
\mathfrak N_S(U)\bar h^2(d+p).
\]
Under \eqref{eq:endpoint-error-small}, this is at most
\(\min\{b_0,B_0/(2\kappa)\}\sqrt d\) after reducing \(c_0\).  The momentum
conclusion follows from \eqref{eq:exact-endpoint-typicality}; the gradient conclusion
follows from \(\kappa\)-Lipschitz continuity of \(\nabla\widetilde U\).
\end{proof}

Choose
\begin{align}
S_{\mathrm{sm}}
=
\frac{1}{128(1+B_0/b_0)},
\qquad
T_{\mathrm{sm}}=\frac{S_{\mathrm{sm}}}{\sqrt L}.
\label{eq:small-time-cutoff-general}
\end{align}

\begin{lemma}[Small checked intervals cannot turn]
\label{lem:adaptive-small-intervals}
On the event \eqref{eq:typical-numerical-endpoints}, every nontrivial checked leapfrog interval with physical duration at most \(T_{\mathrm{sm}}\) has both endpoint diagnostics strictly positive.  Singleton intervals have zero displacement and hence U-turn indicator zero.
\end{lemma}

\begin{proof}
Use the left endpoint of the checked interval as the initial state in \Cref{lem:no-uturn}.  In physical variables,
\[
\norm{\nabla U(x)}
=\sqrt m\norm{\nabla\widetilde U(y)}
\le\sqrt m B_0\sqrt d,
\]
whereas
\[
\sqrt L\norm v
=\sqrt{m\kappa}\norm w
\ge\sqrt{m\kappa}\,b_0\sqrt d.
\]
Thus the gradient--momentum ratio in \Cref{lem:no-uturn} holds with \(B=B_0/(b_0\sqrt\kappa)\le B_0/b_0\).  The definition \eqref{eq:small-time-cutoff-general} gives
\[
\sqrt L\,T_{\mathrm{sm}}
\le\frac{1}{64(B+1)}.
\]
The shifted deterministic lemma proves strict positivity of both endpoint diagnostics.
\end{proof}

\subsection{Intrinsic terminal-depth certificate and proof of the main theorem}
\label{sec:intrinsic-master-proof}

\begin{proposition}[Intrinsic terminal-depth certificate]
\label{prop:intrinsic-certificate}
Assume the trajectory hypotheses of \Cref{thm:adaptive}, including
\eqref{eq:intrinsic-margin-condition}.  Let \(\mathcal C_\star\) be the phase-space event on which
\begin{enumerate}[label=(\roman*)]
\item every checked exact diagnostic differs from its profile value by at most
\[
e\,\mathfrak C_p(S;U)\frac{\sqrt{dp}+p}{\sqrt m};
\]
\item the energy window \(\mathcal H_{h,K_\star}(\log2)\) holds;
\item every checked numerical diagnostic differs from its exact counterpart by at most
\[
e\,\mathfrak E_{p,h}(K_\star,\log2;U)\frac d{\sqrt m};
\]
\item every nontrivial checked interval shorter than \(T_{\mathrm{sm}}\) has both numerical endpoint
 diagnostics strictly positive; singleton intervals are handled separately by their identically zero U-turn indicator.
\end{enumerate}
Then
\begin{align}
\bar\pi(\mathcal C_\star^c)\le\zeta_{p,h}(K_\star).
\label{eq:intrinsic-certificate-failure}
\end{align}
On \(\mathcal C_\star\), every nontrivial proper checked dyadic interval below depth \(k_\star\)
 has both numerical endpoint diagnostics strictly positive, while every singleton has U-turn indicator zero, every possible full interval at
 depth \(k_\star\) has both numerical endpoint diagnostics strictly negative, and every
 realization of the doubling bits returns cardinality \(K_\star\) through the U-turn rule
 before the cap.
\end{proposition}

\begin{proof}
For a fixed \(I\in\mathscr I_{K_\star}\), stationarity and time translation of exact
 Hamiltonian flow imply that \(D_I^\sigma\) has the same distribution as
 \(D_{\ell_h(I)}^\sigma\).  By \eqref{eq:intrinsic-Cp} and Markov's inequality,
\[
\bar\pi\left(
 \abs{D_I^\sigma-u(\ell_h(I))}>
 e\,\mathfrak C_p(S;U)\frac{\sqrt{dp}+p}{\sqrt m}
\right)\le e^{-p}.
\]
A union bound over the two endpoint diagnostics and the at most \(N_{K_\star}\) checked
 intervals gives failure probability at most \(2N_{K_\star}e^{-p}\).  By
 \eqref{eq:intrinsic-Ep} and Markov's inequality, the numerical-fidelity bound in (iii)
 fails on the energy event with probability at most \(e^{-p}\).  The remaining two
 failures have probabilities \(\delta_H(h,K_\star,\log2)\) and
 \(\delta_{\mathrm{sm}}(h,K_\star,T_{\mathrm{sm}})\), respectively.  This proves
 \eqref{eq:intrinsic-certificate-failure}.

On \(\mathcal C_\star\), \eqref{eq:intrinsic-Xi} and
 \eqref{eq:intrinsic-margin-condition} imply
\begin{align}
\max_{I\in\mathscr I_{K_\star}}
\max_{\sigma\in\{-,+\}}
\abs{\widehat D_I^\sigma-u(\ell_h(I))}
\le\frac{\rho_0}{4}\frac d{\sqrt m}
\label{eq:intrinsic-sign-transfer}
\end{align}
for every checked interval of duration at least \(T_{\mathrm{sm}}\).  If such an interval
 has depth \(r<k_\star\), then \eqref{eq:profile-preterminal-margin} and
 \eqref{eq:intrinsic-sign-transfer} give
\[
\widehat D_I^\sigma\ge\frac{3\rho_0}{4}\frac d{\sqrt m}>0.
\]
Nontrivial intervals shorter than \(T_{\mathrm{sm}}\) are positive by (iv), while singleton indicators are zero by definition.  At depth
 \(k_\star\), \eqref{eq:profile-terminal-margin} and
 \eqref{eq:intrinsic-sign-transfer} give
\[
\widehat D_I^\sigma\le-\frac{3\rho_0}{4}\frac d{\sqrt m}<0
\]
for every possible full interval.  The conclusion follows from
 \Cref{prop:terminal-depth}.
\end{proof}

\begin{proof}[Proof of \Cref{thm:adaptive}]
Let \(\mathcal C_\star\) be the event in \Cref{prop:intrinsic-certificate}, and put
\[
s=\frac{\epsilon^2}{128M^2}.
\]
For selector \(a\in\{\mathrm M,\mathrm B\}\), choose
\begin{align}
\delta_a=\frac{3\underline a_a}{32}
\label{eq:intrinsic-delta-a}
\end{align}
and define the conditional phase-space failure
\[
r_a(x)=\Pp_V(\mathcal C_\star^c\mid X_0=x),
\qquad
D_a=\{x:r_a(x)\le\delta_a\},
\qquad
G_{a,x}=\{v:(x,v)\in\mathcal C_\star\}.
\]
Fubini's theorem, Markov's inequality, and \eqref{eq:intrinsic-certificate-failure} yield
\begin{align*}
\sup_{x\in D_a}\Pp_V(G_{a,x}^c)\le\delta_a,
\qquad
\pi(D_a^c)\le\frac{\zeta_{p,h}(K_\star)}{\delta_a}.
\end{align*}
After fixing the universal constant \(c_0\) in
 \eqref{eq:intrinsic-failure-condition} sufficiently small, the second inequality implies
\begin{align}
\pi(D_a^c)\le\frac{\theta_\star s}{64}.
\label{eq:intrinsic-D-tail}
\end{align}
By \Cref{prop:intrinsic-certificate}, every momentum in \(G_{a,x}\) is
 \((K_\star,\log2)\)-certified and termination is caused by a genuine U-turn.

Since \(k_\star\ge2\),
\begin{align}
T_\star<K_\star h\le\frac43T_\star.
\label{eq:intrinsic-Kh-T}
\end{align}
The terminal profile is negative, while \eqref{eq:profile-positive} is positive on
 \((0,\pi/(2\sqrt L)]\).  Hence
\begin{align}
T_\star>\frac{\pi}{2\sqrt L}.
\label{eq:intrinsic-terminal-lower}
\end{align}
After decreasing \(c_{\mathrm{ov}}\), equations
 \eqref{eq:intrinsic-overlap-time}, \eqref{eq:intrinsic-Kh-T}, and
 \eqref{eq:intrinsic-terminal-lower} give \(0<r_\star\le1/4\).  The resolution
 condition \eqref{eq:intrinsic-resolution} gives \(r_\star K_\star\ge8\).  Define
\[
\mathcal J_\star
=\left\{j:\left\lceil\frac{r_\star K_\star}{2}\right\rceil
\le j\le\left\lfloor r_\star K_\star\right\rfloor\right\}.
\]
By \Cref{lem:short-band-mass}, its multinomial and BPS masses are at least
 \(\underline a_{\mathrm M}\) and \(\underline a_{\mathrm B}\), respectively.
Moreover, every \(j\in\mathcal J_\star\) satisfies
\[
\frac{t_{\mathrm{ov}}}{2}\le jh\le t_{\mathrm{ov}}.
\]
The fixed-index overlap estimate \Cref{prop:overlap-input}, with
 \(c_{\mathrm{ov}}\) sufficiently small, therefore gives
\begin{align*}
\ov(Q_{j,x},Q_{j,y})\ge\frac34
\end{align*}
whenever \(x,y\in D_a\) and \(\norm{x-y}\le t_{\mathrm{ov}}/128\).

Apply \Cref{thm:terminal-transfer} with
\[
K=K_\star,
\quad
\eta=\log2,
\quad
\Delta=\frac{t_{\mathrm{ov}}}{128},
\quad
\beta=\frac34,
\quad
\mathcal J=\mathcal J_\star.
\]
The choice \eqref{eq:intrinsic-delta-a} gives
\[
a_{a,K_\star}\beta-2\delta_a
\ge\frac{9}{16}\underline a_a,
\]
so \(\rho_a\ge c\underline a_a\).  Together with
 \eqref{eq:intrinsic-D-tail}, this gives
\begin{align}
\Phi_s(R_a)\ge c\underline a_a\theta_\star.
\label{eq:intrinsic-conductance}
\end{align}
By \eqref{eq:intrinsic-Kh-T},
\begin{align*}
r_\star
\ge\frac{c}{a_\star\sqrt\kappa(1+\gamma)^{1/3}},
\qquad
\theta_\star
\ge\frac{c}{\sqrt\kappa(1+\gamma)^{1/3}}.
\end{align*}
For multinomial selection, \(\underline a_{\mathrm M}\ge cr_\star\); for BPS,
 \(\underline a_{\mathrm B}\ge cr_\star^2\).  Hence
 \[
   \underline a_{\mathrm M}\theta_\star
   \ge
   \frac{c}{a_\star\kappa(1+\gamma)^{2/3}},
   \qquad
   \underline a_{\mathrm B}\theta_\star
   \ge
   \frac{c}{a_\star^2\kappa^{3/2}(1+\gamma)}.
 \]
 Therefore \(\Phi_s(R_{\mathrm M})^{-2}\) is bounded by a constant multiple of
 \(a_\star^2\kappa^2(1+\gamma)^{4/3}\), while
 \(\Phi_s(R_{\mathrm B})^{-2}\) is bounded by a constant multiple of
 \(a_\star^4\kappa^3(1+\gamma)^2\).  Substitution into
 \eqref{eq:intrinsic-conductance}, followed by \Cref{prop:PSD-mixing}, proves
 \eqref{eq:intrinsic-transitions-M}--
 \eqref{eq:intrinsic-tv}. On the certification event, the returned orbit contains exactly
 \(K_\star\) leapfrog states, and \eqref{eq:intrinsic-K} is the identity
 \(K_\star=1+T_\star/h\).  The deterministic cap-aware work bound
 \eqref{eq:intrinsic-work-cap} and the refined expected and
 high-probability bounds follow from \Cref{prop:exceptional-work}.
\end{proof}

\section{Non-Gaussian specializations}
\label{sec:non-gaussian-specializations}

The main theorem is formulated in terms of profile separation and
intrinsic diagnostic stability.  This section gives several non-Gaussian
settings in which the population profile or the intrinsic stability
conditions can be verified.  The results serve different purposes.
The spectral criterion below establishes that the stationary profile
must eventually become negative.  The product and near-isotropic
results provide concrete profile-separated classes.  The tangent-flow
and metric-adaptation results give sufficient conditions for controlling
the diagnostic concentration and numerical-fidelity terms appearing in
\Cref{thm:adaptive}.  However, we emphasize that none of these conditions is asserted to be
necessary.

\subsection{A spectral criterion for profile crossing}
\label{sec:non-gaussian-spectral-crossing}

Let \(\bar\pi=\pi\otimes N(0,I_d)\), let \(\flow_t\) denote exact
Hamiltonian flow, and recall the stationary profile
\[
  u(t)
  =
  \E_{\bar\pi}\!\left[
    V_0^\top(X_t-X_0)
  \right]
  =
  \frac12\frac{\dd}{\dd t}
  \E_{\bar\pi}\norm{X_t-X_0}^2.
\]
The following theorem represents \(u\) through the spectral measure of
the position observables under the Hamiltonian Koopman group.

\begin{theorem}[Koopman representation and positive-frequency crossing]
\label{thm:koopman-crossing}
There is a finite positive measure \(\nu_U\) on \([0,\infty)\) such
that
\begin{align}
  \sum_{i=1}^d
  \E\!\left[
    (X_{t,i}-\E X_i)(X_{0,i}-\E X_i)
  \right]
  =
  \int_{[0,\infty)}
  \cos(\omega t)\,\nu_U(\dd\omega),
  \label{eq:koopman-correlation}
\end{align}
and
\begin{align}
  u(t)
  =
  \int_{(0,\infty)}
  \omega\sin(\omega t)\,\nu_U(\dd\omega),
  \qquad
  \int_{[0,\infty)}
  \omega^2\,\nu_U(\dd\omega)
  =
  d.
  \label{eq:koopman-profile}
\end{align}
If, for some \(\omega_0>0\),
\begin{align}
  \supp\nu_U
  \subset
  \{0\}\cup[\omega_0,\infty),
  \label{eq:position-spectral-gap}
\end{align}
then $  u(t)<0$ for at least one $  0<t<\frac{2\pi}{\omega_0}.$
\end{theorem}

Theorem~\ref{thm:koopman-crossing} concerns the population profile.  It
does not by itself certify a practical NUTS tree: one must additionally
show that the dyadic grid encounters a negative profile value with a
positive margin and that the exact and numerical diagnostics remain
within that margin.  The proof is given in
\Cref{app:koopman-proof}.

\subsection{Independent product targets}
\label{sec:non-gaussian-products}

The spectral condition in
\eqref{eq:position-spectral-gap} can be verified for every
one-dimensional strongly convex oscillator.  This gives a nonlinear
product class with arbitrary finite condition number.

\begin{corollary}[Profile and exact-diagnostic verification for nonlinear independent products]
\label{cor:product-arbitrary-kappa}
Let \(W\in C^3(\R)\) satisfy
\begin{align}
  m\le W''(q)\le L,
  \qquad
  q\in\R,
  \label{eq:one-dimensional-curvature}
\end{align}
and suppose that \(W''\) is Lipschitz.  For the one-dimensional
equilibrium Hamiltonian process, the profile \(u_W\) is negative at
some time before \(2\pi/\sqrt m\).

For
\begin{align}
  U_d(x)
  =
  \sum_{i=1}^d W(x_i),
  \label{eq:iid-product-potential}
\end{align}
the stationary profile satisfies $  u_d(t)=d\,u_W(t).$ Moreover, for every fixed dimensionless horizon \(S<\infty\), there is
a constant \(C_{W,S}<\infty\), independent of \(d\), such that
\begin{align}
  \mathfrak C_p(S;U_d)
  \le
  C_{W,S},
  \qquad
  p\ge2.
  \label{eq:product-Cp}
\end{align}

If the zero set of \(u_W\) has empty interior before its first negative
component---in particular, if \(u_W\) is real analytic and not
identically zero---then a positive-Lebesgue-measure nonresonant set of
sufficiently small step sizes admits a profile-separated depth with
target-dependent constants independent of \(d\).

For every such dyadic grid, \Cref{thm:adaptive} applies to either
selector once the scalar numerical-fidelity condition
\eqref{eq:intrinsic-margin-condition} and the whole-orbit energy condition
\eqref{eq:intrinsic-failure-condition} are verified.  Thus exact-flow
diagnostic concentration for the product class is dimension free.  No
uniform curvature-only numerical-fidelity claim is made.
\end{corollary}

The proof is given in \Cref{app:product-proof}.

\subsection{Near-isotropic coupled targets}
\label{sec:near-isotropic-specialization}

The next result is a nonlinear perturbation of the canonical Gaussian
profile.  It shows that a globally near-isotropic Hessian preserves the
sine-profile mechanism used in Gaussian analyses of NUTS.

\begin{theorem}[Near-isotropic nonlinear targets]
\label{thm:near-isotropic}
Suppose that, for some \(\omega>0\) and \(\eta\in[0,1)\),
\begin{align*}
  \norm{
    \omega^{-2}\nabla^2U(x)-I_d
  }_{\mathrm{op}}
  \le
  \eta,
  \qquad
  x\in\R^d.
\end{align*}
Then, for \(0\le\omega t\le2\pi\),
\begin{align}
  \abs{
    \omega u(t)-d\sin(\omega t)
  }
  \le
  C_0\eta d,
  \label{eq:near-isotropic-profile}
\end{align}
where one may take \(C_0=2\pi e^{2\pi}\).

Consequently, suppose that a dyadic time grid has a positive reference sine margin
\(\rho_{\sin}\) at every preterminal duration above the short-time cutoff and has
reference sine value at most \(-\rho_{\sin}\) at its candidate terminal duration.  If
\(\rho_{\sin}>C_0\eta\), then the grid is profile separated with the normalized margin
\[
  \rho_0=\sqrt{1-\eta}\,\bigl(\rho_{\sin}-C_0\eta\bigr),
\]
because \(m=(1-\eta)\omega^2\) in \Cref{def:profile-separated}.
\end{theorem}

The theorem verifies the population profile condition.  Application of
\Cref{thm:adaptive} additionally requires the intrinsic concentration,
leapfrog-fidelity, energy-window, and short-interval conditions stated
there.  The proof is given in
\Cref{app:near-isotropic-proof}.

\subsection{Polynomial diagnostic stability for coupled targets}
\label{sec:polynomial-diagnostic-stability}

We next record a sufficient route from finite-time stability of the
Hamiltonian flow to the intrinsic quantities in
\Cref{thm:adaptive}.  This route is useful for coupled nonlinear
targets, but it is deliberately separated from the main theorem:
the theorem itself requires only the scalar diagnostic moduli.

Recall the selected-horizon quantities
\(S_\kappa^+\), \(\mathfrak A_S(U)\), \(\mathfrak C_S(U)\), and
\(\mathfrak N_S(U)\) from
\eqref{eq:tangent-amplification}--\eqref{eq:numerical-envelope}.  By
\Cref{prop:stability-envelope},
\[
  \mathfrak C_p(S;U)\le\mathfrak C_S(U),
\]
and, on the \(\log2\) energy event, the same proof with a fixed energy width
changes only the universal constant and gives
\begin{align}
  \mathfrak E_{p,h}(K,\log2;U)
  \le
  C\mathfrak N_S(U)(mh^2)\frac{(d+p)^{3/2}}{d}.
  \label{eq:tangent-implies-Ep}
\end{align}

For a profile margin \(\rho_0\), put
\begin{align}
  \eta_{\mathrm{ad}}
  =
  \min\left\{
    1,\,
    \frac{
      \rho_0\kappa(1+\bar\gamma)
    }{
      64\mathfrak N_S(U)
    },\,
    \frac{
      1
    }{
      64(1+S\sqrt\kappa)
    }
  \right\}.
  \label{eq:eta-ad-explicit}
\end{align}

\begin{corollary}[Global tangent-flow verification of the intrinsic certificate]
\label{thm:tangent-adaptive}
Suppose \Cref{ass:convex-smooth,ass:hessian-lip} hold, fix a
verification horizon \(S>0\) and a profile margin \(\rho_0>0\), let \(\nu\) be
\(M\)-warm, and let \(\epsilon\in(0,1/2)\). Choose an even integer
\(\ell\) satisfying
\begin{align}
  \ell
  \ge
  C\left[
    1+\log(ed)
    +\log\!\left(\frac{eM}{\epsilon}\right)
    +\log\!\left(
      e+\mathfrak C_S(U)+\mathfrak N_S(U)
      +\rho_0^{-1}
      +\eta_{\mathrm{ad}}^{-1}
      +\bar\gamma
    \right)
  \right],
  \label{eq:adaptive-ell}
\end{align}
set $d_\ell=d+2(\ell-1),$ and assume
\begin{align}
  d
  \ge
  C
  \left(
    \frac{\mathfrak C_S(U)}{\rho_0}
  \right)^2
  \ell.
  \label{eq:adaptive-dimension}
\end{align}
Choose
\begin{align}
  h^2
  =
  \frac{
    c\,\eta_{\mathrm{ad}}
  }{
    L(1+\bar\gamma)d_\ell^{1/2}\ell^{3/2}
  },
  \label{eq:adaptive-h}
\end{align}
where \(c>0\) is a sufficiently small universal constant.  Let
\(T_{\mathrm{sm}}\) be the cutoff in
\eqref{eq:small-time-cutoff-general}.  Suppose that some
\(k_\star<k_{\max}\) is
\((\rho_0,S,T_{\mathrm{sm}})\)-profile separated, and set $T_\star=\tau_{k_\star},$ and $a_\star=\sqrt m\,T_\star\in(0,S].$ All quantities in \eqref{eq:adaptive-ell}--\eqref{eq:adaptive-h}
use the fixed horizon \(S\).  For a specific candidate, one may set
\(S=a_\star\) only if profile separation and all stability, step-size, and
terminal-depth conditions are verified jointly for \((h,k_\star)\).  Thus
\(S=a_\star\) is a simultaneous feasibility condition, not a sequential definition.

Then the hypotheses of \Cref{thm:adaptive} hold after increasing the
logarithmic constant in \eqref{eq:adaptive-ell}.  In particular, on the certification event, both
kernels terminate through the U-turn rule before the cap, and the unconditional transition bounds satisfy
\begin{align}
  n_{\mathrm M}
  &\le
  C\left[
    1+a_\star^2\kappa^2(1+\gamma)^{4/3}
  \right]
  \log\!\left(\frac{CM}{\epsilon}\right),
  \label{eq:adaptive-transitions-M}\\
  n_{\mathrm B}
  &\le
  C\left[
    1+a_\star^4\kappa^3(1+\gamma)^2
  \right]
  \log\!\left(\frac{CM}{\epsilon}\right),
  \label{eq:adaptive-transitions-B}
\end{align}
with total-variation error at most \(\epsilon\).  Moreover,
\begin{align}
  K_\star
  \le
  C a_\star\sqrt\kappa
  \left(
    \frac{1+\bar\gamma}{\eta_{\mathrm{ad}}}
  \right)^{1/2}
  d_\ell^{1/4}\ell^{3/4},
  \label{eq:adaptive-K-explicit}
\end{align}
Let
\[
  K_{\mathrm{cap}}=2^{k_{\max}},
  \qquad
  \zeta_\star=\zeta_{p,h}(K_\star).
\]
For either selector and \(n=n_a\), \Cref{prop:exceptional-work} gives
\begin{align}
  \E_\nu\mathsf W_{n_a}^{(a)}
  \le
  C n_a
  \left[
    K_\star+(K_{\mathrm{cap}}-K_\star)
    \min\{1,M\zeta_\star\}
  \right].
  \label{eq:adaptive-expected-work}
\end{align}
If, additionally,
\(K_{\mathrm{cap}}\le C_{\mathrm{cap}}K_\star\), then the following
deterministic bounds hold:
\begin{align}
  N_{\nabla U}^{(\mathrm M)}
  &\le
  C C_{\mathrm{cap}} a_\star\sqrt\kappa
  \left(
    \frac{1+\bar\gamma}{\eta_{\mathrm{ad}}}
  \right)^{1/2}
  \left[
    1+a_\star^2\kappa^2(1+\gamma)^{4/3}
  \right]
  d_\ell^{1/4}\ell^{3/4}
  \log\!\left(\frac{CM}{\epsilon}\right),
  \label{eq:adaptive-complexity-M}\\
  N_{\nabla U}^{(\mathrm B)}
  &\le
  C C_{\mathrm{cap}} a_\star\sqrt\kappa
  \left(
    \frac{1+\bar\gamma}{\eta_{\mathrm{ad}}}
  \right)^{1/2}
  \left[
    1+a_\star^4\kappa^3(1+\gamma)^2
  \right]
  d_\ell^{1/4}\ell^{3/4}
  \log\!\left(\frac{CM}{\epsilon}\right).
  \label{eq:adaptive-complexity-B}
\end{align}
Without cap comparability, the corresponding high-probability bound is
\eqref{eq:high-probability-work-exceptional}.
\end{corollary}

\begin{corollary}[Polynomially stable profile-separated targets]
\label{cor:polynomial-stability}
In addition to the hypotheses of
\Cref{thm:tangent-adaptive}, suppose that, for some \(A_0\ge1\) and
\(q,r\ge0\),
\begin{align}
  \mathfrak A_S(U)
  \le
  A_0(1+\kappa)^q(1+S)^r.
  \label{eq:polynomial-tangent-stability}
\end{align}
Then
\begin{align}
  \mathfrak C_S(U)
  &\le
  C A_0
  (1+\kappa)^{q+1/2}
  (1+S)^r,
  \label{eq:polynomial-concentration-envelope}\\
  \mathfrak N_S(U)
  &\le
  C A_0
  (1+\kappa)^{q+3/2}
  (1+S)^{r+1}
  \bigl(1+\kappa^2+\bar\gamma\bigr).
  \label{eq:polynomial-numerical-envelope}
\end{align}
Every dimension, step-size, transition, and certified-orbit-size condition in
\Cref{thm:tangent-adaptive} is therefore polynomial in
\(\kappa\), \(S\), \(\rho_0^{-1}\), \(A_0\), and
\(1+\bar\gamma\).  The same is true of the deterministic work bound under
cap comparability, and of the expected or high-probability work bounds
whenever the cap ratio and exceptional-work term in
\Cref{prop:exceptional-work} are polynomially controlled.  In particular,
\begin{align*}
  d
  \ge
  C A_0^2
  (1+\kappa)^{2q+1}
  (1+S)^{2r}
  \rho_0^{-2}\ell
\end{align*}
is sufficient for exact sign concentration, and
\begin{align*}
  \eta_{\mathrm{ad}}^{-1}
  \le
  C\left[
    1+S\sqrt\kappa
    +
    \frac{
      A_0(1+\kappa)^{q+1/2}
      (1+S)^{r+1}
      (1+\kappa^2+\bar\gamma)
    }{
      \rho_0(1+\bar\gamma)
    }
  \right].
\end{align*}
\end{corollary}

\begin{proposition}[Adiabatic relative log-Hessian variation]
\label{prop:adiabatic-stability}
Along a scaled exact trajectory, set $H_s=\nabla^2\widetilde U(Y_s),$ and define
\begin{align*}
  \mathcal V_S(U)
  =
  \sup_{z\in\R^{2d}}
  \sup_{|t|\le S_\kappa^+}
  \int_0^{|t|}
  \left\|
    H_{\operatorname{sgn}(t)s}^{-1/2}
    \frac{\dd}{\dd s}
    H_{\operatorname{sgn}(t)s}
    H_{\operatorname{sgn}(t)s}^{-1/2}
  \right\|_{\mathrm{op}}
  \dd s.
\end{align*}
If, for some \(A_0\ge0\),
\begin{align}
  \mathcal V_S(U)
  \le
  A_0\log\kappa,
  \label{eq:adiabatic-condition}
\end{align}
then
\begin{align}
  \mathfrak A_S(U)
  \le
  \kappa^{(A_0+1)/2}.
  \label{eq:adiabatic-A}
\end{align}
Consequently, \Cref{cor:polynomial-stability} applies with
\[
  q=\frac{A_0+1}{2},
  \qquad
  r=0.
\]
\end{proposition}

\begin{corollary}[Curvature-only control for fast selected orbits]
\label{cor:fast-selected-stability}
Under the curvature assumptions alone,
\begin{align}
  \mathfrak A_S(U)
  \le
  e^{1/2}\sqrt\kappa
  \exp\!\left(
    \frac12S\sqrt\kappa
  \right).
  \label{eq:selected-horizon-curvature-fallback}
\end{align}
Consequently, if \(a_\star\le b/\sqrt\kappa\) and the
candidate-specific assumptions are verified simultaneously with \(S=a_\star\), then
\begin{align}
  \mathfrak A_S(U)
  \le
  e^{(b+1)/2}\sqrt\kappa.
  \label{eq:fast-selected-polynomial-A}
\end{align}
Thus the curvature-only route is polynomial for a fast selected orbit.
For a critical slow-scale orbit,
\eqref{eq:selected-horizon-curvature-fallback} is only a fallback and
is not used in the Gaussian or product verification results below.
\end{corollary}

The proofs of the preceding results, together with a parametric
amplification example explaining why curvature bounds alone cannot
provide a uniform polynomial tangent estimate on slow-scale horizons,
are given in \Cref{app:tangent-stability-proof}.

\subsection{Metric adaptation and relative-Hessian stability}
\label{sec:metric-adaptation}

A fixed post-warmup metric can remove linear anisotropy before the profile and diagnostic-stability conditions are evaluated. Throughout this subsection, the metric NUTS algorithm uses the Hamiltonian
\[
  H_G(x,p)=U(x)+\frac12p^\top Gp,
  \qquad p\sim N(0,G^{-1}),
\]
its leapfrog update
\[
  p_{1/2}=p-\frac h2\nabla U(x),\qquad
  x'=x+hGp_{1/2},\qquad
  p'=p_{1/2}-\frac h2\nabla U(x'),
\]
and the canonical-momentum endpoint test
\[
  p_a^\top(x_b-x_a)<0
  \quad\text{or}\quad
  p_b^\top(x_b-x_a)<0,
\]
with the same recursive subtree rule as in \eqref{eq:uturn}--\eqref{eq:subuturn}.  Under the
canonical transformation
\(x=b+G^{1/2}z\), \(p=G^{-1/2}v\), one has
\[
  v_a^\top(z_b-z_a)=p_a^\top(x_b-x_a),
  \qquad
  v_b^\top(z_b-z_a)=p_b^\top(x_b-x_a),
\]
and the metric leapfrog map is conjugate to the identity-kinetic leapfrog map
for the transformed potential.  Hence the orbit builder and every endpoint and
recursive U-turn decision are exactly transformed into the algorithm analyzed
above.  The result below does not cover alternative mass-matrix stopping rules
based directly on the velocity \(Gp\) unless they are shown separately to be
equivalent.

\begin{theorem}[Relative-Hessian stability after metric adaptation]
\label{thm:relative-hessian-stability}
Let \(G\) be a fixed positive-definite kinetic metric and consider
\begin{align*}
  H_G(x,p)
  =
  U(x)+\frac12p^\top Gp,
  \qquad
  p\sim N(0,G^{-1}).
\end{align*}
Suppose that, for some \(\delta\in[0,1)\),
\begin{align}
  \norm{
    G^{1/2}\nabla^2U(x)G^{1/2}-I_d
  }_{\mathrm{op}}
  \le
  \delta,
  \qquad
  x\in\R^d.
  \label{eq:relative-Hessian-condition}
\end{align}
After the canonical affine transformation to identity kinetic energy, the curvature
constants are
\begin{align}
  m_G=1-\delta,
  \qquad
  L_G=1+\delta,
  \qquad
  \kappa_G=\frac{1+\delta}{1-\delta},
  \label{eq:relative-Hessian-kappa}
\end{align}
and the selected-horizon tangent amplification satisfies
\begin{align}
  \mathfrak A_S(U_G)
  \le
  \frac{1}{\sqrt{1-\delta}}
  \exp\!\left\{
    \frac{\delta(S+1)}{2\sqrt{1-\delta}}
  \right\}.
  \label{eq:relative-Hessian-A}
\end{align}
Moreover, for \(0\le t\le2\pi\), the transformed stationary profile obeys
\begin{align}
  \abs{u_G(t)-d\sin t}
  \le
  C_0\delta d.
  \label{eq:relative-Hessian-profile}
\end{align}

Consequently, a dyadic grid with reference sine margin \(\rho_{\sin}\) is profile
separated whenever \(\rho_{\sin}>C_0\delta\), with the
normalized margin
\[
  \rho_0=\sqrt{1-\delta}\,\bigl(\rho_{\sin}-C_0\delta\bigr).
\]  For every fixed \(\delta_0<1\), if
\(\delta\le\delta_0\) and this reduced margin is bounded below by a positive constant,
then \Cref{thm:tangent-adaptive}, and hence \Cref{thm:adaptive}, holds with constants
depending on \(\delta_0\), the reduced profile margin, the selected horizon, and the
transformed Hessian regularity, but not on the condition number of \(G\) or on the raw
Euclidean condition number of \(\nabla^2U\).
\end{theorem}

We remark that the algorithm is not given \(k_\star\), \(a_\star\), \(\rho_0\), or
any intrinsic diagnostic modulus.  It runs the ordinary endpoint and
recursive sub-U-turn tests.  The theorem proves that those tests select
\(k_\star\) on the certification event.  For transition mixing, the maximum depth needs only
to exceed the selected depth.  A deterministic order-\(K_\star n_a\)
work guarantee additionally requires a comparable cap; otherwise the
cap-aware expected and high-probability bounds of
\Cref{prop:exceptional-work} apply.  The role of the fixed post-warmup metric
is to alter the transformed profile and intrinsic stability quantities;
\Cref{thm:relative-hessian-stability} gives one condition under which
these quantities are independent of raw Euclidean anisotropy. The proof is given in \Cref{app:metric-adaptation-proof}.

\section{Gaussian specialization and mixing regimes}
\label{sec:gaussian-specialization}

For Gaussian targets, the stationary profile, diagnostic fluctuations,
and leapfrog dynamics can be analyzed directly.  This yields
polynomial bounds without invoking a global tangent-flow envelope.  It
also makes the relationship with the anisotropic Gaussian analysis of
\citet{Oberdoerster2025} explicit.

\subsection{The Gaussian stationary profile}
\label{sec:gaussian-profile}

\begin{theorem}[Gaussian profile and separated step sizes]
\label{thm:gaussian-arbitrary-kappa}
Let \(\pi=N(0,C)\), set \(\Lambda=C^{-1}\), and assume
\begin{align}
  mI_d
  \preceq
  \Lambda
  \preceq
  LI_d.
  \label{eq:gaussian-precision-bounds}
\end{align}
Then
\begin{align}
  u(t)=
  \Tr\!\left(
    \Lambda^{-1/2}
    \sin(\Lambda^{1/2}t)
  \right)=
  \Tr\!\left(
    \sin(C^{-1/2}t)C^{1/2}
  \right)
  =
  \sum_{i=1}^d
  \frac{\sin(\omega_i t)}{\omega_i}.
  \label{eq:gaussian-profile}
\end{align}
where $\omega_i=\sqrt{\lambda_i(\Lambda)}$. The middle expression is exactly the deterministic uniform term in
\citet{Oberdoerster2025}.  Thus the Gaussian specialization of the
stationary profile used here is identical to the concentration mean
used in that analysis.

For every positive-definite \(C\), \(u(t)<0\) for some
\(t<2\pi/\sqrt m\).  For every fixed \(\kappa=L/m\), there is
\(c_{\mathrm G}(\kappa)>0\) such that every covariance satisfying
\eqref{eq:gaussian-precision-bounds} has an interval in
\((0,2\pi/\sqrt m)\) on which
\begin{align}
  u(t)
  \le
  -c_{\mathrm G}(\kappa)
  \frac{d}{\sqrt m}.
  \label{eq:gaussian-uniform-negative}
\end{align}
For each fixed covariance, there is a positive-Lebesgue-measure set of
sufficiently small step sizes for which the first negative dyadic
profile value is profile separated with a target-dependent positive
margin.  No sharper uniform lower bound on that margin is claimed.

The exact scalar diagnostics satisfy
\begin{align*}
  \mathfrak C_p(S;U)
  \le
  C,
  \qquad
  p\ge2,\quad S>0,
\end{align*}
where \(C\) is universal and independent of \(\kappa\).  The practical
leapfrog fidelity and energy conditions are verified directly in
\Cref{thm:gaussian-leapfrog-stability}, without a bound on the full
Hamiltonian Jacobian.

For comparison, in the scaled identity-mass coordinates the exact
Gaussian tangent amplification satisfies
\begin{align}
  \mathfrak A_S(U)
  \le
  2\sqrt\kappa,
  \qquad
  S>0,
  \label{eq:gaussian-tangent-amplification}
\end{align}
while under the whitening metric it equals one.
\end{theorem}

\begin{proposition}[Finite-dimensional Gaussian profile certificate]
\label{prop:gaussian-margin-certificate}
Under the assumptions of
\Cref{thm:gaussian-arbitrary-kappa}, let \(\mathscr T\) be a finite
family of \(N\) exact-flow intervals along a stationary Gaussian
Hamiltonian trajectory.  For \(I=[a,b]\), write \(\ell(I)=b-a\).  For every \(r\ge1\),
\begin{align}
  \Pp\!\left[
    \max_{I\in\mathscr T}
    \max_{\sigma\in\{-,+\}}
    \abs{
      D_I^\sigma-u(\ell(I))
    }
    >
    \frac{C}{\sqrt m}
    \bigl(\sqrt{dr}+r\bigr)
  \right]
  \le
  4Ne^{-r}.
  \label{eq:gaussian-margin-concentration}
\end{align}
Consequently, with probability at least \(1-\delta\),
\begin{align}
  \max_{I\in\mathscr T}
  \max_{\sigma\in\{-,+\}}
  \abs{
    D_I^\sigma-u(\ell(I))
  }
  \le
  \frac{C}{\sqrt m}
  \left[
    \sqrt{
      d\log\!\frac{4N}{\delta}
    }
    +
    \log\!\frac{4N}{\delta}
  \right].
  \label{eq:gaussian-margin-concentration-delta}
\end{align}
If the normalized population margin on the family is at least
\(\rho_0\in(0,1]\) and
\begin{align}
  d
  \ge
  C\rho_0^{-2}
  \log\!\frac{4N}{\delta},
  \label{eq:gaussian-margin-dimension}
\end{align}
then every exact-flow diagnostic in \(\mathscr T\) has the sign
prescribed by the profile.

In particular, a margin of order
\(\rho_0\asymp\kappa^{-1/2}\) requires the sign-certification threshold
\[
  d
  \gtrsim
  \kappa\log(4N/\delta).
\]
This threshold is distinct from the mixing penalty determined by the
selected physical time.
\end{proposition}

\subsection{Practical leapfrog certification}
\label{sec:gaussian-practical-certification}

\begin{theorem}[Intrinsic Gaussian leapfrog stability]
\label{thm:gaussian-leapfrog-stability}
Under the assumptions of
\Cref{thm:gaussian-arbitrary-kappa}, let \(K\ge8\) be a power of two,
assume
\begin{align*}
  Kh
  \le
  \frac{S+1}{\sqrt m},
  \qquad
  h\sqrt L
  \le
  1,
\end{align*}
and let \(p\ge2\) satisfy
\begin{align}
  p
  \ge
  \log(16N_K).
  \label{eq:gaussian-lf-p}
\end{align}
There is a universal constant \(C\) such that
\begin{align}
  \mathfrak C_p(S;U)
  &\le
  C,
  \label{eq:gaussian-Cp-direct}\\
  \mathfrak E_{p,h}(K,\log2;U)
  &\le
  C(1+S)Lh^2
  \left(
    1+\sqrt{\frac pd}+\frac pd
  \right).
  \label{eq:gaussian-Ep-direct}
\end{align}

Moreover, if
\begin{align}
  C\left\{
    Lh^2(\sqrt{dp}+p)
    +
    L^2h^4d
  \right\}
  \le
  \log2,
  \label{eq:gaussian-energy-condition}
\end{align}
then
\begin{align}
  \delta_H(h,K,\log2)
  \le
  2e^{-p}.
  \label{eq:gaussian-energy-tail}
\end{align}
There are universal constants
\(c_{\mathrm{sm}},c,C_0>0\) such that, with
\begin{align*}
  T_{\mathrm{sm}}
  =
  \frac{c_{\mathrm{sm}}}{\sqrt L},
\end{align*}
if \(d\ge C_0p\) and \(Lh^2\le c\), then
\begin{align}
  \delta_{\mathrm{sm}}(h,K,T_{\mathrm{sm}})
  \le
  2e^{-p}.
  \label{eq:gaussian-short-tail}
\end{align}
All bounds are uniform over the covariance spectrum and contain no
exponential factor in \(\kappa\).
\end{theorem}

\begin{corollary}[Gaussian profile separation for practical leapfrog NUTS]
\label{cor:gaussian-practical}
Assume the setting of
\Cref{thm:gaussian-leapfrog-stability}, and suppose that
\(k_\star<k_{\max}\) is
\((\rho_0,S,T_{\mathrm{sm}})\)-profile separated.  Choose \(p\) so that
\begin{align}
  p
  \ge
  C\log\!\left(
    \frac{
      eN_{K_\star}M^2
    }{
      \epsilon^2\underline a_{\mathrm B}\theta_\star
    }
  \right),
  \label{eq:gaussian-practical-p}
\end{align}
assume
\begin{align}
  d
  \ge
  C\rho_0^{-2}p,
  \qquad
  C(1+S)Lh^2
  \le
  \rho_0,
  \label{eq:gaussian-practical-sign}
\end{align}
and impose
\eqref{eq:intrinsic-resolution} and
\eqref{eq:gaussian-energy-condition}.  Then the hypotheses of
\Cref{thm:adaptive} hold simultaneously for multinomial and
biased-progressive NUTS.  In particular, on the certification event, both practical leapfrog trees
stop at depth \(k_\star\) through the U-turn criterion before the cap.
Unconditionally, they satisfy the transition and total-variation bounds
\eqref{eq:intrinsic-transitions-M}--\eqref{eq:intrinsic-tv}; their total
work is governed by \Cref{prop:exceptional-work}.

A sufficient energy-scale choice is
\begin{align}
  h^2
  =
  \frac{
    c_h
  }{
    L(\sqrt{dp}+p)
  },
  \label{eq:gaussian-practical-h}
\end{align}
with \(c_h>0\) sufficiently small.  When \(p\) is logarithmic and
\(d\gtrsim p\),
\begin{align}
  h
  =
  \widetilde\Theta(
    L^{-1/2}d^{-1/4}
  ),
  \qquad
  K_\star
  =
  \widetilde O(
    a_\star\sqrt\kappa\,d^{1/4}
  ).
  \label{eq:gaussian-practical-cost}
\end{align}
This is the cost of a certified transition.  If
\(K_{\mathrm{cap}}\le C_{\mathrm{cap}}K_\star\), it yields deterministic
total work \(\widetilde O(C_{\mathrm{cap}}K_\star n_a)\); otherwise
\eqref{eq:expected-work-exceptional} and
\eqref{eq:high-probability-work-exceptional} give the corresponding
cap-aware bounds.
\end{corollary}

\subsection{The two-scale phase diagram}
\label{sec:gaussian-two-scale}

Consider
\begin{align}
  \pi_{2\mathrm S}
  =
  N(0,m_1^{-1}I_{d_1})
  \otimes
  N(0,m_2^{-1}I_{d_2}),
  \qquad
  0<m_1\le m_2,
  \label{eq:two-scale-target}
\end{align}
and put
\begin{align*}
  \kappa
  =
  \frac{m_2}{m_1},
  \qquad
  q
  =
  \frac{d_2}{d_1},
  \qquad
  s
  =
  \sqrt{m_2}\,t.
\end{align*}
Define
\begin{align}
  g_{\kappa,q}(s)
  =
  \sin(s/\sqrt\kappa)
  +
  \frac{q}{\sqrt\kappa}\sin s.
  \label{eq:g-kappa-q}
\end{align}
Then
\begin{align}
  u(t)
  =
  \frac{d_1}{\sqrt{m_1}}
  g_{\kappa,q}(\sqrt{m_2}t),
  \qquad
  \frac{\sqrt{m_1}}{d_1+d_2}u(t)
  =
  \frac{g_{\kappa,q}(s)}{1+q}.
  \label{eq:two-scale-profile-identity}
\end{align}
This is the two-scale uniform term of
\citet{Oberdoerster2025} in normalized coordinates.

Let
\[
  \bar h=\sqrt{m_2}h
\]
and
\begin{align*}
  s_k(\bar h)
  =
  \bar h(2^k-1).
\end{align*}
For a candidate terminal depth, define
\begin{align}
  \mathfrak r_{\kappa,q}(\bar h,k)
  =
  \frac1{1+q}
  \min\left\{
    -g_{\kappa,q}(s_k),
    \min_{\substack{
      1\le r<k\\
      s_r\ge\sqrt{m_2}T_{\mathrm{sm}}
    }}
    g_{\kappa,q}(s_r)
  \right\}.
  \label{eq:two-scale-margin}
\end{align}
By \eqref{eq:two-scale-profile-identity},
\Cref{def:profile-separated} holds with \(m=m_1\) and margin \(\rho_0\)
exactly when the horizon condition holds and $\mathfrak r_{\kappa,q}(\bar h,k)\ge \rho_0.$

Following \citet{Oberdoerster2025}, define the accelerated phase
\begin{align*}
  \mathcal A
  =
  \left\{
    (\kappa,q)\in[1,\infty)\times(0,\infty):
    g_{\kappa,q}(s)\ge0
    \text{ for every }s\in(0,2\pi)
  \right\},
\end{align*}
and the fast negative set $\mathcal J_{\kappa,q}=  \left\{
    s\in(0,2\pi):
    g_{\kappa,q}(s)<0
  \right\}.$

\begin{theorem}[Profile separation recovers the two-scale dichotomy]
\label{thm:two-scale-reconciliation}
For the target \eqref{eq:two-scale-target}, the following statements
hold.
\begin{enumerate}[label=(\roman*)]
\item
\((\kappa,q)\in\mathcal A\) if and only if
\(\mathcal J_{\kappa,q}=\varnothing\).  Thus membership in the
accelerated phase is exactly the absence of a negative stationary
profile value during the first fast period.

\item
If \((\kappa,q)\notin\mathcal A\), then
\(\mathcal J_{\kappa,q}\) is a nonempty open set and
\begin{align}
  \mathcal H_{\mathrm{trap}}
  =
  \bigcup_{k\ge1}
  \frac{
    \mathcal J_{\kappa,q}
  }{
    \sqrt{m_2}(2^k-1)
  }
  \label{eq:trap-step-set}
\end{align}
is exactly the deterministic-profile version of the trapped
step-size set of \citet{Oberdoerster2025}.  For sufficiently small step sizes in the
interior of this set, away from dyadic preimages of profile zeros, satisfying
\eqref{eq:intrinsic-resolution}, and whose first negative depth is at least two, the first negative dyadic value is profile separated at a fast depth
with
\begin{align*}
  a_\star
  =
  \sqrt{m_1}T_\star
  =
  \Theta(\kappa^{-1/2}).
\end{align*}
Any later critical candidate fails the preterminal positivity
requirement because the grid has already encountered a negative
fast-profile value.  At the boundary of
\(\mathcal H_{\mathrm{trap}}\), the deterministic margin
\eqref{eq:two-scale-margin} vanishes.

\item
If \((\kappa,q)\in\mathcal A\), the first negative continuous profile
value satisfies
\begin{align}
  \frac{\pi}{2\sqrt{m_1}}
  \le
  \inf\{t>0:u(t)<0\}
  <
  \frac{2\pi}{\sqrt{m_1}}.
  \label{eq:phase-A-critical-window}
\end{align}
There is a positive-measure set of sufficiently small step sizes for
which the first negative dyadic value is profile separated at a
critical depth, with \(a_\star=\Theta(1)\).

\item
If a compact set of step sizes stays a positive distance from all
relevant dyadic preimages of zeros of \(g_{\kappa,q}\), then the
corresponding profile margins are bounded below uniformly.  Profile
separation therefore provides a finite-dimensional quantitative
version of the asymptotic nonresonance condition in
\citet{Oberdoerster2025}.
\end{enumerate}
\end{theorem}

\begin{remark}[Trapped orbits and profile separation]
A trapped step size need not make profile separation fail altogether.
It can produce a well-separated \emph{fast} terminal depth.  What
fails is profile separation at a later slow, critical depth, because
an earlier preterminal dyadic time is already negative.  Conversely,
near a boundary of the trapped set the first negative value has little
or no margin, so the finite-dimensional certificate becomes
ineffective.  This is the precise sense in which the two descriptions
coincide.
\end{remark}

\subsection{Selected-time interpolation}
\label{sec:gaussian-selected-time}

For a deterministic terminal cardinality \(K\), exact Hamiltonian flow
and equal orbit energies make the selected signed index independent of
the current state.  Let
\(\mathsf P_{\mathrm M}^{K,h}\) and
\(\mathsf P_{\mathrm B}^{K,h}\) denote the resulting Gaussian position
kernels with signed-index laws
\eqref{eq:terminal-mul-law} and
\eqref{eq:terminal-bps-law}. The physical span of the
\(K\)-point orbit is \(T_K=(K-1)h\).  For the continuous selector approximation we
instead use
\[
  \widetilde T=Kh,
  \qquad
  \widetilde a=\sqrt m\,\widetilde T.
\]
For \(K\ge4\), \(T_K<\widetilde T\le\frac43T_K\), so the two time scales are
universally comparable but not identical.  For a self-adjoint Markov operator \(P\), let
\[
  L_0^2(\pi)=\{f\in L^2(\pi):\pi(f)=0\}.
\]
Define the absolute \(L^2(\pi)\) spectral gap by
\begin{align*}
  \gap_{\mathrm{abs}}(P)
  =
  1-\left\|P\big|_{L_0^2(\pi)}\right\|_{2\to2}
  =
  1-
  \sup\left\{
    |\lambda|:
    \lambda\in\sigma\!\left(P\big|_{L_0^2(\pi)}\right)
  \right\}.
\end{align*}

\begin{theorem}[Gaussian selected-time interpolation]
\label{thm:gaussian-time-interpolation}
There are universal constants \(c_\sigma,C_\sigma>0\),
\(\sigma\in\{\mathrm M,\mathrm B\}\), such that, for every Gaussian
precision matrix satisfying
\eqref{eq:gaussian-precision-bounds},
\begin{align}
  \gap_{\mathrm{abs}}
  \left(
    \mathsf P_\sigma^{K,h}
  \right)
  \ge
  c_\sigma\min\{m\widetilde T^2,1\}
  -
  C_\sigma\frac{\sqrt L\,\widetilde T}{K}.
  \label{eq:gaussian-discrete-gap}
\end{align}
If
\begin{align}
  K
  \ge
  \frac{
    2C_\sigma\sqrt L\,\widetilde T
  }{
    c_\sigma\min\{m\widetilde T^2,1\}
  },
  \label{eq:gaussian-K-gap-condition}
\end{align}
then an \(M\)-warm start reaches total-variation accuracy
\(\epsilon\) after
\begin{align}
  n_\sigma
  \le
  C_\sigma\min\{\widetilde a^2,1\}^{-1}
  \log\!\left(
    \frac{CM}{\epsilon}
  \right)
  \label{eq:gaussian-interpolation-transitions}
\end{align}
transitions.  The orbit work satisfies
\begin{align}
  Kn_\sigma
  \le
  C_\sigma\frac{\widetilde T}{h}
  \min\{\widetilde a^2,1\}^{-1}
  \log\!\left(
    \frac{CM}{\epsilon}
  \right).
  \label{eq:gaussian-interpolation-work}
\end{align}
On the Gaussian leapfrog scale
\[
  h
  =
  \widetilde\Theta(
    L^{-1/2}d^{-1/4}
  ),
\]
provided \eqref{eq:gaussian-K-gap-condition} holds,
\begin{align}
  Kn_\sigma
  =
  \widetilde O\!\left(
    \sqrt\kappa\,d^{1/4}
    \frac{
      \widetilde a
    }{
      \min\{\widetilde a^2,1\}
    }
    \log\!\frac{M}{\epsilon}
  \right).
  \label{eq:gaussian-smooth-interpolation}
\end{align}
At this basic step-size scale,
\eqref{eq:gaussian-K-gap-condition} is equivalent, up to universal
and logarithmic factors, to
\[
  d^{-1/4}
  \lesssim
  \min\{\widetilde a^2,1\}.
\]
\end{theorem}

\begin{corollary}[Critical and trapped Gaussian regimes]
\label{cor:gaussian-recovered-rates}
Under the resolution condition of
\Cref{thm:gaussian-time-interpolation},
\[
  \widetilde a=\Theta(1)
  \quad\Longrightarrow\quad
  Kn_\sigma
  =
  \widetilde O\!\left(
    \sqrt\kappa\,d^{1/4}
    \log\!\frac{M}{\epsilon}
  \right),
\]
whereas
\[
  \widetilde a=\Theta(\kappa^{-1/2})
  \quad\Longrightarrow\quad
  Kn_\sigma
  =
  \widetilde O\!\left(
    \kappa\,d^{1/4}
    \log\!\frac{M}{\epsilon}
  \right).
\]
Thus the selector-time formula continuously interpolates between the
critical-orbit and trapped fast-orbit regimes.
\end{corollary}

\begin{remark}[Scope of the sharp interpolation]
\Cref{thm:gaussian-time-interpolation} is an exact-flow, equal-energy
statement for the Gaussian profile kernels induced by the two selector
laws.  The leapfrog analysis of \citet{Oberdoerster2025} proves the
corresponding practical multinomial dichotomy under Gaussian
concentration and energy-control conditions.  The general theorem in
this paper also covers practical leapfrog BPS, but its conductance
constants do not presently yield the sharp interpolation
\eqref{eq:gaussian-smooth-interpolation} for anisotropic BPS.
\end{remark}

\subsection*{Acknowledgements}

KB is supported in part by National Science Foundation (NSF) grant DMS-2413426 and wishes to thank Nawaf Bou-Rabee for various discussions about NUTS at the International Conference on Monte Carlo Methods and Applications (MCM), 2025.

\subsection*{Statement of AI Use}

ChatGPT-5.6 Sol Pro was used during manuscript preparation for brainstorming, drafting, editing, and formatting assistance, including preliminary drafts of proof arguments. The LLM-generated material was not used without author review. All proof arguments were independently checked, corrected, and finalized by the author, who takes full responsibility for the correctness, originality, and presentation of the paper.

\bibliographystyle{abbrvnat}
\bibliography{ref}

\appendix

\section{Positive-semidefinite conductance implies warm-start mixing}
\label{sec:PSD-mixing}

Recalling the definition in~\eqref{equation:overlap}, note that if \((a_j)\) are common nonnegative weights, then
\begin{align}
\ov\!\left(\sum_j a_j\mu_j,\sum_j a_j\nu_j\right)
\ge\sum_j a_j\ov(\mu_j,\nu_j).
\label{eq:mixture-overlap}
\end{align}
Indeed, for nonnegative densities \(f_j,g_j\),
\(
\min\{\sum_j a_jf_j,\sum_j a_jg_j\}
\ge\sum_j a_j\min\{f_j,g_j\}
\).
If \(\widetilde\mu\le\mu\) and \(\widetilde\nu\le\nu\), then
\begin{align}
\ov(\widetilde\mu,\widetilde\nu)
\ge
\ov(\mu,\nu)
-[\mu(\R^d)-\widetilde\mu(\R^d)]
-[\nu(\R^d)-\widetilde\nu(\R^d)].
\label{eq:restriction-overlap}
\end{align}
This follows by writing the restricted densities as \(f-r\) and \(g-s\), with \(r,s\ge0\), and using
\(
\min\{f-r,g-s\}\ge\min\{f,g\}-r-s
\).

For a reversible kernel \(R\), define the stationary flow
\begin{align*}
\mathcal Q_R(A,B)=\int_A R(x,B)\pi(\dd x)
\end{align*}
and, for \(s\in(0,1/2)\), its restricted conductance
\begin{align}
\Phi_s(R)
=
\inf_{s<\pi(S)\le1/2}
\frac{\mathcal Q_R(S,S^c)}{\pi(S)-s}.
\label{eq:s-conductance}
\end{align}

The corresponding lazy-kernel conductance bound is used by
\citet[Lemma~1]{ChenGatmiryJiang2026}, who invoke the classical
Lovasz--Simonovits theorem.  Here we use positive semidefiniteness
in place of laziness in order to obtain a mixing bound for the
original NUTS kernel, rather than for an artificially lazified
version.

For a reversible Markov kernel \(R\), define
\begin{align*}
  \mathcal E_R(f,f)
  =
  \frac12\iint[f(x)-f(y)]^2\pi(\dd x)R(x,\dd y)
  =\ip{f}{(I-R)f}_{L^2(\pi)}.
\end{align*}

\begin{proposition}[Approximate Poincar\'e inequality]
\label{prop:approx-poincare}
Let \(R\) be reversible, and let \(\Phi_s(R)\) be defined by \eqref{eq:s-conductance}.  For every bounded \(f\in L^2(\pi)\) with \(\pi(f)=0\),
\begin{align}
  \mathcal E_R(f,f)
  \ge
  \frac{\Phi_s(R)^2}{2}
  \left[\norm{f}_2^2-16s\norm{f}_\infty^2\right].
  \label{eq:approx-poincare}
\end{align}
\end{proposition}

\begin{proof}
We first prove a one-sided inequality.  Let \(g\ge0\) be bounded with \(\pi(\supp g)\le1/2\), and define
\[
  S_t=\{x:g(x)^2>t\},
  \qquad t\ge0.
\]
Let \(\mathcal Q_R(\dd x,\dd y)=\pi(\dd x)R(x,\dd y)\).  Reversibility makes \(\mathcal Q_R\) symmetric.  The identity
\[
  |a-b|
  =\int_0^\infty
  \left|\1_{\{a>t\}}-\1_{\{b>t\}}\right|\dd t,
  \qquad a,b\ge0,
\]
and Tonelli's theorem give
\begin{align*}
  I_g
  &:=
  \iint|g(x)^2-g(y)^2|\mathcal Q_R(\dd x,\dd y)
  \notag\\
  &=2\int_0^\infty\mathcal Q_R(S_t,S_t^c)\dd t.
\end{align*}
Every \(S_t\) has measure at most \(1/2\).  By the definition of \(\Phi_s(R)\),
\[
  \mathcal Q_R(S_t,S_t^c)
  \ge
  \Phi_s(R)[\pi(S_t)-s]_+.
\]
Therefore
\begin{align}
  I_g
  &\ge
  2\Phi_s(R)
  \int_0^\infty[\pi(S_t)-s]_+\dd t
  \notag\\
  &\ge
  2\Phi_s(R)
  \left[\norm{g}_2^2-s\norm{g}_\infty^2\right]_+.
  \label{eq:Ig-lower}
\end{align}
For the second inequality, use \([a-s]_+\ge a-s\) and note that \(S_t=\varnothing\) for \(t\ge\norm{g}_\infty^2\).

Cauchy--Schwarz yields
\begin{align*}
  I_g^2
  &\le
  \left[\iint(g(x)-g(y))^2\dd\mathcal Q_R\right]
  \left[\iint(g(x)+g(y))^2\dd\mathcal Q_R\right].
\end{align*}
The first factor is \(2\mathcal E_R(g,g)\).  Since
\[
  (g(x)+g(y))^2\le2g(x)^2+2g(y)^2
\]
and both marginals of \(\mathcal Q_R\) are \(\pi\), the second factor is at most \(4\norm{g}_2^2\).  Hence
\begin{align}
  I_g^2\le8\mathcal E_R(g,g)\norm{g}_2^2.
  \label{eq:Ig-upper}
\end{align}
Set \(a=\norm{g}_2^2\) and \(b=s\norm{g}_\infty^2\).  If \(a\le b\), then \(a-2b\le0\), and the next inequality follows from nonnegativity of the Dirichlet form.  If \(a>b\), combine \eqref{eq:Ig-lower} and \eqref{eq:Ig-upper}:
\[
  \mathcal E_R(g,g)
  \ge
  \frac{\Phi_s(R)^2}{2}\frac{(a-b)^2}{a}.
\]
Since
\[
  \frac{(a-b)^2}{a}
  =a-2b+\frac{b^2}{a}
  \ge a-2b,
\]
both cases give
\begin{align}
  \mathcal E_R(g,g)
  \ge
  \frac{\Phi_s(R)^2}{2}
  \left[\norm{g}_2^2-2s\norm{g}_\infty^2\right].
  \label{eq:one-sided-poincare}
\end{align}

Now let \(f\) be bounded with \(\pi(f)=0\), and let \(c\) be a median of \(f\).  Define
\[
  g_+=(f-c)_+,
  \qquad
  g_-=(c-f)_+.
\]
The supports of \(g_+\) and \(g_-\) each have \(\pi\)-measure at most \(1/2\).  For real \(a,b\),
\begin{align}
  (a-b)^2
  \ge
  (a_+-b_+)^2
  +((-a)_+-(-b)_+)^2.
  \label{eq:positive-negative-split}
\end{align}
To verify \eqref{eq:positive-negative-split}, consider the cases in which \(a,b\) have the same sign or opposite signs.  If both are nonnegative, the first term equals \((a-b)^2\) and the second is zero.  If both are nonpositive, the second equals \((a-b)^2\) and the first is zero.  If \(a\ge0\ge b\), the right side is \(a^2+b^2\), while the left side is \((a-b)^2=a^2+b^2-2ab\ge a^2+b^2\); the remaining opposite-sign case is symmetric.

Apply \eqref{eq:positive-negative-split} with \(a=f(x)-c\) and \(b=f(y)-c\), then integrate:
\begin{align}
  \mathcal E_R(f,f)
  \ge
  \mathcal E_R(g_+,g_+)+\mathcal E_R(g_-,g_-).
  \label{eq:E-split}
\end{align}
Since \(\pi(f)=0\),
\begin{align*}
  \norm{g_+}_2^2+\norm{g_-}_2^2
  &=\norm{f-c}_2^2
  =\norm{f}_2^2+c^2
  \ge\norm{f}_2^2.
\end{align*}
A median may be chosen in the essential range of \(f\), so \(|c|\le\norm{f}_\infty\).  Consequently,
\[
  \norm{g_+}_\infty\le2\norm{f}_\infty,
  \qquad
  \norm{g_-}_\infty\le2\norm{f}_\infty,
\]
and
\begin{align}
  \norm{g_+}_\infty^2+\norm{g_-}_\infty^2
  \le8\norm{f}_\infty^2.
  \label{eq:Linf-split}
\end{align}
Apply \eqref{eq:one-sided-poincare} to \(g_+\) and \(g_-\), and then use \eqref{eq:E-split}--\eqref{eq:Linf-split}.  This proves \eqref{eq:approx-poincare}.
\end{proof}

\begin{proposition}[Warm-start mixing for a positive kernel]
\label{prop:PSD-mixing}
Let \(R\) be reversible and positive semidefinite, let \(\nu\) be \(M\)-warm with
\(M\ge1\), and let \(\epsilon\in(0,1)\).  Put
\[
  s=\frac{\epsilon^2}{128M^2},
  \qquad
  \bar\phi=\min\{\phi,1\}.
\]
If \(\Phi_s(R)\ge\phi>0\), then
\begin{align}
  \norm{\nu R^n-\pi}_{\TV}\le\epsilon
  \label{eq:PSD-mixing-conclusion}
\end{align}
whenever
\begin{align}
  n\ge C\bar\phi^{-2}\log\!\left(\frac{CM}{\epsilon}\right)
  \label{eq:PSD-mixing-time}
\end{align}
for a universal constant \(C\).
\end{proposition}

\begin{proof}
Let
\[
  r_0=\frac{\dd\nu}{\dd\pi},
  \qquad
  f_0=r_0-1,
  \qquad
  f_n=R^nf_0.
\]
Since \(\nu\) is \(M\)-warm, \(0\le r_0\le M\) almost everywhere.  Hence
\begin{align*}
  \norm{f_0}_\infty\le M.
\end{align*}
A Markov operator is an \(L^\infty\) contraction, so
\begin{align}
  \norm{f_n}_\infty\le M,
  \qquad n\ge0.
  \label{eq:fn-Linf}
\end{align}
Moreover,
\begin{align}
  \norm{f_0}_2^2
  &=\int(r_0-1)^2\dd\pi
  =\int r_0^2\dd\pi-1
  \notag\\
  &\le M\int r_0\dd\pi-1
  =M-1
  \le M.
  \label{eq:f0-L2}
\end{align}

Because \(R\) is self-adjoint, positive semidefinite, and contractive, its spectrum lies
in \([0,1]\).  The scalar inequality \(\lambda^2\le\lambda\) on \([0,1]\), together
with the spectral theorem, gives
\begin{align*}
  R^2\preceq R.
\end{align*}
Let \(a_n=\norm{f_n}_2^2\).  Then
\begin{align*}
  a_{n+1}
  &=\ip{f_n}{R^2f_n}
  \le\ip{f_n}{Rf_n}
  \notag\\
  &=a_n-\mathcal E_R(f_n,f_n).
\end{align*}
Since \(\Phi_s(R)\ge\phi\ge\bar\phi\),
\Cref{prop:approx-poincare} may be applied with \(\bar\phi\).  Its right-hand side is
allowed to be nonpositive; when that happens the inequality follows from
\(\mathcal E_R(f_n,f_n)\ge0\).  Using \(\pi(f_n)=0\) and
\eqref{eq:fn-Linf}, we obtain
\begin{align*}
  a_{n+1}
  \le
  \left(1-\frac{\bar\phi^2}{2}\right)a_n
  +8\bar\phi^2sM^2.
\end{align*}
Because \(0<\bar\phi\le1\), the multiplier belongs to \([1/2,1)\).  Iterating the
recursion and using \eqref{eq:f0-L2} gives
\begin{align}
  a_n
  &\le
  \left(1-\frac{\bar\phi^2}{2}\right)^nM
  +8\bar\phi^2sM^2
  \sum_{j=0}^{n-1}
  \left(1-\frac{\bar\phi^2}{2}\right)^j
  \notag\\
  &\le
  M e^{-n\bar\phi^2/2}+16sM^2.
  \label{eq:an-solution}
\end{align}
Finally,
\begin{align}
  \norm{\nu R^n-\pi}_{\TV}
  =\frac12\norm{f_n}_1
  \le\frac12\norm{f_n}_2
  =\frac12\sqrt{a_n}.
  \label{eq:TV-L2}
\end{align}
With \(s=\epsilon^2/(128M^2)\), the second term in
\eqref{eq:an-solution} equals \(\epsilon^2/8\).  If
\[
  n
  \ge
  \frac{2}{\bar\phi^2}
  \log\!\left(\frac{8M}{\epsilon^2}\right),
\]
then the first term is also at most \(\epsilon^2/8\).  Thus
\(a_n\le\epsilon^2/4\), and \eqref{eq:TV-L2} is at most
\(\epsilon/4\), which is stronger than \eqref{eq:PSD-mixing-conclusion}.  Enlarging a
universal constant converts the displayed lower bound on \(n\) into
\eqref{eq:PSD-mixing-time}.
\end{proof}

\section{State-selection laws on a certified orbit}
\label{sec:selector-minorization}

Fix a power of two \(K\ge2\), an energy tolerance \(\eta>0\), and a position set \(D\).  Suppose that the terminal-depth certificate of \Cref{def:terminal-certificate} holds on momentum sets \(G_x=G_x(K,\eta)\), \(x\in D\).  For \(j\in\Z\), write
\begin{align*}
Q_{j,x}=\Law\!\left(\Pi_x\Phi_h^j(x,V)\right),
\qquad V\sim N(0,I_d).
\end{align*}
Leapfrog reversibility gives \(Q_{-j,x}=Q_{j,x}\).

\subsection{Multinomial selection}

At level \(r\in\{0,\ldots,\log_2K-1\}\), let \(B_r=1\) if the newly added block is placed to the left and \(B_r=0\) otherwise.  Since every bit sequence reaches size \(K\), the number of final orbit points to the left of the initial root is
\begin{align*}
A=\sum_{r=0}^{\log_2K-1}2^rB_r
\sim\operatorname{Unif}\{0,\ldots,K-1\}.
\end{align*}
The completed orbit is
\begin{align*}
I_A=\{-A,-A+1,\ldots,K-1-A\}.
\end{align*}
On \(G_x\), every orbit weight lies between \(e^{-H(x,v)-\eta}\) and \(e^{-H(x,v)+\eta}\).  Hence each multinomial probability is at least \(e^{-2\eta}/K\), and the conditional selector law has the decomposition
\begin{align}
\mu_{(x,v)}^{I_A}
=e^{-2\eta}\operatorname{Unif}(I_A)
+(1-e^{-2\eta})R_{x,v,A}
\label{eq:mul-decomposition}
\end{align}
for a probability law \(R_{x,v,A}\).

Let \(C\) be independent and uniform on \(\{0,\ldots,K-1\}\).  Under the uniform component, the signed selected index is \(J=C-A\).  Therefore
\begin{align}
\lambda_{\mathrm M,K}(j)
=\Pp(J=j)
=\frac{K-|j|}{K^2},
\qquad |j|\le K-1.
\label{eq:mul-time-law}
\end{align}
Define
\begin{align*}
Q_{\mathrm M,x}
=\sum_{j=-(K-1)}^{K-1}\lambda_{\mathrm M,K}(j)Q_{j,x},
\end{align*}
\begin{align*}
\widetilde Q_{\mathrm M,x}(B)
=\E\!\left[
\1_{G_x}(V)
\sum_j\lambda_{\mathrm M,K}(j)
\1_B\!\left(\Pi_x\Phi_h^j(x,V)\right)
\right].
\end{align*}

\begin{lemma}[Multinomial minorization]
\label{lem:mul-minorization}
For every \(x\in D\) and measurable \(B\subseteq\R^d\),
\begin{align*}
\PM(x,B)\ge e^{-2\eta}\widetilde Q_{\mathrm M,x}(B).
\end{align*}
\end{lemma}

\begin{proof}
Fix \(v\in G_x\).  Equation \eqref{eq:mul-decomposition} assigns mass \(e^{-2\eta}\) to uniform selection on \(I_A\).  Averaging that component over the independent fair bits gives \eqref{eq:mul-time-law}.  Restricting to \(G_x\) and integrating over the Gaussian momentum proves the claim.
\end{proof}

\subsection{Biased progressive sampling}

At the final accepted doubling, the old orbit and the new half each contain \(K/2\) points.  On \(G_x\),
\begin{align*}
\frac{W_z(I^{\mathrm{new}})}{W_z(I^{\mathrm{old}})}\ge e^{-2\eta},
\end{align*}
so the final replacement probability is at least \(e^{-2\eta}\).  Conditional on replacement, every point in the new half has probability at least \(e^{-2\eta}/(K/2)\).  Thus the final BPS output law, conditional on the completed orbit, dominates
\begin{align}
e^{-4\eta}\operatorname{Unif}(I^{\mathrm{new}}).
\label{eq:bps-final-uniform}
\end{align}
This conclusion is independent of all previous candidate updates, because a final replacement overwrites the previous candidate.

Set \(n=K/2\).  Before the last doubling, the number \(A\) of old points to the left of zero is uniform on \(\{0,\ldots,n-1\}\).  If the last extension is to the right and \(C\) is uniform on \(\{0,\ldots,n-1\}\), the selected index under the uniform component is $J=n-A+C.$ For \(j\in\{1,\ldots,K-1\}\), the number of pairs \((a,c)\) producing \(j\) equals \(\min\{j,K-j\}\).  Accounting for the probability \(1/2\) of a right extension and reflecting the left-extension case gives
\begin{align}
\lambda_{\mathrm B,K}(0)=0,
\qquad
\lambda_{\mathrm B,K}(j)
=\frac{2}{K^2}\min\{|j|,K-|j|\},
\quad 1\le|j|\le K-1.
\label{eq:bps-time-law}
\end{align}
The positive side has total mass one half because
\(
\sum_{j=1}^{K-1}\min\{j,K-j\}=K^2/4
\).

Define
\begin{align*}
Q_{\mathrm B,x}
=\sum_{j=-(K-1)}^{K-1}\lambda_{\mathrm B,K}(j)Q_{j,x},
\end{align*}
\begin{align*}
\widetilde Q_{\mathrm B,x}(B)
=\E\!\left[
\1_{G_x}(V)
\sum_j\lambda_{\mathrm B,K}(j)
\1_B\!\left(\Pi_x\Phi_h^j(x,V)\right)
\right].
\end{align*}

\begin{lemma}[BPS final-update minorization]
\label{lem:bps-minorization}
For every \(x\in D\) and measurable \(B\subseteq\R^d\),
\begin{align*}
\PB(x,B)\ge e^{-4\eta}\widetilde Q_{\mathrm B,x}(B).
\end{align*}
\end{lemma}

\begin{proof}
Fix \(v\in G_x\) and condition on the doubling bits.  Equation \eqref{eq:bps-final-uniform} supplies an \(e^{-4\eta}\)-mass uniform component on the last new half.  Averaging that component over the fair bits gives \eqref{eq:bps-time-law}.  Restrict to \(G_x\) and integrate over the momentum.
\end{proof}

\subsection{Mass in useful time bands}

The following estimates quantify selector mass in fixed-fraction and fixed-physical-time portions of a certified orbit.

\begin{lemma}[Mid-orbit mass]
\label{lem:mid-orbit-mass}
If \(K\ge8\) and
\begin{align*}
\mathcal J_{\mathrm{mid}}
=\{K/4,K/4+1,\ldots,K/2-1\},
\end{align*}
then, for \(a\in\{\mathrm M,\mathrm B\}\),
\begin{align*}
\sum_{j\in\mathcal J_{\mathrm{mid}}}\lambda_{a,K}(j)\ge\frac18.
\end{align*}
\end{lemma}

\begin{proof}
For every \(j\in\mathcal J_{\mathrm{mid}}\),
\[
\lambda_{\mathrm M,K}(j)=\frac{K-j}{K^2}\ge\frac1{2K},
\qquad
\lambda_{\mathrm B,K}(j)=\frac{2j}{K^2}\ge\frac1{2K}.
\]
The set has \(K/4\) elements, so the sum is at least \(1/8\).
\end{proof}

\begin{lemma}[Short-band mass inside a longer orbit]
\label{lem:short-band-mass}
Let \(0<r\le1/4\), assume \(rK\ge8\), and define
\begin{align*}
\mathcal J(r)
=\{j\in\N:\lceil rK/2\rceil\le j\le\lfloor rK\rfloor\}.
\end{align*}
There are universal constants \(c_1,c_2>0\) such that
\begin{align}
\sum_{j\in\mathcal J(r)}\lambda_{\mathrm M,K}(j)\ge c_1r,
\qquad
\sum_{j\in\mathcal J(r)}\lambda_{\mathrm B,K}(j)\ge c_2r^2.
\label{eq:short-band-masses}
\end{align}
\end{lemma}

\begin{proof}
The interval \(\mathcal J(r)\) contains at least \(rK/4\) integers.  Since \(j\le rK\le K/4\),
\[
\lambda_{\mathrm M,K}(j)\ge\frac{3}{4K},
\qquad
\lambda_{\mathrm B,K}(j)=\frac{2j}{K^2}\ge\frac r{K}
\]
for every \(j\in\mathcal J(r)\).  Summation gives \eqref{eq:short-band-masses}, for example with \(c_1=3/16\) and \(c_2=1/4\).
\end{proof}

\section{Proofs for the non-Gaussian specializations}
\label{app:non-gaussian-proofs}

\subsection{Koopman spectral representation and profile crossing}
\label{app:koopman-proof}

Define the Koopman group
\begin{align*}
  (\mathcal U_t f)(z)
  =
  f(\flow_tz),
  \qquad
  f\in L^2(\bar\pi).
\end{align*}
Hamiltonian volume preservation and invariance of \(H\) make
\((\mathcal U_t)_{t\in\R}\) a strongly continuous unitary group.  Let $  f_i(x,v)=x_i-\E_\pi X_i.$

\begin{lemma}[Position spectral measure]
\label{lem:position-spectral-measure}
There is a finite positive measure \(\nu_U\) on \([0,\infty)\) for
which
\eqref{eq:koopman-correlation}--\eqref{eq:koopman-profile} hold.
\end{lemma}

\begin{proof}
By the spectral theorem for strongly continuous unitary groups
\cite{ReedSimon1980}, for each \(f_i\) there is a finite positive
measure \(\mu_i\) on \(\R\) such that
\[
  \ip{
    f_i
  }{
    \mathcal U_tf_i
  }_{L^2(\bar\pi)}
  =
  \int_\R
  e^{\mathrm i\omega t}
  \mu_i(\dd\omega).
\]
Momentum reversal \(S(x,v)=(x,-v)\) satisfies $  S\flow_tS=\flow_{-t},$ while \(f_i\circ S=f_i\).  Hence the correlation is real and even.
Replacing \(\mu_i\) by its symmetrization and pushing the positive and
negative halves to \([0,\infty)\) gives a positive measure \(\nu_i\)
such that
\[
  \ip{
    f_i
  }{
    \mathcal U_tf_i
  }
  =
  \int_{[0,\infty)}
  \cos(\omega t)\,
  \nu_i(\dd\omega).
\]
Set $  \nu_U=\sum_{i=1}^d\nu_i.$ This proves \eqref{eq:koopman-correlation}.

The generator \(A\) of the Koopman group is the Hamiltonian transport
operator
\[
  Af
  =
  v^\top\nabla_xf
  -
  \nabla U(x)^\top\nabla_vf.
\]
For \(f_i\), we have $  Af_i=v_i.$ Thus \(f_i\) lies in the generator domain and
\[
  \int\omega^2\nu_i(\dd\omega)
  =
  \norm{
    Af_i
  }_{L^2(\bar\pi)}^2
  =
  \E V_i^2
  =
  1.
\]
Summing gives the second identity in
\eqref{eq:koopman-profile}.

Differentiation of \eqref{eq:koopman-correlation} is justified by
Cauchy--Schwarz and the finite second spectral moment.  Since
\[
  R(t)
  =
  2\sum_{i=1}^d\norm{f_i}_2^2
  -
  2\sum_{i=1}^d
  \ip{
    f_i
  }{
    \mathcal U_tf_i
  },
\]
we obtain
\[
  \frac12R'(t)
  =
  \int_{(0,\infty)}
  \omega\sin(\omega t)\,
  \nu_U(\dd\omega).
\]
The identity \(u=R'/2\) proves
\eqref{eq:koopman-profile}.
\end{proof}

\begin{lemma}[Positive-frequency crossing]
\label{lem:frequency-crossing}
Let \(a>0\), and let \(\nu\) be a nonzero finite positive measure on
\([a,\infty)\) satisfying $\int_a^\infty
  \omega^2\nu(\dd\omega)
  <
  \infty.$ Define 
  \begin{align*}
  g(t)
  =
  \int_a^\infty
  \omega\sin(\omega t)\,
  \nu(\dd\omega).
\end{align*}
Then \(g(t)<0\) for at least one $t\in(0,2\pi/a).$
\end{lemma}

\begin{proof}
Put
\[
  T=\frac{2\pi}{a},
  \qquad
  w(t)=1-\cos(at).
\]
For \(\omega>a\), direct integration gives
\begin{align*}
  \int_0^T
  (1-\cos(at))
  \omega\sin(\omega t)
  \dd t
  =
  -
  \frac{
    a^2\{1-\cos(\omega T)\}
  }{
    \omega^2-a^2
  }
  \le0.
\end{align*}
At \(\omega=a\), the integral is zero, agreeing with the continuous
extension of the right-hand side.  The finite second moment permits
Fubini, so
\begin{align}
  \int_0^T
  w(t)g(t)\dd t
  \le0.
  \label{eq:weighted-profile-negative}
\end{align}
On the other hand,
\[
  g(0)=0,
  \qquad
  g'(0)
  =
  \int\omega^2\nu(\dd\omega)
  >
  0.
\]
Thus \(g(t)>0\) for all sufficiently small positive \(t\).  If
\(g\ge0\) throughout \([0,T]\), then \(w>0\) on \((0,T)\) would make
the left-hand side of
\eqref{eq:weighted-profile-negative} strictly positive, a
contradiction.
\end{proof}

\begin{proof}[Proof of \Cref{thm:koopman-crossing}]
The representation is
\Cref{lem:position-spectral-measure}.  Under
\eqref{eq:position-spectral-gap}, discard the atom at zero, which does
not contribute to \(u\), and apply
\Cref{lem:frequency-crossing} with \(a=\omega_0\).  The remaining
measure is nonzero because its second moment equals \(d\).
\end{proof}

\subsection{One-dimensional and independent-product targets}
\label{app:product-proof}

\begin{lemma}[Frequency gap for a one-dimensional strongly convex oscillator]
\label{lem:one-dimensional-gap}
Let \(W\) satisfy
\eqref{eq:one-dimensional-curvature}.  The Koopman spectral measure of
the one-dimensional position observable is supported in
\begin{align}
  \{0\}\cup[\sqrt m,\infty).
  \label{eq:one-dimensional-spectral-gap}
\end{align}
\end{lemma}

\begin{proof}
Fix a nonconstant energy level \(E\).  The corresponding orbit is
periodic.  Let \(P(E)\) be its period and
\[
  \omega(E)=\frac{2\pi}{P(E)}.
\]
Choose a turning point as time zero.  Between consecutive turning
points,
\[
  y(t)=\dot x(t)
\]
has no interior zero and satisfies
\begin{align*}
  y''(t)
  +
  W''(x(t))y(t)
  =
  0.
\end{align*}
Sturm comparison with
\[
  z(t)=\sin(\sqrt m\,t)
\]
shows that the next zero of \(y\) occurs no later than
\(\pi/\sqrt m\).  Indeed, otherwise the Wronskian
\[
  y'(t)z(t)-y(t)z'(t)
\]
would be nonincreasing from zero and strictly positive at
\(\pi/\sqrt m\), a contradiction.  Thus
\[
  \frac{P(E)}2
  \le
  \frac{\pi}{\sqrt m},
\]
and therefore
\begin{align*}
  \omega(E)
  \ge
  \sqrt m.
\end{align*}

Conditional on \(E\), the stationary phase is uniform on the orbit.
Write the Fourier series
\begin{align*}
  x_E(s)
  =
  \bar x_E
  +
  \sum_{n\ne0}
  c_{n,E}
  e^{\mathrm i n\omega(E)s}.
\end{align*}
The conditional profile is
\begin{align*}
  u_E(t)
  &=
  \frac1{P(E)}
  \int_0^{P(E)}
  \dot x_E(s)
  \{x_E(s+t)-x_E(s)\}
  \dd s
  \notag\\
  &=
  2
  \sum_{n=1}^\infty
  n\omega(E)
  |c_{n,E}|^2
  \sin(n\omega(E)t).
\end{align*}
Every coefficient is nonnegative, and every nonzero frequency
\(n\omega(E)\) is at least \(\sqrt m\).  Disintegrating the canonical
measure over energy levels produces a positive spectral measure with
support \eqref{eq:one-dimensional-spectral-gap}.  The possible
zero-frequency component is the energy-dependent orbit mean
\(\bar x_E\), which disappears after differentiation and does not
contribute to the profile.
\end{proof}

\begin{proof}[Proof of \Cref{cor:product-arbitrary-kappa}]
Apply
\Cref{thm:koopman-crossing,lem:one-dimensional-gap} to obtain
\(t_W<2\pi/\sqrt m\) such that
\[
  u_W(t_W)<0.
\]
Continuity gives constants \(\rho_W,w_W>0\) such that
\[
  u_W(t)
  \le
  -\frac{\rho_W}{\sqrt m}
  \qquad
  \text{whenever}
  \qquad
  |t-t_W|
  \le
  \frac{w_W}{\sqrt m}.
\]

For the product potential
\eqref{eq:iid-product-potential}, the Hamiltonian coordinates and
stationary law factorize, so
\[
  u_d(t)
  =
  \sum_{i=1}^d u_W(t)
  =
  d\,u_W(t).
\]

We now prove \eqref{eq:product-Cp}.  Normalize
\(W(x_\star)=0\), and let
\[
  E_i
  =
  W(X_{0,i})
  +
  \frac12V_{0,i}^2
\]
be the conserved one-coordinate energy.  Strong convexity and energy
conservation give, uniformly in time,
\[
  |V_{0,i}(X_{t,i}-X_{0,i})|
  +
  |V_{t,i}(X_{t,i}-X_{0,i})|
  \le
  \frac{CE_i}{\sqrt m}.
\]
Under the canonical one-coordinate law, \(E_i\) has an exponential
moment in a fixed neighborhood of zero.  Indeed, for every
\(\lambda\in(0,1)\),
\[
  \E e^{\lambda E_i}
  =
  \frac{
    \displaystyle
    \int_{\R^2}
    \exp\left\{
      -(1-\lambda)
      \left(
        W(x)+\frac12v^2
      \right)
    \right\}
    \dd x\,\dd v
  }{
    \displaystyle
    \int_{\R^2}
    \exp\left\{
      -
      \left(
        W(x)+\frac12v^2
      \right)
    \right\}
    \dd x\,\dd v
  }
  <
  \infty.
\]
Consequently, the centered one-coordinate contribution to either
endpoint diagnostic has subexponential norm at most
\(C_{W,S}/\sqrt m\).  The coordinate contributions are independent. We use the standard
moment form of Bernstein's inequality: for independent centered subexponential variables
\(Z_i\),
\[
  \left\|\sum_{i=1}^d Z_i\right\|_{L^p}
  \le
  C\left[
    \sqrt p\left(\sum_{i=1}^d\|Z_i\|_{\psi_1}^2\right)^{1/2}
    +p\max_i\|Z_i\|_{\psi_1}
  \right],
  \qquad p\ge2.
\]
Applying this with \(\|Z_i\|_{\psi_1}\le C_{W,S}/\sqrt m\) yields
\[
  \norm{
    D_t^\sigma-u_d(t)
  }_{L^p}
  \le
  \frac{C_{W,S}}{\sqrt m}
  \bigl(\sqrt{dp}+p\bigr),
  \qquad
  \sigma\in\{-,+\},
\]
which is \eqref{eq:product-Cp}.

The displayed negative interval has the required \(d/\sqrt m\) margin.  Let
\((a,b)\) be the first connected component of
\(\{t>0:u_W(t)<0\}\), and choose a nondegenerate compact interval
\[
  I\Subset\bigl(a,\min\{b,2a\}\bigr).
\]
Fix a depth \(k\ge2\) and, for \(1\le r<k\), put
\[
  c_r=\frac{2^r-1}{2^k-1}.
\]
Because \(c_r\le c_{k-1}<1/2\) and every \(T\in I\) satisfies \(T<2a\), one has
\(c_rT<a\).  Let
\[
  Z=\{t\in(0,a):u_W(t)=0\},
  \qquad
  Z_r=\{T\in I:c_rT\in Z\}.
\]
By hypothesis, \(Z\) is relatively closed with empty interior in \((0,a)\).  The map
\(T\mapsto c_rT\) is a homeomorphism from \(I\) onto \(c_rI\), so each \(Z_r\) is
relatively closed with empty interior in \(I\), and hence is nowhere dense in \(I\).
The finite union \(\bigcup_{r=1}^{k-1}Z_r\) is therefore nowhere dense.  Its complement
contains a nonempty open subinterval, inside which choose a nondegenerate compact
interval \(I'\).

Since \((a,b)\) is the first negative component, \(u_W(t)\ge0\) for \(0<t<a\).
For \(T\in I'\), the exclusion of every \(Z_r\) therefore gives
\(u_W(c_rT)>0\) for all \(r<k\), while \(I'\subset(a,b)\) gives
\(u_W(T)<0\).  Continuity and compactness now yield uniform positive preterminal and
negative terminal margins on \(I'\).  The image
\[
  \left\{\frac{T}{2^k-1}:T\in I'\right\}
\]
is a positive-Lebesgue-measure interval of step sizes.  Because \(I'\subset I\) and
\(I\) is fixed, these step sizes become arbitrarily small as \(k\to\infty\).  The final
statement follows from \Cref{thm:adaptive} under its stated numerical-fidelity and energy
hypotheses.
\end{proof}

\subsection{Near-isotropic coupled targets}
\label{app:near-isotropic-proof}

\begin{proof}[Proof of \Cref{thm:near-isotropic}]
Let \(J_t\) be the Jacobi field from the profile representation, and set $s=\omega t$ and $\mathcal J_s=\omega J_{s/\omega}$. Then
\begin{align*}
  \mathcal J_s''+(I_d+E_s)\mathcal J_s=0,
  \qquad
  \mathcal J_0=0,
  \qquad
  \mathcal J_0'=I_d,
\end{align*}
where
\[
  E_s
  =
  \omega^{-2}\nabla^2U(X_{s/\omega})-I_d,
  \qquad
  \norm{E_s}_{\mathrm{op}}\le\eta.
\]
Duhamel's formula~\citep{teschl2012ordinary} gives
\begin{align}
  \mathcal J_s
  =
  \sin(s)I_d
  -
  \int_0^s\sin(s-r)E_r\mathcal J_r\dd r.
  \label{eq:J-Duhamel}
\end{align}
For \(0\le s\le2\pi\), define
\[
  M(s)=\sup_{0\le r\le s}\norm{\mathcal J_r}_{\mathrm{op}}.
\]
Taking operator norms in \eqref{eq:J-Duhamel} and using
\(|\sin(s-r)|\le1\) yields
\[
  M(s)
  \le
  1+\eta\int_0^sM(r)\dd r.
\]
Gronwall's inequality therefore gives
\[
  M(s)\le e^{\eta s}\le e^{2\pi},
\]
because \(0\le\eta<1\).  Substituting this estimate into
\eqref{eq:J-Duhamel} gives
\begin{align*}
  \norm{\mathcal J_s-\sin(s)I_d}_{\mathrm{op}}
  \le
  \eta\int_0^sM(r)\dd r \le
  \eta s e^{\eta s}
  \le
  2\pi e^{2\pi}\eta.
\end{align*}
Since $\omega u(t)=\E\Tr\mathcal J_{\omega t}$ and \(|\Tr A|\le d\norm A_{\mathrm{op}}\), taking expectations proves
\eqref{eq:near-isotropic-profile} with \(C_0=2\pi e^{2\pi}\).  At each grid time,
the transformed profile differs from the corresponding reference sine value by at most
\(C_0\eta d\).  Thus a reference margin \(\rho_{\sin}d\) with
\(\rho_{\sin}>C_0\eta\) leaves an \(\omega\)-scaled margin
\((\rho_{\sin}-C_0\eta)d\). Since
\(m=(1-\eta)\omega^2\), the normalization in
\Cref{def:profile-separated} multiplies this by \(\sqrt{1-\eta}\), giving
\(\rho_0=\sqrt{1-\eta}(\rho_{\sin}-C_0\eta)\).  This proves the final assertion.
\end{proof}

\subsection{Tangent-flow and adiabatic stability}
\label{app:tangent-stability-proof}

\begin{proposition}[Tangent-flow terminal-depth certificate]
\label{prop:tangent-certificate}
Assume the hypotheses and parameter choices of
\Cref{thm:tangent-adaptive}, and set $K_\star=2^{k_\star}.$ Fix $\delta\in(0,1/8)$, and $\vartheta\in(0,1).$ After increasing the logarithmic constant in
\eqref{eq:adaptive-ell} by a multiple of
\(1+\log(1/(\delta\vartheta))\), there are a measurable position set
\(D_{\mathrm{ad}}\) and measurable momentum sets
\(G_x^{\mathrm{ad}}\), \(x\in D_{\mathrm{ad}}\), such that
\[
  \pi(D_{\mathrm{ad}}^c)
  \le
  \vartheta s,
  \qquad
  \sup_{x\in D_{\mathrm{ad}}}
  \Pp_V\!\left[
    (G_x^{\mathrm{ad}})^c
  \right]
  \le
  \delta,
  \qquad
  s=\frac{\epsilon^2}{128M^2}.
\]
Every \(v\in G_x^{\mathrm{ad}}\) is
\((K_\star,\log2)\)-certified, and the orbit builder terminates through
the U-turn criterion before the maximum-depth cap.
\end{proposition}

\begin{proof}
Let \(\mathscr I_{K_\star}\) be the family of checked intervals and set
\[
  \zeta
  =
  \frac{\delta\vartheta s}{8},
  \qquad
  p
  =
  \max\left\{
    2,\,
    \log\frac{
      64|\mathscr I_{K_\star}|
    }{
      \zeta
    }
  \right\}.
\]
Since
\[
  K_\star h
  \le
  \frac{S}{\sqrt m}+h
\]
and \(h\sqrt L\le1\), all relevant scaled times lie in
\(S_\kappa^+\).  Moreover,
\[
  K_\star
  \le
  1+
  C S\sqrt\kappa
  \left(
    \frac{1+\bar\gamma}{\eta_{\mathrm{ad}}}
  \right)^{1/2}
  d_\ell^{1/4}\ell^{3/4}.
\]
The logarithmic condition
\eqref{eq:adaptive-ell} therefore implies \(p\le C\ell\).

The exact diagnostic concentration estimate and
\eqref{eq:tangent-implies-Cp} imply, outside an event of probability
at most \(\zeta/8\),
\[
  \max_{I\in\mathscr I_{K_\star}}
  \max_{\sigma\in\{-,+\}}
  |D_I^\sigma-u(\ell_h(I))|
  \le
  \frac{
    \mathfrak C_S(U)
  }{
    \sqrt m
  }
  \bigl(\sqrt{dp}+p\bigr).
\]
The phase-norm moment bound gives $B(Z_0)\le C(\sqrt d+\sqrt p)$, outside another event of probability at most \(\zeta/8\).

Let \(A_\ell(h)\) denote the right-hand side of the stationary
single-step energy estimate.  Since
\(K_\star h\le S/\sqrt m+h\), the leading accumulated term satisfies
\begin{align}
  K_\star(1+\gamma)
  h^3\ell^{3/2}
  L^{3/2}d_\ell^{1/2}
  \le
  Cc
  (1+S\sqrt\kappa)
  \eta_{\mathrm{ad}}
  \frac{1+\gamma}{1+\bar\gamma}
  \le
  \frac1{64}.
  \label{eq:adaptive-energy-leading-general}
\end{align}
after reducing the universal constant \(c\). If
\(A_3,A_5,A_7\) denote the accumulated \(h^3,h^5,h^7\) contributions,
respectively, then the step-size choice gives
\[
  \frac{A_5}{A_3}
  =\frac{h^2L d_\ell^{1/2}}{\ell}
  =\frac{c\eta_{\mathrm{ad}}}{(1+\bar\gamma)\ell^{5/2}},
  \qquad
  \frac{A_7}{A_3}
  =\frac{h^4L^2d_\ell}{\ell^{3/2}}
  =\frac{c^2\eta_{\mathrm{ad}}^2}{(1+\bar\gamma)^2\ell^{9/2}}.
\]
Thus both are bounded by a small universal multiple of the leading term.  Hence
\begin{align}
  \frac{
    K_\star A_\ell(h)
  }{
    \log2
  }
  \le
  e^{-6}.
  \label{eq:adaptive-energy-small-general}
\end{align}
The whole-orbit energy lemma therefore gives
\[
  \bar\pi\!\left[
    \max_{|j|\le K_\star}
    |H(\Phi_h^jZ_0)-H(Z_0)|
    >
    \log2
  \right]
  \le
  4K_\star e^{-6\ell}
  \le
  \frac{\zeta}{8}.
\]

On the energy event, the leapfrog-to-flow comparison and
\eqref{eq:tangent-implies-Ep} give
\[
  \max_{I\in\mathscr I_{K_\star}}
  \max_{\sigma\in\{-,+\}}
  |\widehat D_I^\sigma-D_I^\sigma|
  \le
  \frac{
    C\mathfrak N_S(U)
  }{
    \sqrt m
  }
  (mh^2)
  (d+p)^{3/2}.
\]
The dimension assumption
\eqref{eq:adaptive-dimension} implies
\[
  \mathfrak C_S(U)
  \frac{\sqrt{dp}+p}{d}
  \le
  \frac{\rho_0}{8}.
\]
The step-size choice
\eqref{eq:adaptive-h}, the definition
\eqref{eq:eta-ad-explicit}, and \(p\le C\ell\) imply
\[
  \mathfrak N_S(U)(mh^2)
  \frac{(d+p)^{3/2}}{d}
  \le
  \frac{\rho_0}{8}.
\]
Consequently,
\[
  \max_{I,\sigma}
  |\widehat D_I^\sigma-u(\ell_h(I))|
  \le
  \frac{\rho_0}{4}
  \frac{d}{\sqrt m}.
\]

The typical-endpoint estimate and the deterministic short-interval
lemma certify every nontrivial checked interval shorter than
\(T_{\mathrm{sm}}\); singleton intervals are nonturning by definition. For every checked dyadic interval of depth
\(r<k_\star\) and duration at least \(T_{\mathrm{sm}}\), profile
separation gives
\[
  \widehat D_I^\sigma
  \ge
  \frac{3\rho_0}{4}
  \frac{d}{\sqrt m}
  >
  0.
\]
At depth \(k_\star\),
\[
  \widehat D_I^\sigma
  \le
  -
  \frac{3\rho_0}{4}
  \frac{d}{\sqrt m}
  <
  0.
\]
The deterministic terminal-depth proposition implies that every
doubling realization returns size \(K_\star\) through the U-turn
criterion.

The union of the preceding failures has phase-space probability at
most \(\zeta\).  Let \(r_{\mathrm{ad}}(x)\) be the conditional failure
probability over the refreshed momentum, and define
\[
  D_{\mathrm{ad}}
  =
  \{x:r_{\mathrm{ad}}(x)\le\delta\},
  \qquad
  G_x^{\mathrm{ad}}
  =
  \{v:\text{all certificate conditions hold}\}.
\]
Fubini's theorem and Markov's inequality give
\[
  \sup_{x\in D_{\mathrm{ad}}}
  \Pp_V\!\left[
    (G_x^{\mathrm{ad}})^c
  \right]
  \le
  \delta,
  \qquad
  \pi(D_{\mathrm{ad}}^c)
  \le
  \frac{\zeta}{\delta}
  \le
  \vartheta s.
\]
\end{proof}

\begin{proof}[Proof of \Cref{thm:tangent-adaptive}]
Set $K_\star=2^{k_\star},$ $T_\star=(K_\star-1)h$ and $a_\star=\sqrt m\,T_\star.$ Since \(k_\star\ge2\),
\[
  T_\star
  <
  K_\star h
  \le
  \frac43T_\star.
\]
Let
\[
  t_{\mathrm{ov}}
  =
  \frac{
    c_{\mathrm{ov}}
  }{
    \sqrt L(1+\gamma)^{1/3}
  }.
\]
The terminal profile is negative, whereas the universal initial
positivity result gives \(u(t)>0\) through
\(\pi/(2\sqrt L)\).  Therefore
\[
  T_\star
  >
  \frac{\pi}{2\sqrt L},
  \qquad
  \frac{\pi}{2\sqrt\kappa}
  <
  a_\star
  \le
  S.
\]
After reducing \(c_{\mathrm{ov}}\), define
\[
  r_\star
  =
  \frac{
    t_{\mathrm{ov}}
  }{
    K_\star h
  }
  \le
  \frac14.
\]
The preceding bounds imply
\[
  r_\star
  \ge
  \frac{
    c
  }{
    a_\star\sqrt\kappa(1+\gamma)^{1/3}
  }.
\]
The step-size and dimension assumptions ensure
\[
  r_\star K_\star
  =
  \frac{t_{\mathrm{ov}}}{h}
  \ge
  8.
\]
Let
\[
  \mathcal J_\star
  =
  \left\{
    j:
    \left\lceil
      \frac{r_\star K_\star}{2}
    \right\rceil
    \le
    j
    \le
    \left\lfloor
      r_\star K_\star
    \right\rfloor
  \right\}.
\]
The selector band-mass bounds give
\[
  a_{\mathrm M,K_\star}
  \ge
  cr_\star,
  \qquad
  a_{\mathrm B,K_\star}
  \ge
  cr_\star^2.
\]
Put
\[
  \Delta_{\mathrm{ad}}
  =
  \frac{t_{\mathrm{ov}}}{128},
  \qquad
  \theta_{\mathrm{ad}}
  =
  \min\{1,\psi\Delta_{\mathrm{ad}}\}.
\]
Since
\[
  \psi=(\log2)\sqrt m,
\]
we have
\[
  \theta_{\mathrm{ad}}
  \ge
  \frac{
    c
  }{
    \sqrt\kappa(1+\gamma)^{1/3}
  }.
\]

Choose
\[
  a_{\mathrm M}=cr_\star,
  \qquad
  a_{\mathrm B}=cr_\star^2,
  \qquad
  \delta_a=\frac{3a_a}{32}.
\]
Apply
\Cref{prop:tangent-certificate} separately for the two
selectors with $\delta=\delta_a,$ and $\vartheta=\theta_{\mathrm{ad}}/64.$ This supplies the terminal-depth certificate, the \(\log2\) energy
window, and the good position set required by the terminal-depth
transfer theorem.

For every \(j\in\mathcal J_\star\),
\[
  \frac{t_{\mathrm{ov}}}{2}
  \le
  jh
  \le
  t_{\mathrm{ov}}.
\]
The fixed-index proposal-overlap estimate gives
\begin{align*}
  \ov(Q_{j,x},Q_{j,y})
  \ge
  \frac34 \quad\text{whenever}\quad
  \norm{x-y}
  \le
  \Delta_{\mathrm{ad}}.
\end{align*}
The terminal-depth transfer theorem therefore yields
\begin{align*}
\Phi_s(R_{\mathrm M})
\ge
\frac{
    c
  }{
    a_\star\kappa(1+\gamma)^{2/3}
  },
\quad\text{and}\quad
  \Phi_s(R_{\mathrm B})
  \ge
  \frac{
    c
  }{
    a_\star^2\kappa^{3/2}(1+\gamma)
  }.
\end{align*}
The positive-kernel warm-start mixing proposition gives
\eqref{eq:adaptive-transitions-M} and
\eqref{eq:adaptive-transitions-B}.

Finally,
\[
  K_\star
  \le
  1+\frac{T_\star}{h}
  \le
  C a_\star\sqrt\kappa
  \left(
    \frac{1+\bar\gamma}{\eta_{\mathrm{ad}}}
  \right)^{1/2}
  d_\ell^{1/4}\ell^{3/4},
\]
which proves
\eqref{eq:adaptive-K-explicit}.The expected and high-probability work bounds follow from
\Cref{prop:exceptional-work}.  Under
\(K_{\mathrm{cap}}\le C_{\mathrm{cap}}K_\star\), multiplying the
transition bounds by the certified-orbit estimate proves the
conditional deterministic bounds
\eqref{eq:adaptive-complexity-M} and
\eqref{eq:adaptive-complexity-B}.
\end{proof}

\begin{proof}[Proof of \Cref{cor:polynomial-stability}]
Substitute
\eqref{eq:polynomial-tangent-stability} into
\eqref{eq:concentration-envelope} and
\eqref{eq:numerical-envelope}.  This gives
\eqref{eq:polynomial-concentration-envelope} and
\eqref{eq:polynomial-numerical-envelope}.  Substitution into
\eqref{eq:adaptive-dimension},
\eqref{eq:eta-ad-explicit}, and
\eqref{eq:adaptive-h} proves the remaining claims.
\end{proof}

\begin{proof}[Proof of \Cref{prop:adiabatic-stability}]
Let \((\xi_s,\eta_s)\) solve the variational equation
\[
  \dot\xi_s=\eta_s,
  \qquad
  \dot\eta_s=-H_s\xi_s.
\]
Define
\[
  \mathcal E_s
  =
  \norm{\eta_s}^2+\xi_s^\top H_s\xi_s.
\]
The cross terms cancel, and hence
\[
  \mathcal E_s'=\xi_s^\top\dot H_s\xi_s,
  \qquad
  |\mathcal E_s'|
  \le
  \norm{H_s^{-1/2}\dot H_sH_s^{-1/2}}_{\mathrm{op}}\mathcal E_s.
\]
Integration and \eqref{eq:adiabatic-condition} give
\[
  \mathcal E_t\le\kappa^{A_0}\mathcal E_0.
\]
Because \(I_d\preceq H_s\preceq\kappa I_d\),
\[
  \norm{(\xi_t,\eta_t)}^2
  \le \mathcal E_t
  \le \kappa^{A_0+1}\norm{(\xi_0,\eta_0)}^2.
\]
Taking the supremum proves \eqref{eq:adiabatic-A}.
\end{proof}

\begin{proof}[Proof of \Cref{cor:fast-selected-stability}]
For a variational solution
\[
  \dot\xi=\eta,
  \qquad
  \dot\eta=-H_s\xi,
  \qquad
  I_d\preceq H_s\preceq\kappa I_d,
\]
let $E_s
  =
  \kappa\norm{\xi_s}^2
  +
  \norm{\eta_s}^2.$ Then $E_s'
  =
  2\ip{\eta_s}{(\kappa I_d-H_s)\xi_s}
  \le
  \sqrt\kappa\,E_s.$ Gronwall's inequality and comparison of the weighted and Euclidean
norms give
\[
  \mathfrak A_S(U)
  \le
  \sqrt\kappa
  \exp\!\left(
    \frac12\sqrt\kappa S_\kappa^+
  \right)
  =
  e^{1/2}\sqrt\kappa
  \exp\!\left(
    \frac12S\sqrt\kappa
  \right).
\]
Substituting \(S\le b/\sqrt\kappa\) proves
\eqref{eq:fast-selected-polynomial-A}.
\end{proof}

\begin{proposition}[Parametric amplification under curvature bounds]
\label{prop:parametric-amplification}
For every \(\kappa\ge4\), there is a periodic positive
piecewise-constant coefficient $a_\kappa:\R\to[1,\kappa]$ such that the fundamental matrix of
\begin{align*}
  q''(s)+a_\kappa(s)q(s)=0
\end{align*}
has norm at least $\exp(c\sqrt\kappa\,s)$ along an unbounded sequence of times, where \(c>0\) is universal.
Smooth positive coefficients with the same conclusion are obtained by
arbitrarily small smoothing.  
\end{proposition}

\begin{proof}
Let $\omega_1={\sqrt\kappa}/{2},$ and $\omega_2=\sqrt\kappa.$ Over a quarter-period of the oscillator with frequency \(\omega\), the
phase vector \((q,q')\) is multiplied by
\[
  Q_\omega
  =
  \begin{pmatrix}
    0&\omega^{-1}\\
    -\omega&0
  \end{pmatrix}.
\]
Choose \(a_\kappa=\omega_1^2=\kappa/4\) for time
\(\pi/\sqrt\kappa\), followed by
\(a_\kappa=\omega_2^2=\kappa\) for time
\(\pi/(2\sqrt\kappa)\), and repeat periodically.  One period has length $T_\kappa
  =
  {3\pi}/{2\sqrt\kappa}$ and monodromy matrix
\[
  Q_{\omega_2}Q_{\omega_1}
  =
  \begin{pmatrix}
    -1/2&0\\
    0&-2
  \end{pmatrix}.
\]
After \(n\) periods its norm is \(2^n\).  Since
\[
  2^n
  =
  \exp\!\left\{
    \frac{2\log2}{3\pi}
    \sqrt\kappa\,
    nT_\kappa
  \right\},
\]
the claimed growth follows.  Fundamental matrices depend continuously
on the coefficient in \(L^1\) on compact intervals, so smoothing the
jumps preserves a fixed fraction of the Floquet growth.
\end{proof}

\subsection{Metric-adapted targets}
\label{app:metric-adaptation-proof}

\begin{proof}[Proof of \Cref{thm:relative-hessian-stability}]
Use the canonical affine change of variables
\begin{align*}
  x=b+G^{1/2}z,
  \qquad
  p=G^{-1/2}v.
\end{align*}
Then \(p^\top Gp=\norm v^2\), and $ p^\top\dd x= v^\top\dd z, $ so the symplectic one-form is preserved.  The transformed potential is $U_G(z)=U(b+G^{1/2}z)$, with Hessian
\[
  \nabla^2U_G(z)
  =
  G^{1/2}\nabla^2U(b+G^{1/2}z)G^{1/2}.
\]
Thus \eqref{eq:relative-Hessian-condition} gives
\[
  (1-\delta)I_d
  \preceq
  \nabla^2U_G(z)
  \preceq
  (1+\delta)I_d,
\]
which proves \eqref{eq:relative-Hessian-kappa}.

Let \((\xi_t,\eta_t)\) solve the variational equation along the transformed physical
flow \(\flow_t^G\):
\[
  \dot\xi_t=\eta_t,
  \qquad
  \dot\eta_t=-H_t\xi_t,
  \qquad
  H_t=\nabla^2U_G(Z_t).
\]
Set \(E_t=\norm{\xi_t}^2+\norm{\eta_t}^2\).  Since
\(\norm{I_d-H_t}_{\mathrm{op}}\le\delta\),
\begin{align*}
  |E_t'|
  &=
  \left|
    2\ip{\xi_t}{\eta_t}-2\ip{\eta_t}{H_t\xi_t}
  \right|
  =
  2\left|\ip{\eta_t}{(I_d-H_t)\xi_t}\right|
  \\
  &\le
  2\delta\norm{\eta_t}\norm{\xi_t}
  \le
  \delta E_t.
\end{align*}
Integrating \(|(\log E_t)'|\le\delta\) on the interval joining \(0\) and \(t\)
gives \(E_t\le e^{\delta|t|}E_0\) for every \(t\in\R\).  Therefore $ \norm{D\flow_t^G}_{\mathrm{op}}
  \le
  e^{\delta|t|/2}.$
  
The stability index \(\mathfrak A_S(U_G)\) is defined after lower-curvature scaling.
Let  $m_G=1-\delta$ and $T_G=\operatorname{diag}(\sqrt{m_G}I_d,I_d)$. If \(\widetilde\flow_s^G\) denotes the scaled flow, then $\widetilde\flow_s^G
  =
  T_G\flow_{s/\sqrt{m_G}}^GT_G^{-1}$
 (up to the irrelevant translation by the minimizer).  Because
\(\norm{T_G}_{\mathrm{op}}=1\) and
\(\norm{T_G^{-1}}_{\mathrm{op}}=m_G^{-1/2}\),
\begin{align}
  \norm{D\widetilde\flow_s^G}_{\mathrm{op}}
  \le
  \frac{1}{\sqrt{1-\delta}}
  \exp\!\left\{
    \frac{\delta|s|}{2\sqrt{1-\delta}}
  \right\}.
  \label{eq:relative-scaled-tangent}
\end{align}
For the selected-horizon index, $|s|
  \le
  S_{\kappa_G}^+
  =S+\kappa_G^{-1/2}
  \le S+1.$ Taking the supremum in \eqref{eq:relative-scaled-tangent} proves
\eqref{eq:relative-Hessian-A}.

The transformed Hessian also satisfies the hypotheses of
\Cref{thm:near-isotropic} with \(\omega=1\) and \(\eta=\delta\).  Therefore
\eqref{eq:relative-Hessian-profile} follows from
\eqref{eq:near-isotropic-profile}.  If the reference sine margin is
\(\rho_{\sin}>C_0\delta\), subtracting the profile perturbation leaves the
unit-frequency margin \((\rho_{\sin}-C_0\delta)d\).  Because the lower
curvature of the transformed target is \(m_G=1-\delta\), the normalized
margin in \Cref{def:profile-separated} is
\(\rho_0=\sqrt{1-\delta}(\rho_{\sin}-C_0\delta)\).  Finally, when
\(\delta\le\delta_0<1\), the bounds in \eqref{eq:relative-Hessian-kappa} and
\eqref{eq:relative-Hessian-A} depend only on \(\delta_0\) and \(S\).  Substitution into
\Cref{thm:tangent-adaptive} proves the stated independence from the raw anisotropy,
while retaining the explicitly stated dependence on transformed Hessian regularity.

\end{proof}

\section{Proofs for the Gaussian specialization}
\label{app:gaussian-proofs}

\subsection{The Gaussian profile and separated dyadic grids}
\label{app:gaussian-profile-proof}

\begin{proof}[Proof of \Cref{thm:gaussian-arbitrary-kappa}]
Diagonalize $  \Lambda=C^{-1}$ orthogonally.  In the eigenbasis, exact Hamiltonian dynamics are
independent oscillators:
\[
  X_{i,t}
  =
  X_{i,0}\cos(\omega_it)
  +
  \frac{V_{i,0}}{\omega_i}
  \sin(\omega_it).
\]
Since \(X_{i,0}\) and \(V_{i,0}\) are independent and
\(\E V_{i,0}^2=1\),
\[
  \E\!\left[
    V_{i,0}(X_{i,t}-X_{i,0})
  \right]
  =
  \frac{\sin(\omega_it)}{\omega_i}.
\]
Summing proves the eigenvalue formula in
\eqref{eq:gaussian-profile}.  Functional calculus and
\(C=\Lambda^{-1}\) give
\[
  \Lambda^{-1/2}\sin(\Lambda^{1/2}t)
  =
  \sin(C^{-1/2}t)C^{1/2},
\]
so the trace is exactly the uniform term used by
\citet{Oberdoerster2025}.

In the lower-curvature scaled variables, the precision matrix is $B=m^{-1}\Lambda$, whose spectrum lies in \([1,\kappa]\).  The phase-space flow matrix is
\[
  \begin{pmatrix}
    \cos(B^{1/2}s)
    &
    B^{-1/2}\sin(B^{1/2}s)
    \\
    -B^{1/2}\sin(B^{1/2}s)
    &
    \cos(B^{1/2}s)
  \end{pmatrix}.
\]
The four blocks have operator norms at most
\(1,1,\sqrt\kappa,1\), respectively.  Hence the full matrix has norm
at most \(2\sqrt\kappa\), proving
\eqref{eq:gaussian-tangent-amplification}. The profile frequencies lie in
\([\sqrt m,\sqrt L]\), so
\Cref{lem:frequency-crossing} gives a negative value before
\(2\pi/\sqrt m\).

For the uniform margin, let $  R=\sqrt\kappa,
$ and $  s=\sqrt m\,t,$ and let $  \mu_\Lambda
  =
  \frac1d
  \sum_{i=1}^d
  \delta_{\omega_i/\sqrt m}.$ Then
\begin{align*}
  \frac{\sqrt m}{d}u(t)
  =
  F_{\mu_\Lambda}(s),
  \qquad
  F_\mu(s)
  =
  \int_1^R
  \frac{\sin(rs)}r
  \mu(\dd r).
\end{align*}
For every probability measure \(\mu\) on \([1,R]\),
\Cref{lem:frequency-crossing} implies
\[
  \min_{0\le s\le2\pi}
  F_\mu(s)
  <
  0.
\]
The space \(\mathcal P([1,R])\) is weakly compact, and the map $\mu\longmapsto F_\mu$ is continuous into \(C([0,2\pi])\) in the supremum norm.  Hence
\begin{align*}
  c_{\mathrm G}(\kappa)
  =
  -
  \sup_{\mu\in\mathcal P([1,R])}
  \min_{0\le s\le2\pi}
  F_\mu(s)
  >
  0.
\end{align*}
Also $|F_\mu'(s)|\le 1.$ If $  F_\mu(s_0) \le -c_{\mathrm G}(\kappa),$ then $  F_\mu(s)
  \le
  -\frac12c_{\mathrm G}(\kappa)$
whenever $ |s-s_0|
  \le
  \frac12c_{\mathrm G}(\kappa).$ Rescaling proves
\eqref{eq:gaussian-uniform-negative}, after changing
\(c_{\mathrm G}(\kappa)\) by a universal factor.

It remains to construct separated dyadic grids.  Fix the covariance
and let \((a,b)\) be the first connected component of
$  \{t>0:u(t)<0\}.$ The finite sine sum \eqref{eq:gaussian-profile} is real analytic and
not identically zero, so \(0<a<b\) and its zeros are isolated.  Choose
a nonempty compact interval $I
  \subset
  (a,\min\{b,2a\}).$ Fix a depth \(k\ge2\) and put
\[
  c_{r,k}
  =
  \frac{2^r-1}{2^k-1},
  \qquad
  1\le r<k.
\]
Then \(c_{r,k}<1/2\), so \(c_{r,k}t<a\) for every \(t\in I\).
Remove from \(I\) the finitely many points for which
$  u(c_{r,k}t)=0$ 
for some \(r<k\) with \(c_{r,k}t\ge T_{\mathrm{sm}}\).  The remaining
set is open and has positive measure.  Choose a compact interval
\(I_k\) in one of its components.  Continuity gives
\[
  \min_{t\in I_k}
  \frac{\sqrt m}{d}
  \min\left\{
    -u(t),
    \min_{\substack{
      r<k\\
      c_{r,k}t\ge T_{\mathrm{sm}}
    }}
    u(c_{r,k}t)
  \right\}
  >
  0.
\]
For
\[
  h=\frac{t}{2^k-1},
  \qquad
  t\in I_k,
\]
depth \(k\) is profile separated with this target-dependent uniform
margin.  The step-size set \(I_k/(2^k-1)\) has positive Lebesgue
measure, and taking \(k\) large makes all its step sizes arbitrarily
small.
\end{proof}

\subsection{Gaussian diagnostic concentration}
\label{app:gaussian-concentration-proof}

\begin{proof}[Proof of \Cref{prop:gaussian-margin-certificate}]
Diagonalize the precision matrix and write $ \Lambda
  =
  \operatorname{diag}(
    \omega_1^2,\ldots,\omega_d^2
  ).$ At stationarity,
\[
  Y_i=\omega_iX_{0,i},
  \qquad
  V_i=V_{0,i},
  \qquad
  1\le i\le d,
\]
are independent standard Gaussian variables.  Put
\[
  \theta_i=\omega_it,
  \qquad
  c_i=\cos\theta_i,
  \qquad
  s_i=\sin\theta_i.
\]
Exact flow gives
\[
  X_{t,i}-X_{0,i}
  =
  \frac1{\omega_i}
  \{(c_i-1)Y_i+s_iV_i\},
\]
and
\[
  V_{t,i}
  =
  -s_iY_i+c_iV_i.
\]
Therefore
\begin{align}
  D_t^--u(t)
  &=
  \sum_{i=1}^d
  \frac1{\omega_i}
  \left\{
    (c_i-1)Y_iV_i
    +
    s_i(V_i^2-1)
  \right\},
  \label{eq:gaussian-Dminus-quadratic}\\
  D_t^+-u(t)
  &=
  \sum_{i=1}^d
  \frac1{\omega_i}
  \left\{
    s_i(1-c_i)(Y_i^2-1)
    +
    (\cos(2\theta_i)-c_i)Y_iV_i
    +
    s_ic_i(V_i^2-1)
  \right\}.
  \label{eq:gaussian-Dplus-quadratic}
\end{align}
Each expression is a centered quadratic form in the \(2d\)-dimensional
standard Gaussian vector \((Y,V)\).  The corresponding symmetric
matrices are block diagonal, and every \(2\times2\) block has operator
norm at most \(C/\omega_i\).  Consequently,
\[
  \norm{A_{t,\sigma}}_{\mathrm{op}}
  \le
  \frac C{\sqrt m},
  \qquad
  \norm{A_{t,\sigma}}_{\mathrm F}
  \le
  C\sqrt{\frac dm},
  \qquad
  \sigma\in\{-,+\}.
\]
The Gaussian Hanson--Wright inequality yields
\[
  \Pp\!\left[
    |D_t^\sigma-u(t)|
    >
    \frac C{\sqrt m}
    (\sqrt{dr}+r)
  \right]
  \le
  2e^{-r}.
\]
An interval beginning at any deterministic time along the stationary
trajectory has the same phase-space law at its left endpoint.  The
same estimate therefore applies to every member of \(\mathscr T\).
A union bound over the two endpoint diagnostics and the \(N\)
intervals proves
\eqref{eq:gaussian-margin-concentration}.  Taking $  r=\log(4N/\delta)$ proves
\eqref{eq:gaussian-margin-concentration-delta}.  Finally,
\eqref{eq:gaussian-margin-dimension} makes the right-hand side of
\eqref{eq:gaussian-margin-concentration-delta} at most
\(\rho_0d/(2\sqrt m)\), after increasing the universal constant.  Every
sign then agrees with the population profile.
\end{proof}

\subsection{Gaussian leapfrog stability}
\label{app:gaussian-leapfrog-proof}

\begin{proof}[Proof of \Cref{thm:gaussian-leapfrog-stability}]
The exact-flow assertion
\eqref{eq:gaussian-Cp-direct} follows from
\eqref{eq:gaussian-Dminus-quadratic}--%
\eqref{eq:gaussian-Dplus-quadratic} and the Gaussian quadratic-form
moment inequality
\begin{align}
  \norm{
    Z^\top AZ-\Tr A
  }_{L^p}
  \le
  C\left(
    \sqrt p\,\norm A_{\mathrm F}
    +
    p\norm A_{\mathrm{op}}
  \right),
  \qquad
  Z\sim N(0,I).
  \label{eq:gaussian-quadratic-moment}
\end{align}

Put
\begin{align*}
  \alpha_i
  =
  \left(
    1-\frac{h^2\omega_i^2}{4}
  \right)^{1/2},
  \qquad
  \vartheta_i
  =
  2\arcsin\!\left(
    \frac{h\omega_i}{2}
  \right).
\end{align*}
Since \(h\sqrt L\le1\), $ {\sqrt3}/{2}
  \le
  \alpha_i
  \le
  1.$ In the normalized coordinates \(Y_i=\omega_iX_i\), leapfrog has the
exact representation
\begin{align}
  \begin{pmatrix}
    \widehat Y_{i,n}\\
    \widehat V_{i,n}
  \end{pmatrix}
  =
  \begin{pmatrix}
    \cos(n\vartheta_i)
    &
    \alpha_i^{-1}\sin(n\vartheta_i)
    \\
    -\alpha_i\sin(n\vartheta_i)
    &
    \cos(n\vartheta_i)
  \end{pmatrix}
  \begin{pmatrix}
    Y_i\\
    V_i
  \end{pmatrix}.
  \label{eq:gaussian-leapfrog-mode}
\end{align}
Taylor's theorem gives $|\vartheta_i-h\omega_i|
  \le
  Ch^3\omega_i^3$ and $|1-\alpha_i|
  \le
  Ch^2\omega_i^2.$
For an integer \(n\), let
\[
  R_i(n)
  =
  \begin{pmatrix}
    \cos(nh\omega_i)
    &
    \sin(nh\omega_i)
    \\
    -\sin(nh\omega_i)
    &
    \cos(nh\omega_i)
  \end{pmatrix},
\]
and let \(\widehat R_i(n)\) denote the matrix in
\eqref{eq:gaussian-leapfrog-mode}.  Then
\[
  \norm{
    \widehat R_i(n)-R_i(n)
  }_{\mathrm{op}}
  \le
  C\left(
    h^2\omega_i^2
    +
    |n|h^3\omega_i^3
  \right).
  \label{eq:gaussian-mode-matrix-error}
\]
Moreover,
\[
  \norm{R_i(n)}_{\mathrm{op}}=1,
  \qquad
  \norm{\widehat R_i(n)}_{\mathrm{op}}\le C
\]
uniformly in \(n\).

Let \(I=[a:b]\cap\Z\in\mathscr I_K\).  The mode contributions to the
exact and leapfrog endpoint diagnostics are quadratic forms in
\((Y_i,V_i)\).  The product rule for matrix differences and
\eqref{eq:gaussian-mode-matrix-error} give
\[
  \norm{
    A_{I,i}^{\mathrm{LF},\sigma}
    -
    A_{I,i}^{\mathrm{ex},\sigma}
  }_{\mathrm{op}}
  \le
  \frac C{\omega_i}
  \max_{n\in\{a,b\}}
  \left(
    h^2\omega_i^2
    +
    |n|h^3\omega_i^3
  \right).
\]
Because
\[
  \max\{|a|,|b|\}h
  \le
  Kh
  \le
  \frac{S+1}{\sqrt m},
\]
the right-hand side is bounded by
\[
  \frac{
    C(1+S)Lh^2
  }{
    \sqrt m
  }.
\]
Consequently,
\begin{align*}
  \norm{
    A_I^{\mathrm{LF},\sigma}
    -
    A_I^{\mathrm{ex},\sigma}
  }_{\mathrm{op}}
  &\le
  \frac{
    C(1+S)Lh^2
  }{
    \sqrt m
  },
  \notag\\
  \norm{
    A_I^{\mathrm{LF},\sigma}
    -
    A_I^{\mathrm{ex},\sigma}
  }_{\mathrm F}
  &\le
  C(1+S)Lh^2
  \sqrt{\frac dm},
  \notag\\
  \left|
    \Tr\left(
      A_I^{\mathrm{LF},\sigma}
      -
      A_I^{\mathrm{ex},\sigma}
    \right)
  \right|
  &\le
  C(1+S)Lh^2
  \frac d{\sqrt m}.
\end{align*}
Using
\eqref{eq:gaussian-quadratic-moment} and
\[
  \norm{
    \max_{1\le j\le N}|X_j|
  }_{L^p}
  \le
  N^{1/p}
  \max_j\norm{X_j}_{L^p},
\]
condition \eqref{eq:gaussian-lf-p} yields
\begin{align*}
  \left\|
    \max_{I\in\mathscr I_K}
    \max_{\sigma\in\{-,+\}}
    |\widehat D_I^\sigma-D_I^\sigma|
  \right\|_{L^p}
  \le
  \frac{
    C(1+S)Lh^2
  }{
    \sqrt m
  }
  \bigl(
    d+\sqrt{dp}+p
  \bigr).
\end{align*}
This proves
\eqref{eq:gaussian-Ep-direct}.

The same mode formula gives an exactly preserved modified energy $ \widetilde H_{i,h}(q,p)
  =
  \frac12p^2
  +
  \frac12\omega_i^2\alpha_i^2q^2.$ Hence
\begin{align*}
  H_i(\widehat q_{i,n},\widehat p_{i,n})
  -
  H_i(q_i,p_i)
  =
  \frac{h^2\omega_i^2}{8}
  \left(
    \widehat Y_{i,n}^2-Y_i^2
  \right).
\end{align*}
Writing
\(c_{i,n}=\cos(n\vartheta_i)\) and \(s_{i,n}=\sin(n\vartheta_i)\),
we have
\[
  \widehat Y_{i,n}^2-Y_i^2
  =
  \begin{pmatrix}Y_i&V_i\end{pmatrix}
  \begin{pmatrix}
    c_{i,n}^2-1 & c_{i,n}s_{i,n}/\alpha_i\\
    c_{i,n}s_{i,n}/\alpha_i & s_{i,n}^2/\alpha_i^2
  \end{pmatrix}
  \begin{pmatrix}Y_i\\V_i\end{pmatrix}.
\]
Thus the energy-error block is
\[
  A_{i,n}^{H}
  =\frac{h^2\omega_i^2}{8}
  \begin{pmatrix}
    c_{i,n}^2-1 & c_{i,n}s_{i,n}/\alpha_i\\
    c_{i,n}s_{i,n}/\alpha_i & s_{i,n}^2/\alpha_i^2
  \end{pmatrix}.
\]
Since \(\alpha_i\ge\sqrt3/2\),
\(\|A_{i,n}^{H}\|_{\mathrm{op}}\le Ch^2\omega_i^2\), and
\[
  |\operatorname{tr}A_{i,n}^{H}|
  =\frac{h^2\omega_i^2}{8}s_{i,n}^2(\alpha_i^{-2}-1)
  \le Ch^4\omega_i^4.
\]
Consequently, after summing the block-diagonal matrices over modes, the operator norm is at most \(CLh^2\),
the Frobenius norm is at most \(CLh^2\sqrt d\), and the trace is at
most \(CL^2h^4d\) in absolute value.  Therefore
\begin{align*}
  \left\|
    \max_{|n|\le K}
    |H(\Phi_h^nZ_0)-H(Z_0)|
  \right\|_{L^p}
  \le
  C\left\{
    Lh^2(\sqrt{dp}+p)
    +
    L^2h^4d
  \right\}.
\end{align*}
Markov's inequality proves
\eqref{eq:gaussian-energy-tail}.

It remains to control the shortest dyadic intervals.  Fix
\(I=[a:a+r]\cap\Z\), set \(t=rh\), and abbreviate
\[
  \phi_i=r\vartheta_i,
  \qquad
  c_i=\cos\phi_i,
  \qquad
  s_i=\sin\phi_i.
\]
At the left endpoint,
\[
  \begin{pmatrix}
    Y_{i,a}\\
    V_{i,a}
  \end{pmatrix}
  =
  \widehat R_i(a)
  \begin{pmatrix}
    Y_i\\
    V_i
  \end{pmatrix}.
\]
Writing
\[
  A_{i,a}=\E Y_{i,a}^2,
  \qquad
  B_{i,a}=\E V_{i,a}^2,
  \qquad
  C_{i,a}=\E(Y_{i,a}V_{i,a}),
\]
direct multiplication gives
\begin{align*}
  A_{i,a}
  &=
  1+
  \sin^2(a\vartheta_i)
  (\alpha_i^{-2}-1),
  \notag\\
  B_{i,a}
  &=
  1-
  \sin^2(a\vartheta_i)
  (1-\alpha_i^2),
  \notag\\
  C_{i,a}
  &=
  \sin(a\vartheta_i)
  \cos(a\vartheta_i)
  (\alpha_i^{-1}-\alpha_i).
\end{align*}
Thus $ |A_{i,a}-1|
  +
  |B_{i,a}-1|
  +
  |C_{i,a}|
  \le
  Ch^2\omega_i^2$, uniformly in \(a\).

The two numerical endpoint diagnostics over the interval are
\[
  \widehat D_{I,i}^-
  =
  \frac1{\omega_i}
  V_{i,a}
  \left\{
    (c_i-1)Y_{i,a}
    +
    \frac{s_i}{\alpha_i}V_{i,a}
  \right\},
\]
and
\[
  \widehat D_{I,i}^+
  =
  \frac1{\omega_i}
  \left(
    -\alpha_is_iY_{i,a}
    +
    c_iV_{i,a}
  \right)
  \left\{
    (c_i-1)Y_{i,a}
    +
    \frac{s_i}{\alpha_i}V_{i,a}
  \right\}.
\]
Assume \(r\ge1\); the case \(r=0\) is a singleton and is
nonturning by definition. Choose \(c_{\mathrm{sm}}>0\) so that
\begin{align*}
  0\le\phi_i\le\frac18 \quad\text{whenever}\quad
  t\le\frac{c_{\mathrm{sm}}}{\sqrt L},
\end{align*}
and reduce the universal upper bound on \(Lh^2\) if necessary. Directly from the displayed covariance formulas,
\[
  B_{i,a}\ge\alpha_i^2,
  \qquad
  |C_{i,a}|\le\frac12(\alpha_i^{-1}-\alpha_i)
  =\frac{h^2\omega_i^2}{8\alpha_i}.
\]
Moreover, \(s_i\ge\phi_i/2\), \(1-c_i\le\phi_i^2/2\), and
\(\alpha_i\ge\sqrt3/2\).  Therefore
\begin{align*}
  \E\widehat D_{I,i}^-
  &=\frac1{\omega_i}
    \left\{(c_i-1)C_{i,a}+\frac{s_i}{\alpha_i}B_{i,a}\right\}\\
  &\ge\frac1{\omega_i}
    \left\{s_i\alpha_i-(1-c_i)|C_{i,a}|\right\}
  \ge c\frac{\phi_i}{\omega_i}.
\end{align*}
Since \(\vartheta_i=2\arcsin(h\omega_i/2)\ge h\omega_i\),
\(\phi_i/\omega_i=r\vartheta_i/\omega_i\ge rh=t\), and hence
\(\E\widehat D_{I,i}^-\ge c_1t\).
For the right diagnostic, reverse the interval using
\(\Phi_h^{-1}=S\Phi_hS\).  The original right-endpoint statistic equals
the left-endpoint statistic of the reversed interval initialized at
\((\widehat Y_{i,a+r},-\widehat V_{i,a+r})\).  The covariance identities above hold
uniformly in the shift index, and changing the sign of the momentum only changes the sign
of the cross covariance, whose absolute value is what enters the bound.  The same calculation
therefore gives \(\E\widehat D_{I,i}^+\ge c_1t\). The centered diagnostics are Gaussian quadratic forms whose
\(2\times2\) coefficient blocks have operator norm at most \(Ct\).
Summing over modes gives operator norm at most \(Ct\) and Frobenius
norm at most \(Ct\sqrt d\).  Applying
\eqref{eq:gaussian-quadratic-moment} at moment order \(q=cd\) yields
\[
  \Pp\!\left[
    \min\{
      \widehat D_I^-,
      \widehat D_I^+
    \}
    \le0
  \right]
  \le
  4e^{-c_2d}.
\]
A union bound over \(\mathscr I_K\) proves
\eqref{eq:gaussian-short-tail}.
\end{proof}

\begin{proof}[Proof of \Cref{cor:gaussian-practical}]
By
\Cref{thm:gaussian-leapfrog-stability},
\[
  \Xi_{p,h}(S,K_\star)
  \le
  C\left(
    \sqrt{\frac pd}
    +
    \frac pd
  \right)
  +
  C(1+S)Lh^2
  \left(
    1+\sqrt{\frac pd}+\frac pd
  \right).
\]
The two conditions in
\eqref{eq:gaussian-practical-sign} make this at most
\(\rho_0/4\).  Equations
\eqref{eq:gaussian-energy-tail} and
\eqref{eq:gaussian-short-tail} give $  \zeta_{p,h}(K_\star)
  \le
  CN_{K_\star}e^{-p}.$ The choice
\eqref{eq:gaussian-practical-p} therefore implies the
selector-specific failure condition in
\Cref{thm:adaptive} for both selectors.  The main theorem proves the
stopping and mixing claims.

For the step size
\eqref{eq:gaussian-practical-h}, $  Lh^2(\sqrt{dp}+p)=c_h,$ and
\[
  L^2h^4d
  =
  \frac{
    c_h^2d
  }{
    (\sqrt{dp}+p)^2
  }
  \le
  \frac{c_h^2}{p}.
\]
Thus
\eqref{eq:gaussian-energy-condition} holds when \(c_h\) is sufficiently
small.  Finally,
\[
  K_\star
  =
  1+
  \frac{a_\star}{h\sqrt m}
  =
  1+
  a_\star\sqrt\kappa
  \frac1{h\sqrt L},
\]
which proves
\eqref{eq:gaussian-practical-cost}.
\end{proof}

\subsection{The two-scale phase diagram}
\label{app:two-scale-proof}

\begin{proof}[Proof of \Cref{thm:two-scale-reconciliation}]
The identity
\eqref{eq:two-scale-profile-identity} follows from
\eqref{eq:gaussian-profile}.  The slow and fast blocks contribute,
respectively,
\begin{align*}
  d_1m_1^{-1/2}
  \sin(\sqrt{m_1}t)
\qquad \text{
and}\qquad
d_2m_2^{-1/2}
  \sin(\sqrt{m_2}t).
\end{align*}
Substituting
\[
  s=\sqrt{m_2}t,
  \qquad
  \kappa=m_2/m_1,
  \qquad
  q=d_2/d_1
\]
gives
\eqref{eq:g-kappa-q}.  The equivalence between profile separation and
the margin
\eqref{eq:two-scale-margin} is therefore exact.

Statement (i) follows immediately from the definitions.  Suppose that
\((\kappa,q)\notin\mathcal A\).  Since \(g_{\kappa,q}\) is continuous
and positive for all sufficiently small positive \(s\),
\(\mathcal J_{\kappa,q}\) is nonempty and open, with positive distance
from zero.  A dyadic fast time equals $ s_k(\bar h)
  =
  \sqrt{m_2}h(2^k-1).$ Hence the grid encounters a negative value during the first fast period
if and only if
\[
  h
  \in
  \bigcup_{k\ge1}
  \frac{
    \mathcal J_{\kappa,q}
  }{
    \sqrt{m_2}(2^k-1)
  },
\]
which is
\eqref{eq:trap-step-set}.

Fix a compact step-size interval contained in the interior of
\(\mathcal H_{\mathrm{trap}}\), and remove the finitely many points at
which a relevant dyadic time is a zero of \(g_{\kappa,q}\).  On each
remaining compact component, the first negative dyadic value is
bounded above by a negative constant and every preceding dyadic value
above the short-time cutoff is bounded below by a positive constant.
Compactness gives a uniform positive value of
\(\mathfrak r_{\kappa,q}\).  The selected physical time belongs to $  m_2^{-1/2}(0,2\pi),$ so $  a_\star
  =
  \sqrt{m_1}T_\star
  =
  \Theta(\kappa^{-1/2}).$ If one chooses a later critical candidate, the earlier negative fast
dyadic time enters the preterminal minimum in
\eqref{eq:two-scale-margin}; hence profile separation at the critical
depth fails.  At a boundary point of
\(\mathcal H_{\mathrm{trap}}\), a dyadic value lies at a profile zero
and the margin vanishes.  This proves (ii).

Assume now that \((\kappa,q)\in\mathcal A\).  We show that
\(g_{\kappa,q}\) cannot be negative before
\((\pi/2)\sqrt\kappa\).  If \(1\le\kappa\le16\), membership in
\(\mathcal A\) gives nonnegativity on \((0,2\pi)\), and $  2\pi
  \ge
  \frac\pi2\sqrt\kappa.$ Suppose \(\kappa\ge16\).  Let $  0<s\le\pi\sqrt\kappa-2\pi$ and write
\[
  s=2\pi n+y,
  \qquad
  n\in\Nzero,
  \qquad
  y\in[0,2\pi).
\]
The slow phase satisfies
\[
  \frac y{\sqrt\kappa}
  \le
  \frac s{\sqrt\kappa}
  \le
  \pi-\frac y{\sqrt\kappa},
  \qquad
  0\le
  \frac y{\sqrt\kappa}
  \le
  \frac\pi2.
\]
Therefore $  \sin(s/\sqrt\kappa)
  \ge
  \sin(y/\sqrt\kappa).$ If \(\sin y\ge0\), both terms in \(g_{\kappa,q}(s)\) are nonnegative.
If \(\sin y<0\), membership in \(\mathcal A\) gives
\[
  \sin(y/\sqrt\kappa)
  +
  \frac q{\sqrt\kappa}\sin y
  \ge
  0,
\]
and the preceding comparison again yields
\(g_{\kappa,q}(s)\ge0\).  Hence there is no negative value before
\(\pi\sqrt\kappa-2\pi\), which is at least
\((\pi/2)\sqrt\kappa\) for \(\kappa\ge16\).

By \Cref{thm:gaussian-arbitrary-kappa}, the profile is negative before
\(2\pi/\sqrt{m_1}\), equivalently before fast time
\(2\pi\sqrt\kappa\).  This proves
\eqref{eq:phase-A-critical-window}.

It remains to construct a positive-measure set of separated step
sizes.  Let \((a,b)\) be the first connected component of $  \{s>0:g_{\kappa,q}(s)<0\},$ and choose a compact interval $I
  \subset
  (a,\min\{b,2a\}).$ For a fixed depth \(k\ge2\), set
\[
  c_{r,k}
  =
  \frac{2^r-1}{2^k-1}.
\]
Then \(c_{r,k}<1/2\), so \(c_{r,k}s<a\) for all \(s\in I\) and
\(r<k\).  Since \(g_{\kappa,q}\) is real analytic and not identically
zero, only finitely many points of \(I\) satisfy $  g_{\kappa,q}(c_{r,k}s)=0$ at a preterminal time above the short-time cutoff.  Any compact
interval \(I_k\) in the remaining open set has a uniform positive
terminal and preterminal margin.  The step-size interval
\[
  \left\{
    \frac{
      s
    }{
      \sqrt{m_2}(2^k-1)
    }:
    s\in I_k
  \right\}
\]
has positive Lebesgue measure and yields profile separation.  Taking
\(k\) large makes these step sizes arbitrarily small.  In phase
\(\mathcal A\), the first negative component lies at critical scale;
outside \(\mathcal A\), the same construction applied to the first
component of \(\mathcal J_{\kappa,q}\) gives fast scale.
\end{proof}

\subsection{Gaussian profile kernels and selected-time interpolation}
\label{app:gaussian-interpolation-proof}

Let \(H_n\) be the probabilists' Hermite polynomial of degree \(n\),
normalized in \(L^2(N(0,1))\).  For a multi-index
\(\alpha\in\Nzero^d\), write
\[
  H_\alpha(y)
  =
  \prod_{i=1}^dH_{\alpha_i}(y_i).
\]
After whitening, $  Y_i=\omega_iX_i,$ these functions form an orthonormal basis of \(L^2(\pi)\).

For a signed probability law \(\lambda\) on integration indices and a
deterministic step grid \(h\Z\), exact Hamiltonian flow gives $  Y_i'
  =
  \cos(\omega_ihJ)Y_i
  +
  \sin(\omega_ihJ)G_i$,
$  G_i\sim N(0,1),$ where \(J\sim\lambda\) and the \(G_i\) are independent.  Mehler's
identity yields
\begin{align}
  P_\lambda H_\alpha
  =
  \left[
    \E_J
    \prod_{i=1}^d
    \cos(\omega_ihJ)^{\alpha_i}
  \right]
  H_\alpha.
  \label{eq:Hermite-eigenvalue}
\end{align}

Let \(U_{\mathrm M}\) have density
\begin{align}
  w_{\mathrm M}(u)
  =
  1-|u|,
  \qquad
  |u|\le1,
  \label{eq:triangular-density}
\end{align}
and let \(U_{\mathrm B}\) have density
\begin{align}
  w_{\mathrm B}(u)
  =
  2\min\{|u|,1-|u|\},
  \qquad
  |u|\le1.
  \label{eq:double-tent-density}
\end{align}

\begin{lemma}[Selector characteristic functions]
\label{lem:selector-characteristic}
With
\[
  \sinc z=\frac{\sin z}{z},
  \qquad
  \sinc(0)=1,
\]
we have
\begin{align}
  \phi_{\mathrm M}(\xi)
  &:=
  \E e^{\mathrm i\xi U_{\mathrm M}}
  =
  \sinc^2(\xi/2),
  \label{eq:phi-M}\\
  \phi_{\mathrm B}(\xi)
  &:=
  \E e^{\mathrm i\xi U_{\mathrm B}}
  =
  \frac{
    4\{2\cos(\xi/2)-\cos\xi-1\}
  }{
    \xi^2
  }.
  \label{eq:phi-B}
\end{align}
If $q_a(z)
  =
  \E\cos^2(zU_a),$ then
\begin{align}
  q_{\mathrm M}(z)
  &=
  \frac12\{1+\sinc^2z\},
  \label{eq:q-M}\\
  q_{\mathrm B}(z)
  &=
  \frac12
  \left\{
    1+
    \frac{
      2\cos z(1-\cos z)
    }{
      z^2
    }
  \right\}.
  \label{eq:q-B}
\end{align}
There are universal constants \(c_{\mathrm M},c_{\mathrm B}>0\) such
that
\begin{align}
  1-\sqrt{q_a(z)}
  \ge
  c_a\min\{z^2,1\},
  \qquad
  z\ge0,
  \quad
  a\in\{\mathrm M,\mathrm B\}.
  \label{eq:q-gap}
\end{align}
\end{lemma}

\begin{proof}
The triangular variable is the difference of two independent
\(\operatorname{Unif}[0,1]\) variables, so its characteristic function
is the squared modulus of
\[
  \frac{e^{\mathrm i\xi}-1}{\mathrm i\xi},
\]
which proves
\eqref{eq:phi-M}.  Direct integration of
\eqref{eq:double-tent-density} gives
\[
  4\left[
    \int_0^{1/2}
    u\cos(\xi u)\dd u
    +
    \int_{1/2}^1
    (1-u)\cos(\xi u)\dd u
  \right],
\]
which simplifies to
\eqref{eq:phi-B}.  Equations
\eqref{eq:q-M} and
\eqref{eq:q-B} follow from $  \cos^2x=\frac12(1+\cos2x).$ For \(0\le z\le1\), concavity of sine on \([0,1]\) gives $\sin x\ge x\sin1.$ Therefore
\[
  1-\sqrt{q_a(z)}
  \ge
  \frac{1-q_a(z)}2
  =
  \frac12\E\sin^2(zU_a)
  \ge
  \frac{\sin^2(1)}2
  z^2\E U_a^2.
\]
A direct calculation gives
\[
  \E U_{\mathrm M}^2
  =
  \frac16,
  \qquad
  \E U_{\mathrm B}^2
  =
  \frac7{24}.
\]
For \(z\ge1\), $|\sinc z| \le  \sin1 < 1,$ so \(q_{\mathrm M}(z)\) is uniformly bounded away from one.  For BPS, $2\cos z(1-\cos z)
  \le 1/2,$ and hence \(q_{\mathrm B}(z)\le3/4\).  Combining the small- and
large-\(z\) bounds proves
\eqref{eq:q-gap}.
\end{proof}

\begin{lemma}[Discrete selector laws approximate their continuous limits]
\label{lem:selector-Wasserstein}
Let $U_{a,K}=J/K$ where \(J\) has law \(\lambda_{a,K}\) from
\eqref{eq:terminal-mul-law} or
\eqref{eq:terminal-bps-law}.  There is a coupling with \(U_a\) such
that
\begin{align}
  \E|U_{a,K}-U_a|
  \le
  \frac2K.
  \label{eq:selector-W1}
\end{align}
Consequently, for every \(z\ge0\),
\begin{align}
  \E|\cos(zU_{a,K})|
  \le
  \sqrt{q_a(z)}
  +
  \frac{2z}{K}.
  \label{eq:discrete-cos-bound}
\end{align}
\end{lemma}

\begin{proof}
For multinomial selection, write $J=C-A$, where \(A,C\) are independent and uniform on
\(\{0,\ldots,K-1\}\).  Couple \(A/K\) and \(C/K\) with independent
uniform variables on \([0,1]\) by adding independent
\(\operatorname{Unif}[0,1]/K\) jitters.  The total error in the
difference is at most \(2/K\), and the continuous difference has
density
\eqref{eq:triangular-density}.

For BPS, set \(K=2n\).  Conditional on a right final extension, $  J=n-A+C,$ where \(A,C\) are independent and uniform on
\(\{0,\ldots,n-1\}\); a left extension gives the reflected law.
Coupling \(A/K\) and \(C/K\) with half of independent
\(\operatorname{Unif}[0,1]\) variables again incurs total error at most
\(2/K\), and the continuous reflected law has density
\eqref{eq:double-tent-density}.  This proves
\eqref{eq:selector-W1}.

The function $  u\longmapsto|\cos(zu)|$ is \(z\)-Lipschitz.  Hence
\[
  \E|\cos(zU_{a,K})|
  \le
  \E|\cos(zU_a)|
  +
  \frac{2z}{K}
  \le
  \sqrt{
    \E\cos^2(zU_a)
  }
  +
  \frac{2z}{K},
\]
which proves
\eqref{eq:discrete-cos-bound}.
\end{proof}

\begin{proof}[Proof of \Cref{thm:gaussian-time-interpolation}]
Let \(\alpha\ne0\) and choose an index \(i\) with \(\alpha_i\ge1\).
From \eqref{eq:Hermite-eigenvalue},
\[
  \left|
    \E_J
    \prod_{r=1}^d
    \cos(\omega_rhJ)^{\alpha_r}
  \right|
  \le
  \E_J|\cos(\omega_ihJ)|
  =
  \E|\cos(z_iU_{\sigma,K})|,
  \qquad
  z_i=\omega_i\widetilde T.
\]
By \Cref{lem:selector-characteristic,lem:selector-Wasserstein},
\[
  1-|\lambda_\alpha|
  \ge
  c_\sigma\min\{z_i^2,1\}
  -
  \frac{2z_i}{K}.
\]
Since \(z_i\ge\sqrt m\,\widetilde T\) and
\(z_i\le\sqrt L\,\widetilde T\),
\[
  1-|\lambda_\alpha|
  \ge
  c_\sigma\min\{m\widetilde T^2,1\}
  -
  C_\sigma\frac{\sqrt L\,\widetilde T}{K}.
\]
The Hermite basis diagonalizes the kernel, so taking the infimum over
\(\alpha\ne0\) proves \eqref{eq:gaussian-discrete-gap}.

Under \eqref{eq:gaussian-K-gap-condition}, the absolute spectral gap is at
least $  \frac{c_\sigma}{2}\min\{\widetilde a^2,1\}.$ If $  r_0=\frac{\dd\nu}{\dd\pi},$ and $  f_0=r_0-1,$ warmness gives \(\norm{f_0}_2^2\le M\).  Absolute spectral contraction yields
\[
  \norm{P^nf_0}_2
  \le
  \left(
    1-
    \frac{c_\sigma}{2}
    \min\{\widetilde a^2,1\}
  \right)^n
  \sqrt M.
\]
Using total variation at most one half of the \(L^2\) norm and solving
for \(n\) proves \eqref{eq:gaussian-interpolation-transitions}.
Multiplication by \(K=\widetilde T/h\) gives
\eqref{eq:gaussian-interpolation-work}.  Finally,
\[
  \widetilde T=\frac{\widetilde a}{\sqrt m},
  \qquad
  h=\widetilde\Theta(L^{-1/2}d^{-1/4}),
\]
so
\[
  \frac{\widetilde T}{h}
  =
  \widetilde\Theta(
    \widetilde a\sqrt\kappa\,d^{1/4}
  ),
\]
which proves \eqref{eq:gaussian-smooth-interpolation}.
\end{proof}

\end{document}